\documentclass[reqno,11pt]{amsart}
\usepackage{amsmath,amssymb,mathrsfs,color}
\usepackage{graphicx}
\usepackage{epstopdf}
\usepackage{soul}
\usepackage{cases}
\usepackage{hyperref}
\hypersetup{hypertex=true,
	colorlinks=true,
	linkcolor=blue,
	anchorcolor=blue,
	citecolor=blue}
\usepackage[margin=1in]{geometry}

\def\b1{{\mathbbm 1}}

\allowdisplaybreaks

\usepackage{
	indentfirst, latexsym, bm, enumerate,
	wrapfig, adjustbox, subfigure, appendix,
	float, bbding, xcolor, multirow, framed, lipsum, epsdice,
	url, amsthm, booktabs, makecell, caption, algpseudocode, pifont,
	balance, dcolumn, fontenc, extarrows, arydshln 
}
\allowdisplaybreaks[4]
\usepackage{tikz}
\usetikzlibrary{arrows}

\newcommand*{\dif}{\ \mathrm{d}}

\newcommand{\ebra}[1]{(#1)}
\newcommand{\eangle}[1]{\langle#1\rangle}

\newcommand{\norm}[1]{\left\Vert #1\right\Vert}

\newtheorem{lemma}{Lemma}[section]
\newtheorem{theorem}[lemma]{Theorem}
\theoremstyle{definition}  
\newtheorem{remark}[lemma]{Remark}

\newtheorem{corollary}[lemma]{Corollary}

\newtheorem{example}[lemma]{Example}
\allowdisplaybreaks
\numberwithin{equation}{section}

\newtheorem{prfbase}{Proof}[section]

\renewcommand{\hl}[1]{#1}

\begin{document}
	\author{Junlang Hu}
	\address{J. Hu: School of Mathematical Sciences,
		Dalian University of Technology, Dalian 116024, P. R. China}
	\email{junlanghu@hotmail.com}
	
	\author{Zhenxin Liu}
	\address{Z. Liu (Corresponding author): School of Mathematical Sciences,
		Dalian University of Technology, Dalian 116024, P. R. China}
	\email{zxliu@dlut.edu.cn}
	
	\title[McKean-Vlasov Phase Transitions via bifurcation theory]{Continuity and Discontinuity of McKean-Vlasov Phase Transitions via Bifurcation Theory}
	
	\date{\today}
	
	\begin{abstract}
		It is well known that the McKean-Vlasov stochastic differential equation with a symmetric double-well potential exhibits a continuous phase transition. In contrast, for an asymmetric double-well potential, the system undergoes a discontinuous phase transition, in which an invariant measure abruptly appears in the shallower well and subsequently splits into two as the temperature decreases.  In this work, we systematically investigate the types of phase transitions driven by non-convex confining potentials and cooperative interactions in $\mathbb{R}^n$. Our main results consist of a pitchfork bifurcation theorem characterizing continuous phase transitions and a saddle-node bifurcation theorem governing discontinuous phase transitions. In addition, despite the dimension-dependent nature of phase transitions, we are able to extend Dawson's criterion for phase transition points -- originally formulated for quadratic interactions in $1$-dimensional space -- to $n$-dimensional space. This extension directly relates the critical parameter and bifurcation direction to the eigenvalue and eigenvector of the critical covariance matrix, respectively. Finally, we apply our theoretical results to the aforementioned double-well models, to a class of four-well models in two-dimensional Euclidean space that exhibit multiple phase transitions, and to an attractive Gaussian interaction model. 
	\end{abstract}
	
	\maketitle
	\tableofcontents  
	
	\section{Introduction}
	\subsection{Model formulation}
	A phase transition is typically understood as a qualitative change in the limiting state of a system as its underlying parameters vary. Representative examples include variations in the number of invariant measures or transitions from invariant to periodic measures. This is precisely the central concern of bifurcation theory. We employ an abstract bifurcation framework to investigate the specific types of phase transitions occurring in gradient-type McKean-Vlasov stochastic differential equations (MVSDEs, also referred to as mean-field SDEs) in $\mathbb{R}^n$
	\begin{equation}\label{eq.MVSDE}
		\begin{aligned}
			\text{d}X_t = -\nabla V\ebra{X_t} \text{d}t - \kappa\cdot \ebra{\nabla W* \mathcal{L}_{X_t}}\ebra{X_t} + \sigma \text{d} B_t,
		\end{aligned}
	\end{equation}
	\hl{where $V$ and $W$ represent the external potential and the interaction potential, respectively. The positive parameters $\sigma$ and $\kappa$ denote the strengths of the white noise and the interaction. The symbol $*$ denotes convolution, and $\mathcal{L}_{X_t}$ represents the law of $X_t$. }Equation \eqref{eq.MVSDE} originates from the mean-field limit of an $N$-particle system, where each particle interacts with all others through the kernel $\nabla W$ with equal weight of $1/N$, and $B^i_t$ ($i=1, \dots, N$) are $N$ independent $\mathbb{R}^n$-valued Brownian motions
	\begin{equation}\label{eq.N-par}
		\text{d}X_t^i = -\nabla V\ebra{X_t^i} \text{d}t - \kappa\cdot \frac{1}{N} \sum_{j \neq i} \nabla W\ebra{X_t^i - X_t^j} \text{d}t + \sigma \text{d}B_t^i, \qquad i \in \{1, \dots, N\}.
	\end{equation}
	The propagation of chaos argument \cite{McK67, Szn91} states that, provided the particles are exchangeable, correlations among any fixed finite subset of particles vanish as $N \to +\infty$. Consequently, the particles become asymptotically independent, making it sufficient to investigate the limiting dynamical behavior of a single particle via the mean-field equation \eqref{eq.MVSDE}. We emphasize that when phase transitions occur, this approximation is not uniform-in-time; that is, the limits $N \to \infty$ and $t \to \infty$ cannot be interchanged \cite{DGP21}. As a result, the long-time behaviors of \eqref{eq.MVSDE} and \eqref{eq.N-par} can be drastically distinct. Under a standard dissipativity assumption, the $N$-particle system possesses a unique invariant measure, whereas the mean-field equation may exhibit multiple invariant measures or even periodic measures. Nevertheless, the mean-field limit remains ``close to'' the finite $N$-particle system over large but finite time horizons, provided that $N$ is sufficiently large \cite{Mon25, LcP25}.
	
	To apply bifurcation methods, we henceforth study the nonlinear Fokker-Planck equation that describes the time evolution of the law of the solution to \eqref{eq.MVSDE}
	\begin{equation}\label{eq.1}
		\partial_t \mu_t = \frac{\sigma^2}{2} \Delta \mu_t + \text{div} \ebra{\nabla V \mu_t} + \text{div} \ebra{\kappa \cdot \ebra{\nabla W * \mu_t} \mu_t}.
	\end{equation}
	\hl{Invariant measures} of \eqref{eq.1} are associated with the nonlocal elliptic PDE obtained by setting $\partial_t \mu_t = 0$. From a variational perspective, \eqref{eq.1} can be identified as a gradient flow of the free energy functional $\mathcal{H}$ in the $2$-Wasserstein space
	\begin{equation*}\label{eq.9}
		\begin{aligned}
			\partial_t \mu_t & = -\nabla_\text{W} \mathcal{H}\ebra{\mu_t}, \quad \mathcal{H}\ebra{\mu} = \int_{\mathbb{R}^n} \frac{\sigma^2}{2} \cdot \mu \log \mu  + \ebra{V + \frac{\kappa}{2} \cdot W * \mu} \mu \dif x,
		\end{aligned}
	\end{equation*}
	where $\nabla_{\operatorname{W}}$ denotes the formal Wasserstein gradient, and the interaction potential $W$ is required to be symmetric (even) within this gradient flow framework. This underlying gradient flow structure was rigorously established in the seminal works \cite{JKO98, CMV03}. Within this geometric setting, invariant measures correspond precisely to the critical points of the free energy $\mathcal{H}$, and are thus the zeros of the associated energy dissipation rate functional
	\begin{equation}\label{eq.10}
		\begin{aligned}
			\partial_t \mathcal{H}\ebra{\mu_t} &= -\eangle{\nabla_W H\ebra{\mu_t}, \nabla_W H\ebra{\mu_t}}_{\mu_t}\\
			&= - \int_{\mathbb{R}^n} |\frac{\sigma^2}{2} \cdot \nabla \log\mu_t + \nabla V + \kappa \cdot \nabla W * \mu_t|^2 \mu_t \dif x=0.
		\end{aligned}
	\end{equation}
	The equation $\ebra{\ref{eq.10}}$ leads to a self-consistent implicit equation for invariant measures
	\begin{equation*}\label{eq.7}
		\begin{aligned}
			\mu = \frac{1}{Z\ebra{\mu, \kappa, \sigma}} \cdot \exp\ebra{\ebra{-V - \kappa \cdot W * \mu}\cdot \frac{2}{\sigma^2}},
		\end{aligned}
	\end{equation*}
	with $Z\ebra{\mu, \kappa, \sigma} = \int_{\mathbb{R}^n} \exp \ebra{\ebra{-V - \kappa \cdot W * \mu}\cdot 2/\sigma^2} \dif x$ being the
	normalizing constant. This formulation provides a framework for applying abstract bifurcation theory. 
	
	\subsection{Related work}
	\begin{itemize}
		\item \textbf{Early work.} The propagation of chaos, critical dynamics, fluctuations, and large deviations for the classical cooperative double-well model were extensively investigated in the foundational works of Dawson \cite{Daw83, DG89}. In particular, the weak Riemannian structure of the space of probability measures was heuristically discussed in \cite{DG89}. Crucially, the extension of the Ellis-Monroe-Newman's GHS inequality to continuous-space measures played a pivotal role in characterizing the system's global dynamical behavior \cite{Ell75, EMN76}. Early contributions to the bifurcation analysis and stability of invariant measures were also made by Tamura \cite{Tam84}. Going beyond the multiplicity of steady states, Scheutzow demonstrated that cooperative interactions can even give rise to periodic measures \cite{Sch86}.
		
		\item \textbf{McKean-Vlasov dynamics in $\mathbb{R}^n$.} \hl{Building upon these earlier foundations, the field has seen several notable advancements in recent years. For instance, recent studies address the non-uniqueness and global convergence of invariant measures via the underlying gradient flow structure }\cite{Tug13, HT10}, \hl{while a concise proof of convergence within the Wasserstein framework is provided in {\cite{Bas20}}. Furthermore, recent work establishes connections between several foundational concepts {\cite{DGPS23}}. Long-time behavior has also been analyzed by utilizing uniform-in-particle concentration inequalities {\cite{LWZ21}}. Additional progress includes establishing the local existence of invariant measures for MVSDEs with jumps in the low-temperature regime through a local dissipative condition {\cite{BW25}}. Researchers have also characterized the emergence of periodic measures within coupled neural models {\cite{MQT16}}, slow-fast system frameworks {\cite{LcP20}}, and via Poincar\'{e} maps {\cite{Sch86}}.}
		
		\item \textbf{Abstract bifurcation theory and phase transitions.} Local bifurcation theory has been widely developed to characterize systems near critical parameter values that correspond to transitions between distinct steady states. The classical local bifurcation framework in infinite-dimensional Banach spaces was established by Crandall and Rabinowitz \cite{CR71, CR73}, providing a powerful tool for describing phase transitions. Carrillo, Gvalani, Pavliotis, and Schlichting \cite{CGPS20} insightfully leveraged these techniques to study different types of McKean-Vlasov phase transitions driven by pure interactions and diffusion on the torus. Their motivating examples include the noisy Kuramoto, Keller-Segel, and opinion dynamics models, where the compactness of the torus and its inherent Fourier basis structure significantly simplify the analysis. While phase transitions in finite volumes were studied in \cite{MR2594901}, analyzing continuous and discontinuous phase transitions in the unbounded space $\mathbb{R}^n$ requires fundamentally different treatments. Recently, Zhang established a Krasnosel'skii-type local bifurcation theorem \cite{Zha25} to investigate continuous bifurcations from a family of trivial solutions, utilizing the classical one-dimensional double-well model as a prototypical example. Notably, besides static bifurcations derived from self-consistent implicit equations, Hopf bifurcations can also emerge in mean-field dynamics, as demonstrated by Cormier, Tanr\'e, and Veltz \cite{CTV21}. In all these settings, a major analytical challenge lies in verifying the abstract spectral conditions for concrete models. 
		
		\item \textbf{Recent works.} While by no means exhaustive, we highlight several recent preprints that are closely related to our work. Liu, Qu, Yao, and Zhi \cite{LQYZ26}, alongside Qu, Yao, and Zhi \cite{QYZ26}, investigated global dynamics by leveraging stochastic order relations and Zhang's local dissipative condition. Extending the framework of \cite{CGPS20}, Shalova and Schlichting \cite{SS25} analyzed phase transitions on high-dimensional spheres and other Riemannian manifolds. Concurrently, Gerber, Gvalani, Hairer, Pavliotis, and Schlichting \cite{GGH25} observed that the coalescence of clusters in finite $N$-particle dynamics cannot be fully captured by the corresponding mean-field limit. Cluster formation in kinetic Langevin interacting particle systems under a low-friction regime was further explored by Leimkuhler, Lohmann, Pavliotis, and Whalley \cite{LLPW25}. Finally, Balasubramanian, Banerjee, and Rigollet \cite{BBR25} studied phase transitions on a circle using bifurcation theory, with a particular emphasis on applications to transformer architectures. Their theory is formulated in a sequence space generated by Fourier coefficients rather than in a traditional function space.
	\end{itemize}
	
	Our results are closely related to, yet distinct from, those of Carrillo, Gvalani, Pavliotis, and Schlichting \cite{CGPS20} and Zhang \cite{Zha25}. The work in \cite{CGPS20} illuminated the deep connections between phase transitions and bifurcations. Their framework was established for MVSDEs on compact manifolds without a confining potential. Zhang \cite{Zha25} rigorously applied a local Krasnosel'skii bifurcation theorem to MVSDEs with confining potentials in $\mathbb{R}^n$. His analysis focuses primarily on continuous phase transitions under the assumption that a family of trivial solutions exists. 
	
	\hl{The present paper contributes theoretical perspectives that extend current understanding in the literature. Specifically, we provide theoretical proofs for the discontinuous phase transition phenomena previously observed in numerical studies by utilizing an abstract saddle-node bifurcation theorem. Additionally, our approach to continuous phase transitions via an abstract pitchfork bifurcation theorem and the $n$-dimensional generalization of Dawson's criterion expands the available analytical tools in this area. Moreover, a monotonicity inequality for the moment with respect to temperature is also established in Appendix {\ref{numer}}, which is of independent interest.} The multiplicity of invariant measures for general potentials in the small-noise regime has also been investigated in many existing works \cite{LWZ21, DGPS23, LQYZ26}. In contrast to these works, our results primarily focus on bifurcation, especially the (dis)continuity of phase transitions at the critical temperature.
	
	\subsection{Numerical observations}\label{obser}
	\begin{figure}[htb]
		\centering
		\includegraphics[width=0.8\linewidth]{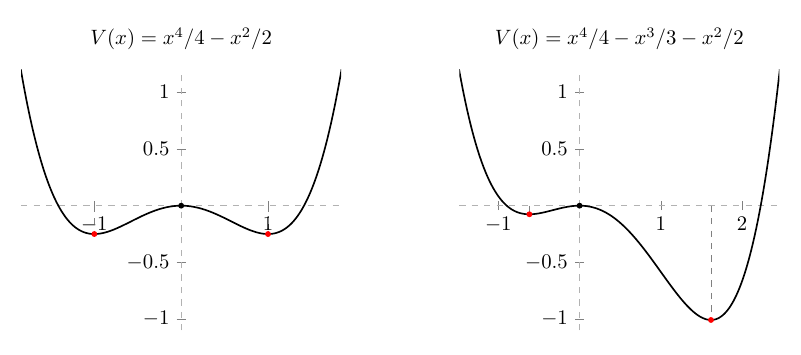}
		\caption{Sketches of symmetric and asymmetric double-well potentials. }\label{Fig.0}
	\end{figure}
	
	We illustrate these observed phenomena using the standard prototypical examples of one-dimensional symmetric and asymmetric double-well potentials: $V_{\text{symm}}(x) = x^4/4 - x^2/2$ and $V_{\text{asymm}}(x) = x^4/4 - x^3/3 - x^2/2$ (see Figure \ref{Fig.0}). For both cases, we choose a quadratic attractive interaction potential $W(x) = x^2/2$ with interaction strength $\kappa=1$, and analyze the resulting phase transitions as the effective parameter $\lambda := -2\kappa/\sigma^2$ varies.
	
	We say that $\lambda_\star$ is a \emph{critical parameter} if new critical points of the free energy functional that are distinct from the original ones emerge as $\lambda$ crosses $\lambda_\star$. The phase transition is said to be \emph{continuous} if these new critical points converge to the original one in an appropriate topology as $\lambda \to \lambda_\star$, a scenario that typically corresponds to a ferromagnetic phase transition. Otherwise, the phase transition is termed \emph{discontinuous}, which frequently characterizes a liquid-vapor phase transition. This definition follows the conventions of \cite{CGPS20, GGH25}. From the perspective of bifurcation theory, such a phase transition point corresponds precisely to a bifurcation point. A formal definition from this bifurcation viewpoint is provided in Section \ref{the_main_results}.
	
	The main numerical observations are summarized below (see also the bifurcation diagram, Figure \ref{Fig.1}, for mean values of invariant measures as \hl{$\lambda$} changes)
	\begin{enumerate}
		\item For the symmetric double-well potential, there is always a family of zero-mean trivial invariant measures. Two mutually mirrored invariant measures bifurcate continuously from the trivial solution below $\lambda_\star$.
		\item For the asymmetric double-well potential, there is always an invariant measure concentrated mainly on the deeper well. An invariant measure suddenly emerges in the shallower well at the critical parameter $\lambda_\star$, and then splits into two when $\lambda < \lambda_\star$. 
	\end{enumerate}
	Note that discontinuous phase-transition phenomena in multi-well potentials have been numerically observed in \cite{GKKP19}. As a simplification of the multi-well model, the asymmetric double-well model is adopted here, while the phase transition in the former can be addressed similarly.
	
	\begin{figure}[htb]
		\centering
		\subfigure[]{
			\includegraphics[scale=0.46]{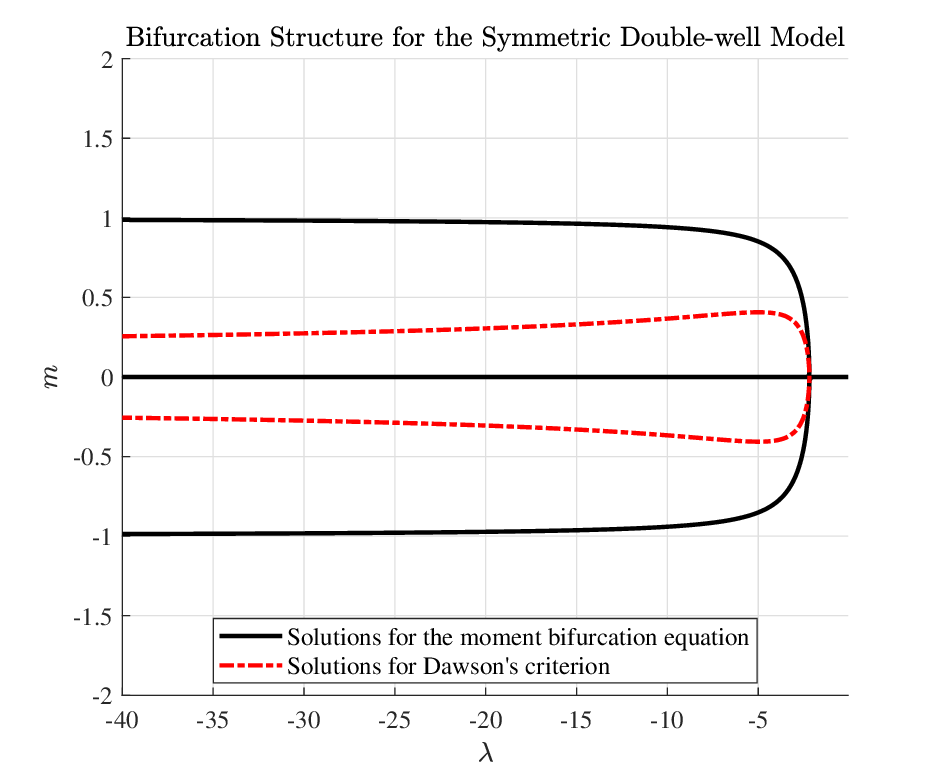}
			\label{Fig.1a}}
		\subfigure[]{
			\includegraphics[scale=0.46]{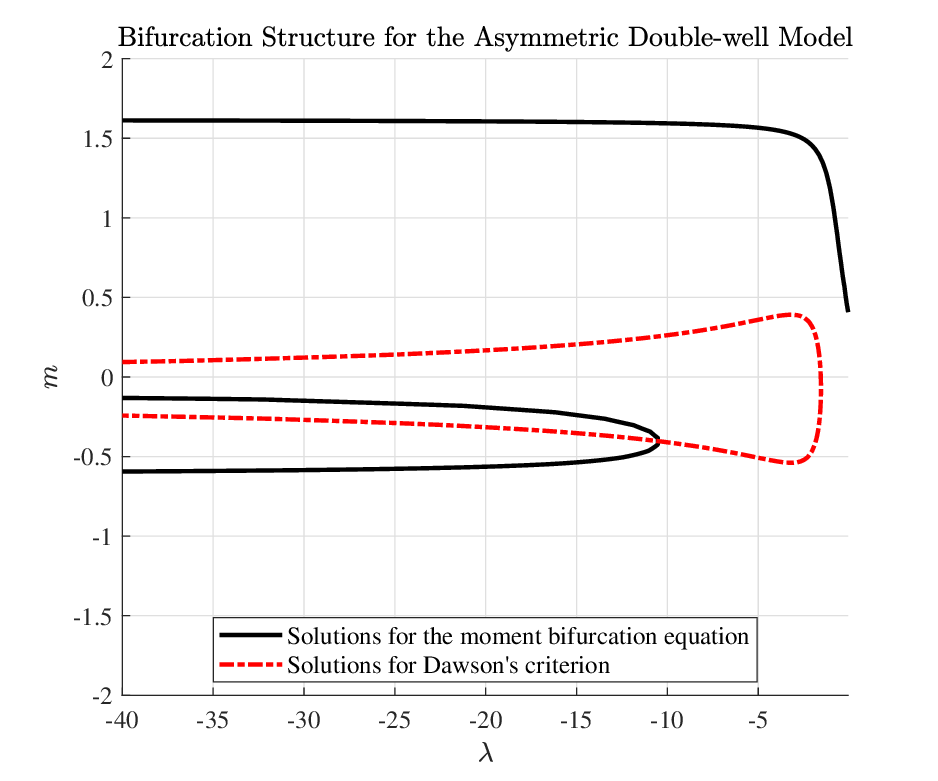}
			\label{Fig.1b}}
		\caption{Bifurcation diagrams for mean values of multiple invariant measures obtained by solving the moment bifurcation equation for symmetric and asymmetric double-well potentials: $V_{sym} = x^4/4 - x^2/2,\ V_{asym} = x^4/4 - x^3/3 - x^2/2 $, $\lambda = -2\kappa/\sigma^2$. The intersection point of solutions of the bifurcation equation (\ref{eq.16}) (black line) and Dawson's criterion (\ref{eq.102}) (red dashed line) is exactly the (continuous or discontinuous) phase transition point. }\label{Fig.1}
	\end{figure}

	\subsection{Limitations and contributions}
	Limitations. Since the Crandall-Rabinowitz bifurcation theory only treats the $1$-dimensional null space case, we do not currently handle the higher-dimensional null-space case, for example, when $V\ebra{x_1, x_2} = (x_1^2+x_2^2)^2/4 - (x_1^2 + x_2^2)/2$. Furthermore, for general interactions, the spectral condition remains abstract and need not reduce to a finite-dimensional covariance criterion; accordingly, the Gaussian interaction example is treated mainly as a numerical illustration. 
	
	Contributions. In this work, we establish abstract pitchfork and saddle-node bifurcation theorems that characterize continuous and discontinuous phase transitions, respectively (see Theorems \ref{thm.7}, \ref{thm.8}, and \ref{thm.9}). Unlike the torus case, where the homogeneous branch and Fourier modes are fixed, the whole-space problem involves $\lambda$-dependent Gibbs profiles, making the verification of spectral and bifurcation conditions substantially more delicate. Utilizing this bifurcation framework, we extend Dawson's classical criterion from the one-dimensional setting to higher dimensions. Furthermore, we investigate several representative applications, including one-dimensional symmetric and asymmetric double-well models, as well as a two-dimensional model that exhibits multiple phase transitions (see Section \ref{exs}). Moreover, a monotonicity inequality for the moment with respect to temperature is also established, which is of independent interest. 
	
	The remainder of this paper is organized as follows. Section \ref{the_main_results} delineates the abstract bifurcation theorems and details our core examples of continuous and discontinuous phase transitions. The proofs of these main theoretical results are presented in Section \ref{sec.proof_thms}. Lastly, Section \ref{ex_proof} is dedicated to the technical proofs and derivations for the concrete examples.
	
	\section{Preliminaries}
	We first state the definitions of continuous and discontinuous phase transitions. A solution $\ebra{\lambda_\star, \rho_\star}$ of the bifurcation equation is called a \textit{phase-transition point} (or bifurcation point) if solutions emerge or vanish near $\rho_\star$ as $\lambda$ crosses $\lambda_\star$. If there exists a family of trivial solutions $\{\rho_0\ebra{\lambda}\}$ passing through the phase-transition point under consideration, and non-trivial solution curves exist in arbitrarily small neighborhoods of $\ebra{\lambda_\star, \rho_\star}$, we say the system exhibits a \textit{continuous phase transition}. If no pre-existing branch of solutions passes through the phase-transition point, we say that the system undergoes \textit{a discontinuous phase transition}. See also Section \ref{obser} for an alternative definition. Conceptually similar but not identical definitions designed for compact manifolds can be found in \cite{MR2594901, CGPS20}. A \textit{pitchfork bifurcation} emerges at a continuous phase transition point when conditions (\ref{eq.33}), (\ref{eq.35.2}), and (\ref{eq.36.1}) hold. A \textit{saddle-node bifurcation} takes place at a discontinuous phase transition point whenever equations (\ref{eq.48}) and (\ref{eq.51.1}) are satisfied.

	\subsection{Notations}
	\begin{enumerate}
		\item $\mathcal{P}\ebra{\mathbb{R}^n}$: the space of probability measures on $\mathbb{R}^n$, $\norm{\mu}_p^p := \int_{\mathbb{R}^n} |x|^p \mu\ebra{\text{d}x}$.
		\item $\mathcal{P}_p\ebra{\mathbb{R}^n} := \{\mu \in \mathcal{P}\ebra{\mathbb{R}^n}: \norm{\mu}_p  < \infty\}$, $\mathcal{P}_p^{ac}\ebra{\mathbb{R}^n} := \{\mu \in \mathcal{P}_p\ebra{\mathbb{R}^n}: \mu \ll \mathcal{L}^n\}$ where $\mathcal{L}^n$ denotes the Lebesgue measure on $\mathbb R^n$.
		\item $B_R\ebra{0} := \left\{x \in \mathbb{R}^n:|x| \leq R\right\}$, $\tau_h f := f\ebra{\cdot - h}$, $\chi_A$: the indicator function of the set $A$.
		\item Moments: $m_0 \ebra{f} := \int_{\mathbb{R}^n} 1 f\ebra{\text{d}x}$, $m_1 \ebra{f} := \int_{\mathbb{R}^n} x f\ebra{\text{d}x}$, $m_2 \ebra{f} := \int_{\mathbb{R}^n} |x|^2 f\ebra{\text{d}x}$. For a multi-index $\alpha$, the corresponding moment is $m_\alpha\ebra{f} = \int_{\mathbb{R}^n} \ebra{\prod_{i = 1}^n x_{i}^{\alpha_i}}  f \ebra{\text{d} x},$ where $\alpha = \ebra{\alpha_1, \dots, \alpha_n}$ is a multi-index.
		\item $K^*$: adjoint operator of $K$; when $K$ is compact with finitely many eigenvalues, we denote its eigenvalues by $\tilde{\theta}_1,\dots,\tilde{\theta}_n$ and its characteristic values by $\lambda_i:=\tilde{\theta}_i^{-1}$ for $i=1,\dots,n$.
		\item $\delta_{ij}$: the Kronecker symbol for indices $i$ and $j$.
		\item $\eangle{f}_{\rho} := \int_{\mathbb{R}^n} f \cdot \rho \dif x$.
	\end{enumerate}
	
	\subsection{Assumptions} Consider the MVSDE \eqref{eq.MVSDE} and the corresponding nonlinear Fokker-Planck equation \eqref{eq.1}. We assume that both the external potential $V$ and the interaction potential $W$ are $C^1$ and satisfy the following conditions (H1), (H2), and (H3).
	
	\begin{itemize}
		\item[(H1)] (One-sided Lipschitz) For some $L \in \mathbb{R}$, $\mu, \tilde{\mu} \in \mathcal{P}_p\ebra{\mathbb{R}^n}$ ($p \geq 2$),
		\begin{equation*}\label{eq.4}
			\begin{aligned}
				\eangle{x-y, \ebra{-\nabla V}\ebra{x} - \ebra{-\nabla V}\ebra{y}} & \leq L \cdot |x-y|^2, \\
				\eangle{x-y, \ebra{- \nabla W * \mu}\ebra{x} - \ebra{- \nabla W * \tilde{\mu}}\ebra{y}} & \leq L \cdot \ebra{|x-y|^2 + W_2\ebra{\mu, \tilde{\mu}}^2}.
			\end{aligned}
		\end{equation*}
		\item[(H2)] (Growth controlled by polynomial) For some $C > 0$ and $p \in \mathbb{N}$,
		\begin{equation*}\label{eq.5}
			\begin{aligned}
				|\nabla V(x)| & \leq C \cdot \ebra{1 + |x|^p}, \\
				|(\nabla W * \mu)(x)| & \leq C \cdot \ebra{1 + |x|^p + \norm{\mu}_p^p}.
			\end{aligned}
		\end{equation*}
		\item[(H3)] (Weak hyper-dissipativity) For some $\alpha, \beta, \delta > 0$ and $\alpha^\prime > \gamma^\prime > 0$, $\beta^\prime > 0$,
		\begin{equation*}\label{eq.3}
			\begin{aligned}
				\eangle{x, -\nabla V (x)} & \leq -\alpha \cdot |x|^{2+\delta} + \beta, \\
				\eangle{x, -\nabla V \ebra{x} - \kappa \cdot (\nabla W * \mu)(x)} & \leq -\alpha^\prime \cdot |x|^2 + \gamma^\prime \cdot \norm{\mu}_2^2 + \beta^\prime .
			\end{aligned}
		\end{equation*}
	\end{itemize}
	
	Under assumptions (H1), (H2) and (H3), equations \eqref{eq.MVSDE} and \eqref{eq.1} are well-posed;
	see, e.g., \cite[Theorem 2.1]{HHL24}.

	\subsection{Characterization of invariant measures}
	Preliminary material about bifurcation theorems can be found in \cite{Dei85, Kie12, BT03}.
	
	Invariant measures of mean-field equations can be viewed as solutions of fixed-point equations. We define the \textit{Gibbs map }$\mathcal{T}$, which sends a given measure $\rho$ to the invariant measure of the following decoupled classical Fokker-Planck equation
	\begin{equation}\label{eq.13}
		\partial_t \bar{\mu}_t = \frac{\sigma^2}{2} \Delta \bar{\mu}_t + \text{div} \ebra{\nabla V \bar{\mu}_t} + \text{div} \ebra{\kappa \cdot \ebra{\nabla W * \rho} \bar{\mu}_t}.
	\end{equation}
	The map $\mathcal{T}$ is well-defined once the existence and uniqueness of the invariant measure of the classical Fokker-Planck equation $\ebra{\ref{eq.13}}$ hold. The invariant measures of the MVSDE are characterized by three statements (\ref{eq.15}), (\ref{eq.16}), and (\ref{eq.mom_eq}) below.
	
	\medskip
	
	\textit{1. The bifurcation equation.} Invariant measures are fixed points of $\mathcal{T}$ defined on $X$, an appropriate Banach space containing all invariant measures under investigation. If $\kappa$ is fixed and $\sigma$ is allowed to vary, so that the temperature is free to change, then the phase transition is essentially controlled by the parameter $\lambda = -2\kappa/\sigma^2$. The bifurcation equation is given by
	\begin{equation}\label{eq.15}
		\begin{aligned}
			\text{(Bifurcation Equation)} \quad \quad & F(\lambda, \rho):= (id - \mathcal{T}\ebra{\lambda, \cdot}) \ebra{\rho} = 0, \\
			& \mathcal{T}\ebra{\lambda, \cdot}: \ \rho \mapsto \frac{1}{Z\ebra{\lambda, \rho}} \cdot \exp\ebra{\lambda \cdot \ebra{V\kappa^{-1}  + W * \rho}}. \qquad\qquad
		\end{aligned}
	\end{equation}
	
	Denote by $L^1 \ebra{\mathbb{R}^n, \Lambda}$ the weighted $L^1$ space equipped with the norm
	\begin{equation*}\label{eq.14}
		\norm{\cdot}_{L^1\ebra{\mathbb{R}^n,\ \Lambda}} := \int_{\mathbb{R}^n} |\cdot|\ \Lambda \dif x.
	\end{equation*}
	When $W = |x|^2/2$, we typically choose $X = L^1\ebra{\mathbb{R}^n, 1 + |x|^2}$. Note that $\mathcal{P}_2^{ac}\ebra{\mathbb{R}^n} \hookrightarrow L^1\ebra{\mathbb{R}^n, 1 + |x|^2}$ by identifying measures with their density functions.
	
	\begin{remark}\label{rem.3}
		Searching for fixed points of $\mathcal{T}$ on the weighted space $L^1\ebra{\mathbb{R}^n, 1 + |x|^2}$ is equivalent to searching for them on the probability measure space $\mathcal{P}^{ac}_2\ebra{\mathbb{R}^n}$. On the one hand, any invariant measure on $\mathcal{P}^{ac}_2\ebra{\mathbb{R}^n}$ belongs to $L^1\ebra{\mathbb{R}^n, 1 + |x|^2}$ by the embedding. On the other hand, under assumptions (H1), (H3), images of $\mathcal{T}$ are probability measures with finite second moments (by the coercivity estimate for $V$, see Lemma \ref{lem.1}), so its fixed points must belong to $\mathcal{P}^{ac}_2\ebra{\mathbb{R}^n} \cap L^1\ebra{\mathbb{R}^n, 1 + |x|^2}$.
	\end{remark}

	\textit{2. The moment bifurcation equation when $W = |x|^2/2$.} Under a quadratic interaction, $\mathcal{T}\ebra{\lambda, \rho}$ is determined once $m_1 \ebra{\rho}$ is known. Indeed, $m_0\ebra{\rho}$ can be replaced by $1$ since we are searching for solutions among probability measures. The second-order moment term is independent of $x$, and is therefore eliminated after normalization by $\tilde{Z}$. Hence this equation depends only on $m_1\ebra{\rho}$. Therefore, \eqref{eq.15} is equivalent to $m_1\ebra{\ebra{id - \mathcal{T}\ebra{\lambda, \cdot}}\ebra{\rho}} = 0$, which is
	\begin{equation}\label{eq.16}
		\begin{aligned}
			\text{(Moment Bifurcation Equation) }  &  m_1 - \int_{\mathbb{R}^n } x \cdot \tilde{\mathcal{T}}\ebra{\lambda, m_1}\dif x = 0, \\
			& \tilde{\mathcal{T}}\ebra{\lambda, m_1}
			:= \frac{1}{\tilde{Z}\ebra{\lambda, m_1}} \cdot \exp\ebra{\lambda\cdot \ebra{V \kappa^{-1} + \frac{1}{2}|x|^2 - \eangle{m_1, x}}}.
		\end{aligned}
	\end{equation}

	\medskip
	
	\textit{3. The moment equation and the Hankel inequality when $n=1$, $V$ is a polynomial of degree greater than 2 and $W = |x|^2/2$.} For a polynomial potential $V$, applying It\^o's formula yields the moment equation with respect to time $t$
	\begin{equation*}\label{eq.16.1}
		\frac{\text{d}}{\text{d} t} m_p\ebra{\mu_t} = p\cdot \ebra{\frac{\sigma^2}{2}\cdot \ebra{p-1} m_{p-2} - \int_{\mathbb{R}} x^{p-1} V^\prime \cdot\mu_t \dif x - \kappa\cdot \ebra{m_p - m_1m_{p-1}}}\ebra{\mu_t},
	\end{equation*}
	where the $p$-th moment is defined as $m_p\ebra{\rho} := \int_{\mathbb{R}} x^p \cdot \rho\ebra{x} \dif x$ for $p \geq 2$.
	
	In this case, by \cite[Section 3.2.1]{Daw83}, the invariant measure can also be characterized by the following equivalent condition (when $n > 1$, it is necessary but not sufficient)
	\begin{equation}\label{eq.mom_eq}
		\begin{aligned}
			& \rho_0 \text{ is a probability measure generated by a sequence of moments } \{m_p\}_{p \in \mathbb{N}}, \text{where} \\
			& \{m_p\}_{p \in \mathbb{N}} \text{ is determined by the moment equations and Hankel inequalities given below. } \\
		\end{aligned}
	\end{equation}
	Here,
	\begin{equation}\label{eq.mom_eq.1}
		\begin{aligned}
			& \text{(Moment Equations) } \qquad\qquad\quad\  m_0\ebra{V^\prime \rho_0} = 0, \qquad\qquad\qquad\quad\ \ \ebra{p=1}\\
			& \sigma^2/2\cdot \ebra{p-1} m_{p-2} - m_0\ebra{x^{p-1} V^\prime \rho_0} - \kappa\cdot \ebra{m_p - m_1 m_{p-1}} = 0, \ \ \ \ \ \ebra{p \geq 2}\\
			& \text{(Hankel inequalities) } \Delta_p = \det
			\begin{pmatrix}
				m_0 & m_1 & \cdots & m_p\\
				m_1 & m_2 & \cdots & m_{p+1}\\
				\vdots & \vdots & \vdots & \vdots\\
				m_p & m_{p+1} & \cdots & m_{2p}
			\end{pmatrix} \geq 0 \quad \ \ \  (p \geq 0).
		\end{aligned}
	\end{equation}
	Note that $m_0\ebra{x^{p-1} V^\prime \rho_0}$ can be expressed solely in terms of $\{m_p\}_{p \in \mathbb{N}}$ since $V$ is a polynomial.
	
	If we further assume that $V$ is symmetric ($V\ebra{\cdot} = V\ebra{-\cdot}$) and $W = |x|^2/2$, it follows immediately that probability measures $\rho_0$ such that $m_1\ebra{\rho_0} = 0$ form a family of symmetric trivial solutions of (\ref{eq.16}), parameterized by $\sigma$
	\begin{equation*}\label{eq.17}
		\rho_0\ebra{\sigma} = \frac{1}{\tilde{Z}\ebra{\kappa, \sigma}}\cdot \exp \ebra{\ebra{-V - \kappa \cdot \frac{1}{2}|x|^2}\cdot \frac{2}{\sigma^2}},
	\end{equation*}
	or, equivalently, parameterized by $\lambda = -2\kappa/\sigma^2$
	\begin{equation*}\label{eq.17.1}
		\rho_0(\lambda)=\frac{1}{\tilde{Z}\ebra{\lambda}}\cdot \exp\ebra{\lambda \cdot \ebra{V\kappa^{-1} + \frac{1}{2}|x|^2}},
	\end{equation*}
	with $\tilde{Z}\ebra{\kappa, \sigma}$ and $\tilde{Z}\ebra{\lambda}$ denoting the normalizing constants.
	In this case, non-trivial solutions of the bifurcation equation appear in pairs.
	
	Assume the existence of a family of trivial solutions $\{\rho_0\ebra{\lambda}\}$. After applying the following transformation,
	\begin{equation*}
		\tilde{F} \ebra{\lambda, \delta \rho} := F\ebra{\lambda, \rho_0\ebra{\lambda} + \delta \rho} = \left.\ebra{id - \mathcal{T}\ebra{\lambda, \cdot}}\ebra{\rho}\right|_{\rho = \rho_0\ebra{\lambda} + \delta \rho}.
	\end{equation*}
	$\tilde{F}$ depends on $\lambda$ only through the first variable, which is a prerequisite for the pitchfork bifurcation theorem. 
	
	\section{Main Theorems}\label{the_main_results}
	\subsection{The abstract theorem for continuous phase transitions}
	For a general interaction  $W$, we have the following abstract bifurcation theorem on a Banach space $X$ within the Crandall--Rabinowitz framework \hl{(see Appendix {\ref{bifur}})}, derived with respect to the parameter $\lambda = -2\kappa/\sigma^2$.

	Let $(\lambda_\star, \rho_0(\lambda_\star))$ be the phase transition point under consideration. Given $\delta_1, \delta_2 > 0$, set $J := B_{\delta_1}(\lambda_\star)$ and let $\{\rho_0(\lambda)\}_{\lambda \in J}$ be a family of trivial solutions. Define the local neighborhood of the phase transition point by $\Omega:= \cup_{\lambda \in J} \{\lambda\} \times B_{\delta_2}\ebra{\rho_0 \ebra{\lambda}}$. Then the bifurcation equation (\ref{eq.15}) after a transformation is given by
	\begin{equation}\label{eq.91}
		\begin{aligned}
			0 & = \tilde{F}\ebra{\lambda, \delta\rho} = \ebra{id - \mathcal{T}\ebra{\lambda, \cdot}}\ebra{\rho_0\ebra{\lambda} + \delta\rho} = \ebra{id - \lambda K} \ebra{\delta \rho} + G\ebra{\lambda, \delta \rho},
		\end{aligned}
	\end{equation}
	where, for any $\omega \in X$,
	$$K\omega := \rho_0(\lambda) \cdot \ebra{W * \omega - \int_{\mathbb{R}^n} W * \omega \cdot \rho_0(\lambda) \dif x}$$
	is defined so that $\lambda K$ is the linear part of $\mathcal{T}$, and $G$ is the remainder in $\tilde{F}$. Note that in the continuous phase transition case, $K$ implicitly depends on $\lambda$ via $\rho_0(\lambda)$. We shall explicitly write out $K = K_\lambda$ whenever required to avoid ambiguity.
	
	\begin{theorem}[Abstract pitchfork bifurcation theorem for continuous phase transitions]\label{thm.7}
		Consider the equation \eqref{eq.1} and its bifurcation equation \eqref{eq.91}. Assume that (H1)--(H3) hold, and $\tilde{F}$ has Fr\'{e}chet derivatives up to order three. Then 

		(I) If $\ebra{\lambda_\star, \rho_0\ebra{\lambda_\star}}$ is a phase transition point, then
		\begin{equation}\label{eq.93}
			\partial_{\delta\rho} \tilde{F} \ebra{\lambda_\star, 0} \text{ is not invertible}.
		\end{equation}
		
		(II) In addition to \eqref{eq.93}, assume further that $id - \lambda_\star K_{\lambda_\star}$ has Fredholm index zero with $\text{codim range}$\newline$(id - \lambda_\star K_{\lambda_\star}) = 1$, and $\lambda_\star$ is semi-simple with $\text{null }\ebra{id - \lambda_\star K_{\lambda_\star}} = \text{span}\{\tilde{v}_\star\}$ for some $\tilde{v}_\star$. Additionally, suppose that $K_\lambda$ is a compact operator on $X$ for any fixed $\lambda \in J$, and $G$ is completely continuous with $G\ebra{\lambda, \cdot} = o\ebra{\norm{\delta\rho}}$ uniformly with respect to $\lambda \in J$.
		Let $P_{\lambda_\star}$ be the spectral projection operator onto $\text{null }\ebra{id - \lambda_\star K_{\lambda_\star}}$. We use $[\cdot]$ to denote the action of a multilinear operator on direction vectors. Assume also that
		\begin{enumerate}
			\item (The transversality condition) $P_{\lambda_\star} \partial_{\lambda, \delta \rho}^2 \tilde{F}\ebra{\lambda_\star, 0}[\tilde{v}_\star] \neq 0$, that is
			\begin{equation}\label{eq.95}
				P_{\lambda_\star}\left(-\rho_0\cdot \left(\lambda_\star^{-1}\cdot \frac{\tilde{v}_\star}{\rho_0} + \frac{\tilde{v}_\star}{\rho_0} \cdot \frac{\partial_\lambda \rho_0}{\rho_0} - \int_{\mathbb{R}^n} \frac{\tilde{v}_\star}{\rho_0} \cdot \frac{\partial_\lambda \rho_0}{\rho_0} \cdot \rho_0 \dif x\right)\right)  \neq 0,
			\end{equation}
			\item (The pitchfork condition) $P_{\lambda_\star} \partial_{\delta\rho, \delta\rho}^2 \tilde{F}\ebra{\lambda_\star, 0} [\tilde{v}_\star, \tilde{v}_\star] = 0$, that is
			\begin{equation}\label{eq.96}
				P_{\lambda_\star}\left(-\rho_0\cdot\left(\ebra{\frac{\tilde{v}_\star}{\rho_0}}^2 - \int_{\mathbb{R}^n} \ebra{\frac{\tilde{v}_\star}{\rho_0}}^2\cdot \rho_0 \dif x\right)\right) = 0,
			\end{equation}
			\item (The non-degenerate condition)
			\begin{equation}\label{eq.97}
				\begin{aligned}
					& P_{\lambda_\star}\ebra{1/3\cdot \partial_{\delta \rho, \delta \rho, \delta \rho}^3 \tilde{F}\ebra{\lambda_\star, 0}[\tilde{v}_\star, \tilde{v}_\star, \tilde{v}_\star] + \partial_{\delta\rho, \delta\rho}^2\tilde{F}\ebra{\lambda_\star, 0}[\tilde{v}_\star, 2 \partial_s w\ebra{\lambda_\star, 0}]} \neq 0, \\
					& 2\partial_s w\ebra{\lambda_\star, 0} = -\ebra{\partial_{\delta \rho}\tilde{F}\ebra{\lambda_\star, 0}|_{\text{range } (id - \lambda_\star K_{\lambda_\star})}}^{-1} \ebra{id - P_{\lambda_\star}}\partial_{\delta \rho, \delta \rho}^2 \tilde{F}\ebra{\lambda_\star, 0}[\tilde{v}_\star, \tilde{v}_\star],
				\end{aligned}
			\end{equation}
			where
			\begin{equation}\label{eq.98}
				\begin{aligned}
					& \partial_{\delta\rho, \delta\rho, \delta\rho}^3 \tilde{F}\ebra{\lambda_\star, 0}[\tilde{v}_\star, \tilde{v}_\star, \tilde{v}_\star]
					\\
					& = -\rho_0\cdot \left(\ebra{\frac{\tilde{v}_\star}{\rho_0}}^3 - \int_{\mathbb{R}^n} \ebra{\frac{\tilde{v}_\star}{\rho_0}}^3 \cdot \rho_0 \dif x - 3\cdot \frac{\tilde{v}_\star}{\rho_0}\cdot \int_{\mathbb{R}^n} \ebra{\frac{\tilde{v}_\star}{\rho_0}}^2 \cdot \rho_0 \dif x \right), \\
					& \partial_{\delta \rho, \delta \rho}^2 \tilde{F}\ebra{\lambda_\star, 0}[\tilde{v}_\star, \omega] = -\lambda_\star \cdot \rho_0 \cdot \left(\frac{\tilde{v}_\star}{\rho_0} \cdot \frac{K\omega}{\rho_0} - \int_{\mathbb{R}^n} \frac{\tilde{v}_\star}{\rho_0} \cdot \frac{K \omega}{\rho_0} \cdot \rho_0 \dif x\right).
				\end{aligned}
			\end{equation}
		\end{enumerate}
		Then the model \eqref{eq.1} undergoes a continuous phase transition governed by a pitchfork bifurcation at $\ebra{\lambda_\star, \rho_0\ebra{\lambda_\star}}$.
		
		Moreover, locally, the non-trivial solution curve is represented as
		\begin{equation}\label{eq.99}
			\{\ebra{\lambda\ebra{s}, \rho_0\ebra{\lambda_\star} + s \tilde{v}_\star + s w\ebra{s}} \in \Omega: \ \lambda \ebra{0} = \lambda_\star, \ \lambda^\prime\ebra{0}= 0, \ \lambda^{\prime\prime}\ebra{0} \neq 0, \ sw\ebra{s} = o\ebra{s}\}.
		\end{equation}
		The bifurcation equation can be expanded as the normal form for a pitchfork bifurcation perturbed by higher-order terms
		\begin{equation*}\label{eq.100}
			\begin{aligned}
				& \left(P_{\lambda_\star}\partial_{\lambda, \delta \rho}^2 \tilde{F}\ebra{\lambda_\star, 0}[\tilde{v}_\star]\right) \cdot \ebra{\lambda - \lambda_\star} s + \left(1/2\cdot  P_{\lambda_\star}\partial_{\delta \rho, \delta \rho}^2\tilde{F}\ebra{\lambda_\star, 0}[\tilde{v}_\star, 2 \partial_s \omega\ebra{\lambda_\star, 0}]\right.\\
				& \left. + 1/6\cdot P_{\lambda_\star}\partial_{\delta \rho, \delta \rho, \delta\rho}^3 \tilde{F}\ebra{\lambda_\star, 0}[\tilde{v}_\star, \tilde{v}_\star, \tilde{v}_\star]\right)\cdot s^3 + \cdots = 0,
			\end{aligned}
		\end{equation*}
		which enables us to determine the direction of bifurcation.
		
		Globally, the nontrivial solution component $C$ passing through the phase transition point is either unbounded in $\mathbb{R} \times X$, or is connected to another $(\lambda^\prime, \rho(\lambda^\prime))$ such that $(id - \lambda^\prime K_{\lambda^\prime})\tilde{v}^\prime = 0$ for some $\tilde{v}^\prime \in X$. In addition, $C$ consists of two subcontinua $C^+$ and $C^-$ such that $C^+ \cap C^- \cap B_\epsilon\ebra{\lambda_\star, \rho_0\ebra{\lambda_\star}} = \{\ebra{\lambda_\star, \rho_0\ebra{\lambda_\star}}\}$ and $C^{+, -} \cap \partial B_\epsilon\ebra{\lambda_\star, \rho_0\ebra{\lambda_\star}} \neq \emptyset$ for all sufficiently small $\epsilon > 0$.
	\end{theorem}

	Different interaction kernels may lead to completely distinct behaviors (for example, attractive versus repulsive kernels). We are mainly interested in the phase transition driven by competition between noise and attractive interaction. Hence, we further introduce condition (H4), under which the spectrum of the linear part can be explicitly expressed. This condition is also adopted in \cite{Daw83, LcP20, LcP25, MQT16}. 
	\begin{itemize}
		\item[(H4)] (Attractive symmetric quadratic interaction) $W\ebra{x} = |x|^2/2$.
	\end{itemize}
	
	
	Under (H4), the divergence part induced by interaction becomes
	\begin{equation*}\label{eq.12}
		\text{div}\ebra{\kappa \cdot \ebra{x - m_1\ebra{\mu}}\mu},
	\end{equation*}
	where $m_1\ebra{\mu} = \int_{\mathbb{R}^n} x \mu\ebra{\text{d}x}$ denotes the mean value, or barycenter, of the distribution $\mu$. We refer to the effect of $W$ as ``attraction to the barycenter".
	
	Under (H4), the bifurcation theorem in $\mathbb{R}^n$ possesses a specific and delicate structure. Let us consider the ``covariance matrix" generated by $\rho_0$,
	\begin{equation*}\label{eq.71.2.1}
		\text{Cov}\ebra{\rho_0} := \begin{pmatrix}
			m_{e_1+e_1} - m_{e_1}\cdot m_{e_1} & \cdots & m_{e_n + e_1} - m_{e_n} \cdot m_{e_1}\\
			\vdots & \cdots & \vdots\\
			m_{e_1 + e_n} - m_{e_1} \cdot m_{e_n} & \cdots & m_{e_n + e_n} - m_{e_n} \cdot m_{e_n}
		\end{pmatrix}_{n \times n}\ebra{\rho_0}, 
	\end{equation*}
	where $e_i$ denotes the $i$-th unit vector in $\mathbb{R}^n$, and for $\alpha = \ebra{\alpha_1, \dots, \alpha_n}$, the indexed moments are defined as
	\begin{equation*}\label{eq.101}
		m_\alpha\ebra{\rho_0} := \int_{\mathbb{R}^n} \ebra{\prod_{i = 1}^n x_{i}^{\alpha_i}} \cdot \rho_0 \dif x_1 \cdots \dif x_n.
	\end{equation*}
	Let eigenvalues and corresponding orthogonal unit eigenvectors of the covariance matrix be: $0 < \theta_1 \leq \dots \leq \theta_n$; $v_1, \dots, v_n$, respectively. Choose $X = L^1\ebra{\mathbb{R}^n, 1 + |x|^2}$. Notation $\ebra{v_i}_j$ denotes the $j$-th component of the vector $v_i$. For $i \in \{1, \dots, n\}$, we define:
	\begin{equation*}\label{eq.73.1.1}
		\begin{aligned}
			& \tilde{v}_i := \sqrt{2} \pi^{-n/4}\cdot \rho_0 \cdot \sum_{k=1}^n \ebra{x_k - m_{e_k}} \cdot \ebra{v_i}_k, \\
			& \xi^{(i)}_k \cdot \ebra{v_i}_k:= \frac{\lambda_i}{2} \cdot \sum_{r=1}^n \ebra{m_{2e_k + e_r} - m_{2e_k} \cdot m_{e_r}} \ebra{\rho_0}\cdot \ebra{v_i}_r,\\
			& \tilde{u}_i := \sum_{k=1}^n \ebra{x_k +\xi_k^{(i)}}\cdot \ebra{v_i}_k, \qquad \lambda_i := -\theta_i^{-1}.
		\end{aligned}
	\end{equation*}
	Then one can verify that $K \tilde{v}_i = \lambda_i^{-1} \tilde{v}_i$ and $K^* \tilde{u}_i = \lambda_i^{-1} \tilde{u}_i$. Here $K^*$ denotes the  adjoint of $K$. Moreover, the projection onto $\text{span}\{\tilde{v}_i\}$ is
	\begin{equation*}\label{eq.72.1.2}
		\frac{\eangle{\cdot, \tilde{u}_i}}{\eangle{\tilde{v}_i, \tilde{u}_i}} \cdot  \tilde{v}_i.
	\end{equation*}
	Details are given in Section \ref{aux_lems}. With the above-defined notations, we have the following bifurcation theorem for continuous phase transitions under  quadratic interaction.
	
	\begin{theorem}[Pitchfork bifurcation theorem for continuous phase transitions, quadratic-interaction case]\label{thm.8}
		Consider the equation \eqref{eq.1} and its bifurcation equation \eqref{eq.91}. Assume that (H1)--(H4) hold. \hl{Then $\tilde{F}$ has Fr\'{e}chet derivatives up to order three on $X = L^1\ebra{\mathbb{R}^n, 1 + |x|^2}$ and $K_\lambda$ is a compact operator on $L^1(\mathbb{R}^n, 1+|x|^2)$ for any given $\lambda \in J$. The completely continuous term $G\ebra{\lambda, \cdot} = o\ebra{\norm{\delta\rho}}$ is a small remainder uniformly in $\lambda \in J$. Moreover, }
		
		(I) If $\ebra{\lambda_\star, \rho_0\ebra{\lambda_\star}}$ is a  phase transition point, then (\ref{eq.93}) holds. Under (H4), (\ref{eq.93}) becomes the generalized Dawson's criterion
		\begin{equation}\label{eq.102}
			-\lambda_\star^{-1} \text{ belongs to eigenvalues of } \text{Cov}\ebra{\rho_0\ebra{\lambda_\star}}.
		\end{equation}
		
		(II) Besides (\ref{eq.102}), if we further assume $-\lambda_\star^{-1}$ is a simple eigenvalue of the covariance matrix with unit eigenvector $v_\star$, as well as
		\begin{enumerate}
			\item The transversality condition
			\begin{equation}\label{eq.103}
				\begin{aligned}
					P_{\lambda_\star} \partial_{\lambda, \delta\rho}^2 \tilde{F}\ebra{\lambda_\star, 0}[\tilde{v}_\star]/\tilde{v}_\star = & -\lambda_\star^{-1} + \lambda_\star \cdot \sum_{k, r = 1}^n \ebra{v_\star}_k\ebra{v_\star}_r \cdot \ebra{m_{e_k + e_r}\ebra{\partial_{\lambda} \rho_0} \\
						& - m_{e_k}\ebra{\rho_0} \cdot m_{e_r}\ebra{\partial_\lambda \rho_0} - m_{e_k} \ebra{\partial_\lambda \rho_0} \cdot m_{e_r}\ebra{\rho_0}} \neq 0.
				\end{aligned}
			\end{equation}
			\item The pitchfork condition
			\begin{equation}\label{eq.104}
				\begin{aligned}
					& P_{\lambda_\star} \partial^2_{\delta \rho, \delta \rho} \tilde{F}\ebra{\lambda_\star, 0}  [\tilde{v}_\star, \tilde{v}_\star]/\tilde{v}_\star = \sqrt{2} \pi^{-n/4}\cdot \lambda_\star \cdot \sum_{k, r, s = 1}^n \ebra{v_\star}_k \ebra{v_\star}_r \ebra{v_\star}_s \\
					& \qquad\qquad\qquad \cdot\int_{\mathbb{R}^n} \ebra{x_k - m_{e_k}\ebra{\rho_0}}\ebra{x_r - m_{e_r}\ebra{\rho_0}}\ebra{x_s - m_{e_s}\ebra{\rho_0}}\cdot \rho_0 \dif x = 0.
				\end{aligned}
			\end{equation}
			\item The non-degenerate condition
			\begin{equation*}\label{eq.105}
				P_{\lambda_\star}\ebra{1/3\cdot \partial_{\delta \rho, \delta \rho, \delta \rho}^3 \tilde{F}\ebra{\lambda_\star, 0}[\tilde{v}_\star, \tilde{v}_\star, \tilde{v}_\star] + \partial_{\delta\rho, \delta\rho}^2\tilde{F}\ebra{\lambda_\star, 0}[\tilde{v}_\star, 2 \partial_s w\ebra{\lambda_\star, 0}]} \neq 0,
			\end{equation*}
			where
			\begin{equation}\label{eq.106}
				\begin{aligned}
					& P_{\lambda_\star} \partial_{\delta \rho, \delta \rho, \delta \rho}^3 \tilde{F}\ebra{\lambda_\star, 0}[\tilde{v}_\star, \tilde{v}_\star, \tilde{v}_\star]/\tilde{v}_\star = 2 \pi^{-n/2} \cdot \lambda_\star \cdot \sum_{k, r, s, l = 1}^n \ebra{v_\star}_k \ebra{v_\star}_r \ebra{v_\star}_s \ebra{v_\star}_l \\
					& \cdot\left(\int_{\mathbb{R}^n} \ebra{x_k - m_{e_k}\ebra{\rho_0}} \ebra{x_r - m_{e_r}\ebra{\rho_0}} \cdot \ebra{x_s - m_{e_s}\ebra{\rho_0}} \ebra{x_l - m_{e_l}\ebra{\rho_0}} \cdot \rho_0 \dif x \right.\\
					& \qquad \qquad \quad -3 \cdot \left(\int_{\mathbb{R}^n} \ebra{x_k - m_{e_k}\ebra{\rho_0}} \ebra{x_r - m_{e_r}\ebra{\rho_0}} \cdot \rho_0 \dif x\right) \\
					& \qquad \qquad \qquad \quad \cdot \left. \left(\int_{\mathbb{R}^n} \ebra{x_s - m_{e_s}\ebra{\rho_0}}\ebra{x_l - m_{e_l}\ebra{\rho_0}} \cdot \rho_0 \dif x\right)\right),
				\end{aligned}
			\end{equation}
			and
			\begin{equation}\label{eq.107}
				\begin{aligned}
					P_{\lambda_\star}\partial_{\delta \rho, \delta \rho}^2 & \tilde{F} \ebra{\lambda_\star, 0} [\tilde{v}_\star, 2 \partial_s w\ebra{\lambda_\star, 0}]/\tilde{v}_\star \\
					& \quad = -2 \pi^{-n/2} \cdot \lambda_\star^2 \cdot \sum_{\lambda_j \neq \lambda_\star} \frac{\lambda_j}{\lambda_j - \lambda_\star} \cdot \biggl(\sum_{k, r, s = 1}^n \ebra{v_\star}_k \ebra{v_\star}_r \ebra{v_j}_s \\
					& \qquad \cdot \int_{\mathbb{R}^n} \ebra{x_k - m_{e_k}\ebra{\rho_0}} \ebra{x_r - m_{e_r}\ebra{\rho_0}} \ebra{x_s - m_{e_s}\ebra{\rho_0}} \cdot \rho_0 \dif x\biggr)^2.
				\end{aligned}
			\end{equation}
		\end{enumerate}
		Then, the system \eqref{eq.1} exhibits a continuous phase transition controlled by the pitchfork bifurcation at $\ebra{\lambda_\star, \rho_0\ebra{\lambda_\star}}$. Moreover, both the local characterizations of solution curves and the global arguments in Theorem \ref{thm.7} hold.
	\end{theorem}
	
	\begin{remark}\label{rem.7}
		From (\ref{eq.104}), the pitchfork condition is satisfied as long as $\rho_0\ebra{\lambda_\star}$ is symmetric in the direction of eigenvector $v_\star$. The pitchfork bifurcation implies symmetry breaking and an exchange of stability in the eigendirection.
	\end{remark}
	
	\begin{remark}\label{rem.8}
		We can describe how the barycenter of solutions for the bifurcation equation changes near the bifurcation point. Let $m(\lambda(s))$ be the solution of the moment bifurcation equation (\ref{eq.16}) with $\lambda(s)$ as in \eqref{eq.99}, and let the corresponding invariant measure be
		\begin{equation*}\label{eq.107.2}
			\rho\ebra{s} = \frac{1}{\tilde{Z}\ebra{s}}\cdot \exp\ebra{\lambda\ebra{s}\cdot \ebra{V\kappa^{-1} + \frac{1}{2}|x|^2 - \eangle{m\ebra{\lambda\ebra{s}}, x}}}.
		\end{equation*}
		Differentiating it at $s=0$ and using the fact that $\lambda^\prime\ebra{0} = 0$, we have
		\begin{equation*}\label{eq.107.3}
			\begin{aligned}
				\partial_s \rho\ebra{s}|_{s=0} & = \left. -\lambda\ebra{s} \cdot \eangle{\partial_s m\ebra{\lambda\ebra{s}}, x - \int_{\mathbb{R}^n} x \cdot \rho\ebra{s} \dif x}\cdot \rho\ebra{s}\right|_{s=0}\\
				& = - \lambda_\star \cdot \eangle{\partial_s m\ebra{\lambda\ebra{s}}|_{s=0}, x - \int_{\mathbb{R}^n} x \cdot \rho_0 \dif x} \cdot \rho_0.
			\end{aligned}
		\end{equation*}
		On the other hand, in view of \eqref{eq.99}, we have $\partial_s \rho\ebra{s}|_{s=0} = \tilde{v}_\star$. By comparing the two expressions, we deduce that the direction in $\mathbb{R}^n$ along which the barycenter moves is exactly the eigenvector direction; \hl{that is, there exists $\tilde{k} \in \mathbb{R} \setminus \{0\}$ such that}
		\begin{equation*}\label{eq.107.4}
			\partial_s m\ebra{\lambda\ebra{s}}|_{s = 0} = \tilde{k}  \cdot v_\star.
		\end{equation*}
	\end{remark}
	
	Under the global symmetry assumption on $V$, the McKean-Vlasov diffusion naturally admits a family of symmetric invariant measures $\{\rho_0\ebra{\lambda}\}$ as trivial solutions, which can be written explicitly as
	\begin{equation*}\label{eq.108.2}
		\rho_0\ebra{\lambda} = \frac{1}{\tilde{Z}\ebra{\lambda}}\cdot\exp\ebra{\lambda \cdot \ebra{V \kappa^{-1}  + \frac{1}{2}|x|^2}}.
	\end{equation*}
	Since $V\ebra{\cdot} = V\ebra{-\cdot}$, the odd-order moments of $\{\rho_0\ebra{\lambda}\}$ vanish:
	\begin{equation*}\label{eq.108}
		\int_{\mathbb{R}^n} \ebra{\prod_{i=1}^n x_i^{\alpha_i}} \cdot \rho_0 \dif x = 0, \qquad |\alpha| \text{ is odd}.
	\end{equation*}
	Therefore, when investigating phase transitions originating from $\{\rho_0(\lambda)\}$ as the temperature decreases, the analytical complexity is considerably reduced. We have the following corollary tailored to this case:
	\begin{corollary}\label{cor.1}
		Assume (H1)--(H4) and the symmetry of the potential $V$. Choose the family of symmetric invariant measures $\{\rho_0\ebra{\lambda}\}$ as trivial solutions. It is necessary that the generalized Dawson's criterion is satisfied
		\begin{equation}\label{eq.109}
			-\lambda_\star^{-1} \text{ belongs to eigenvalues of } \text{Cov}\ebra{\rho_0\ebra{\lambda_\star}}.
		\end{equation}
		Moreover, if $-\lambda_\star^{-1}$ is a simple eigenvalue with unit eigenvector $v_\star$, and the following conditions hold:
		\begin{enumerate}
			\item The transversality condition
			\begin{equation*}\label{eq.110}
				P_{\lambda_\star} \partial_{\lambda, \delta\rho}^2 \tilde{F}\ebra{\lambda_\star, 0}[\tilde{v}_\star]/\tilde{v}_\star = -\lambda_\star^{-1} + \lambda_\star \cdot \sum_{k, r = 1}^n \ebra{v_\star}_k\ebra{v_\star}_r \cdot m_{e_k + e_r}\ebra{\partial_{\lambda} \rho_0} \neq 0,
			\end{equation*}
			where
			\begin{equation*}\label{eq.111}
				\partial_\lambda \rho_0 = \rho_0 \cdot \left(V/\kappa + W * \rho_0 - \int_{\mathbb{R}^n} \ebra{V/\kappa + W * \rho_0} \cdot \rho_0 \dif x\right).
			\end{equation*}
			\item The pitchfork condition $P_{\lambda_\star} \partial_{\delta \rho, \delta \rho}^2 \tilde{F}\ebra{\lambda_\star, 0}[\tilde{v}_\star, \tilde{v}_\star] = 0$ is automatically satisfied.
			\item The non-degenerate condition $P\ebra{1/3\cdot \partial_{\delta \rho, \delta \rho, \delta \rho}^3 \tilde{F}\ebra{\lambda_\star, 0}[\tilde{v}_\star, \tilde{v}_\star, \tilde{v}_\star]} \neq 0$.
		\end{enumerate}
		Then the system undergoes a pitchfork bifurcation at $\ebra{\lambda_\star, \rho_0\ebra{\lambda_\star}}$.
	\end{corollary}
	
	\subsection{The abstract theorem for discontinuous phase transitions}We propose a local saddle-node bifurcation (turning-point) theorem for discontinuous phase transitions. Compared with the former pitchfork case, the assumptions are quite different in that there are no trivial solutions. Instead, we focus on a single critical point $\ebra{\lambda_\star, \rho_\star}$: two branches of solution curves lie on one side of the line $\lambda = \lambda_\star$, and no branch lies on the other side.
	
	Let $\ebra{\lambda_\star, \rho_\star} \in \mathbb{R}\times X$ be the phase transition point under investigation. The bifurcation equation \eqref{eq.15} is given by
	\begin{equation}\label{eq.118.1}
		0 = F\ebra{\lambda, \rho} = \ebra{id - \mathcal{T}\ebra{\lambda, \cdot}}\ebra{\rho}.
	\end{equation}
	
	\begin{theorem}[Saddle-node bifurcation theorem]\label{thm.9}
		Consider the system \eqref{eq.1} and its bifurcation equation \eqref{eq.118.1}. Assume that (H1)--(H3) hold, and $F$ has Fr\'{e}chet derivatives up to order two. Then 
		
		(I) If the phase transition point $\ebra{\lambda_\star, \rho_\star}$ satisfies $F\ebra{\lambda_\star, \rho_\star} = 0$, it is necessary that
		\begin{equation}\label{eq.118.2}
			\partial_\rho F\ebra{\lambda_\star, \rho_\star} = id - \lambda_\star K \text{ is not invertible. }
		\end{equation}
		Here $K$ is defined the same as before.
		
		(II) In addition to \eqref{eq.118.2}, assume
		\begin{equation}\label{eq.118.3}
			\begin{aligned}
				& \lambda_\star \text{ is an algebraically simple characteristic value for the compact operator }K, \\
				& \exists \ \tilde{v}_\star:\ \text{null }\ebra{id - \lambda_\star K} =  \text{span} \{\tilde{v}_\star\} \text{ with dimension } 1.
			\end{aligned}
		\end{equation}
		Let $P_{\lambda_\star}$ be the spectral projection operator onto $\text{null } \ebra{id - \lambda_\star K}$. Assume also that
		\begin{enumerate}
			\item (The transversality condition) $ P_{\lambda_\star}\partial_\lambda F\ebra{\lambda_\star, \rho_\star} \neq 0$, that is
			\begin{equation*}\label{eq.118.4}
				P_{\lambda_\star}\left(-\rho_\star\cdot \left(V \kappa^{-1} + W * \rho_\star - \int_{\mathbb{R}^n}\ebra{V \kappa^{-1} + W * \rho_\star}\cdot \rho_\star \dif x\right)\right) \neq 0.
			\end{equation*}
			\item (The non-degenerate condition) $P \partial_{\rho, \rho}^2 F\ebra{\lambda_\star, \rho_\star}[\tilde{v}_\star, \tilde{v}_\star] \neq 0$, that is
			\begin{equation*}\label{eq.118.5}
				P_{\lambda_\star}\left(-\rho_\star \cdot \left(\ebra{\frac{\tilde{v}_\star}{\rho_\star}}^2 - \int_{\mathbb{R}^n} \ebra{\frac{\tilde{v}_\star}{\rho_\star}}^2\cdot \rho_\star \dif x\right)\right) \neq 0.
			\end{equation*}
		\end{enumerate}
		
		Then the system exhibits a phase transition governed by a saddle-node bifurcation at $\ebra{\lambda_\star, \rho_\star}$. The non-trivial solution curve can be locally represented by
		\begin{equation*}\label{eq.118.6}
			\{\ebra{\lambda\ebra{s}, \rho_\star + s \tilde{v}_\star + s \tilde{w}\ebra{s}}: \ \lambda\ebra{0} = \lambda_\star, \ \lambda^\prime\ebra{0} = 0,\ \lambda^{\prime \prime}\ebra{0} \neq 0, \ s\tilde{w}\ebra{s} = sw\ebra{\lambda\ebra{s}, s} = o\ebra{s}\}.
		\end{equation*}
		The bifurcation equation can be locally expanded as a normal form for a saddle-node bifurcation perturbed by higher-order terms
		\begin{equation*}\label{eq.118.7}
			P_{\lambda_\star} \partial_\lambda F\ebra{\lambda_\star, \rho_\star} \cdot \ebra{\lambda - \lambda_\star} + \left(1/2 \cdot P_{\lambda_\star} \partial_{\rho, \rho}^2 F\ebra{\lambda_\star, \rho_\star}[\tilde{v}_\star, \tilde{v}_\star]\right)\cdot s^2 + \cdots = 0.
		\end{equation*}
		
		(III) Under (H4), choose $X = L^1\ebra{\mathbb{R}^n, 1 + |x|^2}$. Then $F$ has Fr\'{e}chet derivatives up to order three. In this case,   (\ref{eq.118.3}) is replaced by the generalized Dawson's criterion in $\mathbb{R}^n$
		\begin{equation*}\label{eq.118.6.1}
			-\lambda_\star^{-1} \text{ is a simple eigenvalue of the covariance matrix } \text{Cov}\ebra{\rho_\star}
		\end{equation*}
		and other assumptions take the form
		\begin{enumerate}
			\item The transversality condition becomes
			\begin{equation*}\label{eq.118.8}
				\begin{aligned}
					P_{\lambda_\star} \partial_\lambda F\ebra{\lambda_\star, \rho_\star}/\tilde{v}_\star & = \ebra{\sqrt{2}\pi^{-n/4}}^{-1} \cdot \lambda_\star \cdot \sum_{k=1}^n \ebra{v_\star}_k \cdot \left(\int_{\mathbb{R}^n} \ebra{V \kappa^{-1} + W * \rho_\star}\cdot x_k \cdot \rho_\star \dif x \right. \\
					& \left. \quad - \left(\int_{\mathbb{R}^n} \ebra{V \kappa^{-1} + W * \rho_\star} \cdot \rho_\star \dif x \right) \cdot \left(\int_{\mathbb{R}^n} x_k \cdot \rho_\star \dif x \right)\right) \neq 0.
				\end{aligned}
			\end{equation*}
			\item The non-degenerate condition becomes
			\begin{equation*}\label{eq.118.9}
				\begin{aligned}
					& P_{\lambda_\star} \partial^2_{\rho, \rho} F\ebra{\lambda_\star, \rho_\star}  [\tilde{v}_\star, \tilde{v}_\star]/\tilde{v}_\star = \sqrt{2} \pi^{-n/4}\cdot \lambda_\star \cdot \sum_{k, r, s = 1}^n \ebra{v_\star}_k \ebra{v_\star}_r \ebra{v_\star}_s  \\
					& \qquad\qquad\qquad \cdot \int_{\mathbb{R}^n} \ebra{x_k - m_{e_k}\ebra{\rho_\star}}\ebra{x_r - m_{e_r}\ebra{\rho_\star}}\ebra{x_s - m_{e_s}\ebra{\rho_\star}}\cdot \rho_\star \dif x \neq 0.
				\end{aligned}
			\end{equation*}
		\end{enumerate}
	\end{theorem}
	
	
	\subsection{Examples} \label{exs}
	The motivating examples presented in this subsection satisfy the conditions of the abstract theorems proposed earlier. Their proofs are provided in Section \ref{ex_proof}.
	\begin{example}\label{ex.1}
		(Dawson's symmetric double-well model in $\mathbb{R}$ \cite{Daw83})
		Let $n = 1$ and consider the symmetric double-well potential
		\begin{equation*}\label{eq.91.1.0.1}
			V\ebra{x} = \frac{1}{4}x^4 - \frac{1}{2}x^2.
		\end{equation*}
		Let the interaction be given by $W(x) = x^2/2$, with interaction strength $\kappa > 0$ fixed. Choose the family of symmetric invariant measures
		\begin{equation*}\label{eq.91.1.1}
			\rho_0 \ebra{\lambda} = \frac{1}{\tilde{Z}\ebra{\lambda}}\cdot\exp\ebra{\lambda\cdot \ebra{V\kappa^{-1} + \frac{1}{2}x^2}}, \qquad \lambda = -\frac{2\kappa}{\sigma^2} < 0,
		\end{equation*}
		as trivial solutions. Then 
		\begin{enumerate}
			\item Locally, there exists a phase transition point $\ebra{\lambda_\star, \rho_0\ebra{\lambda_\star}}$, at which the system exhibits a continuous phase transition controlled by a subcritical pitchfork bifurcation determined by \textit{Dawson's criterion}
			\begin{equation}\label{eq.91.1.2}
				\frac{\sigma_\star^2}{2\kappa} = -\lambda_\star^{-1} = \int_{\mathbb{R}} x^2 \cdot \rho_0\ebra{\lambda_\star} \dif x.
			\end{equation}
			\item Globally, the phase transition point and the solution for (\ref{eq.91.1.2}) are both unique. The bifurcation curve is unbounded in $\mathbb{R}\times L^1\ebra{\mathbb{R}, 1+x^2}$. For any $\sigma > \sigma_\star$, the unique invariant measure is the trivial solution $\rho_0$. For any $0 < \sigma < \sigma_\star$, there exist three invariant measures: $\rho_{-}, \rho_0, \rho_{+}$, where $\rho_-, \rho_+$ are both asymmetric and $\rho_{-}\ebra{-x} = \rho_{+}\ebra{x}$  for $x \in \mathbb{R}$. Moreover, 
			\begin{equation*}\label{eq.91.1.3}
				\int_{\mathbb{R}} x^2 \cdot \rho_0 \dif x > \frac{\sigma^2}{2\kappa}, \qquad \int_{\mathbb{R}} x^2 \cdot \rho_{+, -} \dif x - \left(\int_{\mathbb{R}} x \cdot \rho_{+, -} \dif x\right)^2< \frac{\sigma^2}{2\kappa}.
			\end{equation*}
		\end{enumerate}
	\end{example}
	
	\begin{example}\label{ex.3}
		(An asymmetric double-well model in $\mathbb{R}$ with a discontinuous phase transition)
		Let $n = 1$ and consider the asymmetric double-well potential with a deeper right well
		\begin{equation*}\label{eq.91.1.0.1.1}
			V\ebra{x} = \frac{1}{4}x^4 - \frac{1}{3} x^3 - \frac{1}{2}x^2.
		\end{equation*}
		Let the attractive interaction term be given by $W(x) = x^2/2$ with interaction strength $\kappa = 1$ fixed for convenience. Then 
		\begin{enumerate}
			\item Locally, if we denote
			\begin{equation*}\label{eq.91.1.8}
				\rho\ebra{\lambda, m} := \frac{1}{Z\ebra{\lambda, m}} \cdot \exp \ebra{\lambda \cdot \ebra{V + \frac{1}{2}x^2 - mx}},
			\end{equation*}
			then there exists a phase transition point $\ebra{\lambda_\star, \rho_\star} = \ebra{\lambda_\star, \rho\ebra{\lambda_\star, m_\star}}$ with $m_\star < -2/27$, where $\rho_\star$ is concentrated on the left well, at which the system exhibits a discontinuous phase transition governed by a subcritical saddle-node bifurcation. Moreover, $\ebra{\lambda_\star, m_\star}$ is a solution to the fixed point equation consisting of the moment bifurcation equation (\ref{eq.16}) and Dawson's criterion (\ref{eq.102}):
			\begin{equation}\label{eq.91.1.9}
				\begin{cases}
					& m = \int_{\mathbb{R}} x \cdot \rho\ebra{\lambda, m} \dif x, \\
					& \lambda = -\left(\int_{\mathbb{R}} x^2 \cdot \rho\ebra{\lambda, m} \dif x - \left(\int_{\mathbb{R}} x \cdot \rho\ebra{\lambda, m} \dif x\right)^2\right)^{-1}.
				\end{cases}
			\end{equation}
			The saddle-node structure implies that, for $\lambda \in \ebra{\lambda_\star, \lambda_\star + \epsilon}$ with $\epsilon$ small, there are no invariant measures close to $\rho_\star$. For $\lambda \in \ebra{\lambda_\star - \epsilon, \lambda_\star}$, there are two distinct invariant measures $\rho\ebra{\lambda, m_-^{(1)}}$ and $\rho\ebra{\lambda, m_-^{(2)}}$ close to $\rho_\star$.
			\item Globally, the system has at most three distinct invariant measures. For every $\lambda < \lambda_\star$, there exists an invariant measure concentrated on the right well $\rho\ebra{\lambda, m_+}$ ($m_+ > 0$) with
			\begin{equation}\label{eq.var.1}
				\int_{\mathbb{R}} x^2 \cdot \rho\ebra{\lambda, m_+} \dif x - \left(\int_{\mathbb{R}} x \cdot \rho\ebra{\lambda, m_+} \dif x\right)^2< \frac{\sigma^2}{2}.
			\end{equation}
			For every $\lambda < \lambda_\star$, there are exactly two invariant measures $\rho\ebra{\lambda, m_-^{(1)}}$ and $\rho\ebra{\lambda, m_-^{(2)}}$, $m_-^{(1)} < m_-^{(2)}$ concentrated on the left well with 
			\begin{equation}\label{eq.var.2}
				\text{Var}(\rho\ebra{\lambda, m_-^{(1)}}) < \frac{\sigma^2}{2}, \qquad \text{Var}\ebra{\rho\ebra{\lambda, m_-^{(2)}}} > \frac{\sigma^2}{2}. 
			\end{equation}
		\end{enumerate}
	\end{example}
	
	\begin{remark}\label{rem.stableunstableandvariance}
		By \cite[Theorem 2.1 and Example 2.2]{Zha25a}, the invariant measure is locally stable if $\operatorname{Var}(\rho) < \sigma^2/2\kappa$, and unstable if $\operatorname{Var}(\rho) > \sigma^2/2\kappa$.
	\end{remark}
	
	\begin{example}\label{ex.2}
		(A symmetric four-well model undergoing multiple continuous phase transitions and admitting at most $9$ invariant measures in $\mathbb{R}^2$ (see Figure \ref{Fig.2})) Let $n=2$ and consider the symmetric four-well potential
		\begin{equation*}\label{91.1.3.1}
			\begin{aligned}
				& V\ebra{x_1, x_2} = \ebra{V_1 \oplus V_2} \ebra{x_1, x_2} := V_1\ebra{x_1} + V_2 \ebra{x_2},\\
				& V_1\ebra{x_1} = x_1^4 - 2 x_1^2, \qquad V_2\ebra{x_2} = \frac{1}{4}x_2^4 - \frac{1}{2}x_2^2.
			\end{aligned}
		\end{equation*}
		Let the interaction be $W\ebra{x_1, x_2} = \ebra{x_1^2 + x_2^2}/2$, with interaction strength $\kappa > 0$ given.
		
		We define candidates for trivial invariant measures
		\begin{equation*}\label{eq.91.1.4}
			\begin{aligned}
				& \rho_0\ebra{\lambda, \textbf{m}}\ebra{x_1, x_2} := \ebra{\rho_0^{(1)}\otimes \rho_0^{(2)}}\ebra{x_1, x_2} := \rho_0^{(1)}\ebra{\lambda, m_1}\ebra{x_1} \cdot \rho_0^{(2)}\ebra{\lambda, m_2}\ebra{x_2}, \\
				& \rho_0^{(i)}\ebra{\lambda, m_i}\ebra{x_i} := \frac{1}{\tilde{Z} \ebra{\lambda, m_i}} \cdot \exp\ebra{\lambda \cdot \ebra{V_i\kappa^{-1}  + \frac{1}{2}x_i^2 - m_i x_i}}, \quad i \in \{1, 2\},
			\end{aligned}
		\end{equation*}
		where $\textbf{m} = \ebra{m_1, m_2}$ and $\tilde{Z}$ is the normalization constant. Let $\sigma_\star^{(1)}$ and $\sigma_\star^{(2)}$ be phase-transition temperatures satisfying the one-dimensional Dawson criterion, respectively,
		\begin{equation*}\label{eq.91.1.5}
			\frac{\ebra{\sigma_\star^{(i)}}^2}{2\kappa} = - \ebra{\lambda_\star^{(i)}}^{-1} = \int_{\mathbb{R}} x_i^2 \cdot \rho_0^{(i)}\ebra{x_i} \dif x_i,
		\end{equation*}
		and let $\pm m^{(i)}\ebra{\lambda} = \pm m^{(i)}\ebra{-2\kappa/\sigma^2}$ be the non-trivial solution to the moment bifurcation equation (\ref{eq.16}) with respect to $\rho_0^{(i)}$ when $\sigma < \sigma_\star^{(i)}$. Then $\sigma_\star^{(1)} > \sigma_\star^{(2)}$, implying that a deeper potential well possesses a higher critical phase transition temperature, and
		\begin{enumerate}
			\item Locally, when the temperature drops from high to low, the system undergoes a first-time phase transition at $\ebra{\lambda_\star^{(1)}, \rho_0\ebra{\lambda_\star^{(1)}, \ebra{0, 0}}}$, causing symmetry breaking in the $x_1$ direction. The system undergoes three second-time phase transitions simultaneously at $\ebra{\lambda_\star^{(2)}, \rho_0\ebra{\lambda_\star^{(2)}, \ebra{0, 0}}}$,  $\ebra{\lambda_\star^{(2)}, \rho_0\ebra{\lambda_\star^{(2)}, \ebra{\pm m^{(1)} \ebra{\lambda_\star^{(2)}}, 0}}}$ respectively, causing symmetry breaking in the $x_2$ direction.
			
			All these phase transitions are continuous and are governed by pitchfork bifurcations determined by the generalized Dawson criteria in $\mathbb{R}^2$
			\begin{equation*}\label{eq.91.1.6}
				\begin{aligned}
					& -\lambda^{-1} \in \text{Spec}
					\begin{pmatrix}
						\int_{\mathbb{R}} x_1^2 \cdot \rho_0^{(1)}\ebra{\lambda, 0} \dif x_1 & 0\\
						0 & \int_{\mathbb{R}} x_2^2 \cdot \rho_0^{(2)}\ebra{\lambda, 0} \dif x_2
					\end{pmatrix}, \\
					& -\lambda^{-1} \in \text{Spec}
					\begin{pmatrix}
						\int_{\mathbb{R}} x_1^2 \cdot \rho_0^{(1)}\ebra{\lambda, m^{(1)}\ebra{\lambda}} \dif x_1 - & \multirow{2}{*}{$0$} \\
						\left(\int_\mathbb{R} x_1 \cdot \rho_0^{(1)}\ebra{\lambda, m^{(1)}\ebra{\lambda}}\dif x_1\right)^2 & \\
						0 & \int_{\mathbb{R}} x_2^2 \cdot \rho_0^{(2)}\ebra{\lambda, 0} \dif x_2
					\end{pmatrix},
				\end{aligned}
			\end{equation*}
			where ``Spec" denotes the set of eigenvalues of the covariance matrix.
			\item Globally, the non-trivial solution curve emerging after the second-time phase transitions becomes unbounded in $\mathbb{R} \times L^1\ebra{\mathbb{R}^2, 1+ |x|^2}$. For any $\sigma \geq \sigma_\star^{(1)}$, the unique invariant measure is $\rho_0\ebra{\lambda, \ebra{0, 0}}$. For any $\sigma_\star^{(1)} > \sigma \geq \sigma_\star^{(2)}$, there exist three invariant measures determined by $\textbf{m} = \ebra{0, 0}, \ebra{\pm m^{(1)}\ebra{\lambda}, 0}$ respectively. For any $\sigma_\star^{(2)} > \sigma > 0$, there exist nine invariant measures determined by $\textbf{m} = \ebra{0, 0}, \ebra{\pm m^{(1)}\ebra{\lambda}, 0}, \ebra{0, \pm m^{(2)}\ebra{\lambda}}, $ and $\ebra{\pm m^{(1)}\ebra{\lambda}, \pm m^{(2)}\ebra{\lambda}}$.
		\end{enumerate}
	\end{example}
	A description of multiple phase transitions in Example \ref{ex.2} is given in Figure \ref{Fig.2}.
	\begin{figure}[htb]
		\centering
		\includegraphics[height=8cm, width=14cm]{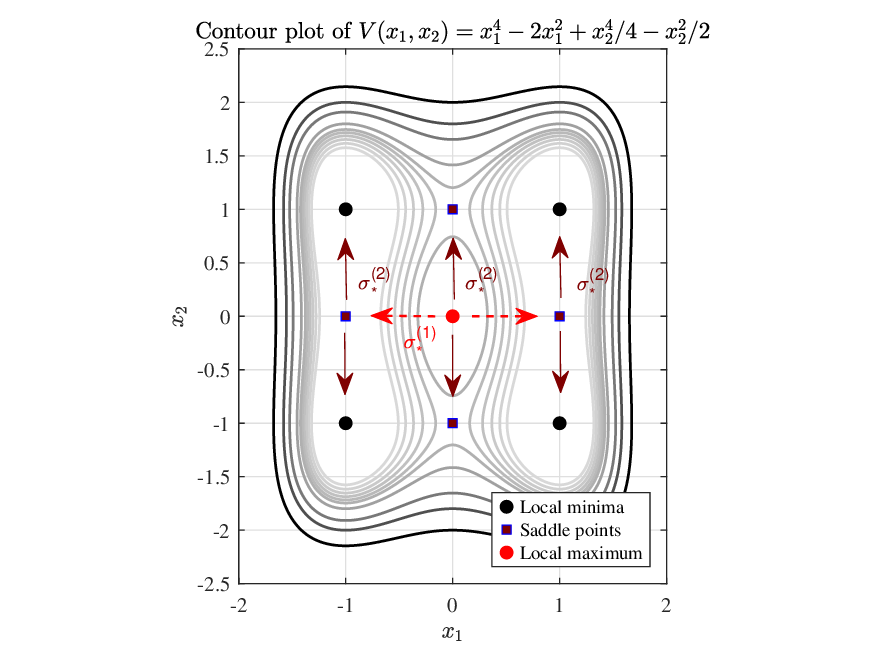}
		\caption{Multiple phase transitions for $V\ebra{x_1, x_2} = x_1^4 - 2x_1^2 + x_2^4/4 - x_2^2/2$. Red dashed arrows denote the first phase transition at temperature $\sigma_\star^{(1)}$; brown arrows denote the second phase transitions at temperature $\sigma_\star^{(2)}$.}\label{Fig.2}
	\end{figure}
	
	\begin{remark}\label{rem.9}
		To ensure that the null space of the linear part of the bifurcation equation is $1$-dimensional, $V_1 \neq V_2$ is assumed. It is numerically observed that if $V_1 = V_2$, all phase transitions occur simultaneously. It is also observed that the deeper the well, the higher the phase transition temperature. In Example \ref{ex.2}, the potential is chosen such that invariant measures are product measures to avoid redundant calculations. The more complicated case of non-product high-dimensional phase transitions is left for further work.
	\end{remark}
	
	\begin{example}\label{ex.4}
		(An asymmetric four-well model with discontinuous phase transitions in $\mathbb{R}^2$) Let $n=2$. Let the potential be symmetric in the $x_1$ direction and asymmetric in the $x_2$ direction
		\begin{equation*}\label{eq.ex.4.1}
			V\ebra{x_1, x_2} = \frac{1}{4}x_1^4 - \frac{1}{2}x_1^2 + \frac{1}{4}x_2^4 - \frac{1}{3}x_2^3 - \frac{1}{2}x_2^2.
		\end{equation*}
		Consider the interaction $W\ebra{x_1, x_2} = \ebra{x_1^2 + x_2^2}/2$ with interaction strength $\kappa = 1$.
		
		Let $\sigma_\star^{(1)}$ and $\sigma_\star^{(2)}$ be critical temperatures in Examples \ref{ex.1} and \ref{ex.3}, respectively. The result is similar to that in the preceding example: when $\sigma > \sigma_\star^{(1)}$, setting $\lambda = -2\kappa/\sigma^2$, the product measure of $\rho_0^{(1)}\ebra{\lambda, 0}\ebra{x_1}$ (the symmetric invariant measure) and $\rho^{(2)}\ebra{\lambda, m_+\ebra{\lambda}}\ebra{x_2}$ (the invariant measure concentrated on the deeper well) is an invariant measure. A continuous phase transition in the $x_1$ direction occurs at $\sigma_\star^{(1)}$. When $\sigma_\star^{(2)} < \sigma < \sigma_\star^{(1)}$, $\rho_{\pm, 0}^{(1)}\ebra{\lambda} \otimes \rho^{(2)}\ebra{\lambda, m_+\ebra{\lambda}}$ are the invariant measures, where $\rho_{-}^{(1)}, \rho_{0}^{(1)}, \rho_{+}^{(1)}$ denote the three invariant measures in Example \ref{ex.1}. At the temperature $\sigma_\star^{(2)}$, the system undergoes a discontinuous phase transition along the $x_2$ direction. If $0 < \sigma < \sigma_{\star}^{(2)}$, there exist at most $9$ invariant measures given by products of the invariant measures shown in Examples \ref{ex.1} and \ref{ex.3}.
	\end{example}
	
	We also investigate a non-quadratic interaction example exhibiting no phase transitions.
	\begin{example}\label{ex.5}
		(Non-existence of phase transitions under cubic interaction) Let $n=1$, and consider the interaction term $W = x^3/3$ with strength $\kappa > 0$ given. Assume that $V$ satisfies one-sided Lipschitz, polynomial growth assumptions, and exhibits growth that is faster than cubic order. Then there is no local phase transition. Indeed, the abstract bifurcation theorems in Appendix \ref{bifur} state that a necessary condition for the emergence of a phase transition point $(\sigma_\star, \rho_\star)$ is
		\begin{equation}\label{eq.eq.1.7}
			-\frac{\sigma_\star^2}{2\kappa} \in \text{spec}(K), \qquad K\omega = \rho_\star \cdot \ebra{W * \omega - \int_{\mathbb{R}}W * \omega \cdot \rho_\star \dif x}.
		\end{equation}
		$\mathcal{T}$ can be proved to be a completely continuous operator on $L^1\ebra{\mathbb{R}, \ebra{1 + |x|^2}^{\frac{3}{2}}}$ by an argument similar to that in Lemma \ref{thm.1}. Thus $K$ is compact. However, using the Hermite-basis representation method of Lemma \ref{lemma.2}, one finds that
		\begin{equation*}\label{eq.eq.1.8}
			\text{spec}(K) = \{0, \sqrt{\Delta_2\ebra{\rho_\star}}\textbf{i}, -\sqrt{\Delta_2\ebra{\rho_\star}}\textbf{i}\}, \qquad \Delta_2\ebra{\rho_\star} := \det \ebra{m_{i+j}}_{0\leq i, j \leq 2} \geq 0.
		\end{equation*}
		Thus $K$ has no real spectral value other than $0$. Thus (\ref{eq.eq.1.7}) is violated, and there is no phase transition.
	\end{example}
	\begin{remark}\label{rem.ex.cubic_inter}
		The spectrum of polynomial interactions may be connected with the Hankel determinant. When $W=x^2/2$, we have shown that the spectrum is related to $\Delta_1 = m_2 - m_1^2$.
	\end{remark}
	
	Finally, Theorem \ref{thm.7} is used to investigate phase transitions under Gaussian attractive interactions. This class of short-range interaction is commonly employed to model the behavior of soft spheres in physics \cite{Sti76, Mladek06}. Unlike the case of long-range quadratic interactions, the mathematical setting here is distinguished by the fact that the linear operator $K$ is not a finite-rank operator.
	\begin{example}\label{ex.6}
		(A Gaussian interaction example with numerical evidence) Let $n=1$. Suppose $V$ is symmetric and satisfies the one-sided Lipschitz, polynomial growth, and weak dissipativity conditions. Let $W(x) = -\exp(-x^2)$ and set $X = L^1(\mathbb{R})$. Then
		
		(I) $\mathcal{T}$ is completely continuous, and for each $\lambda = -2\kappa/\sigma^2 < 0$, there exists a symmetric trivial invariant measures $\rho_0(\lambda)$ possessing a finite second moment.
		
		(II) The necessary condition for the occurrence of phase transitions is described by the following system of equations
		\begin{equation}\label{eq.ex.6.1}
			\mathcal{T}\rho_0\ebra{\lambda_\star} = \rho_0\ebra{\lambda_\star}, \qquad-\frac{\sigma_\star^2}{2\kappa} = \lambda_\star^{-1} \in \text{Spec} (K_{\lambda_\star}),
		\end{equation}
		where, for any $\omega \in L^1\ebra{\mathbb{R}}$,
		$K_{\lambda_\star}\omega = \rho_0 \ebra{\lambda_\star} \cdot \ebra{W *\omega - \int_{\mathbb{R}}W * \omega \cdot \rho_0 \ebra{\lambda_\star} \dif x}. $
		
		In particular, the necessary condition for the occurrence of the first-time phase transition from the family of trivial invariant measures $\{\rho_0(\lambda)\}$ is given by (where $\theta_{\min}(K)$ denotes the smallest eigenvalue of the compact operator $K$)
		$$-\frac{\sigma_\star^2}{2\kappa} = \theta_{\text{min}} (K_{\lambda_\star}).$$
	\end{example}
	
	\begin{remark}
		Since $K$ is not of finite rank when $W(x) = -\exp(-x^2)$, its spectrum is difficult to estimate analytically \cite{LLPW25}. Consequently, we employ a numerical method to illustrate the occurrence of a pitchfork bifurcation associated with the first-time phase transition from the trivial invariant measures. The corresponding source code is available at \url{https://github.com/JunlangHu/MVPTVBT_codes}.
	\end{remark}
	
	The main idea of the numerical method is as follows. Since $K$ maps $L^1(\mathbb{R})$ into $\text{range}(K) \subset L^2\big(\mathbb{R}, \rho_0(\lambda_\star)^{-2} \exp(-x^2)\big) \subset L^1(\mathbb{R})$, a property guaranteed by the dissipativity of $V$ and the boundedness of $W$, it suffices to restrict our spectral analysis to the Hilbert space $L^2\big(\mathbb{R}, \rho_0(\lambda_\star)^{-2} \exp(-x^2)\big)$. This space is spanned by the transformed Hermite basis $\psi_k = \rho_0\ebra{\lambda_\star} \cdot \tilde{h}_k$ (cf. Subsection \ref{aux_lems}). Consequently, equation (\ref{eq.ex.6.1}) is solved numerically using the approximation scheme below
	\begin{equation*}\label{eq.ex.6.2}
		\begin{aligned}
			\rho_{l+1} \gets \mathcal{T}\rho_l, \qquad -\frac{\sigma_\star^2}{2\kappa} = \text{Spec} (P_LK_{\lambda_\star}),
		\end{aligned}
	\end{equation*}
	where the fixed-point iteration is carried out within the subspace of symmetric functions on $\mathbb{R}$, and $P_L$ is the orthogonal projection onto $\text{span}\{\psi_1, \dots, \psi_L\}$ (cf. the convergence analysis in \cite[Theorem 2]{Osb75}). The $(i, j)$-th element of the finite-rank operator $P_L K_{\lambda_\star}$ is given by $\eangle{K \psi_i, \psi_j}_{L^2\ebra{\mathbb{R}, \rho_0\ebra{\lambda_\star}^{-2}\exp(-x^2)}}$.
	
	For $V(x) = x^4/4 - x^2/2$ and $\kappa = 1$, choosing $L$ sufficiently large gives
	$$\lambda_\star \approx -3.173013549816, \qquad \sigma_\star \approx 0.793924189756. $$
	Since $P_L K_{\lambda_\star}$ has a unique simple eigenvalue converging to the infimum of the spectrum of $K_{\lambda_\star}$, with the gap between the lowest and second-lowest eigenvalues being roughly $0.186$, we deduce from \cite{Osb75} that $\theta_{\min}(K_{\lambda_\star})$ is likewise simple. By calculating the left and right eigenvectors of the approximating operator $P_L K_{\lambda_\star}$, the critical projection onto the critical eigenspace can be determined. The numerical criterion for the pitchfork bifurcation is formulated as
	\begin{equation*}\label{eq.ex.6.3}
		\begin{aligned}
			& P \partial_{\lambda, \delta \rho}^2 \tilde{F}\ebra{\lambda_\star, 0}[\tilde{v}_\star] \approx 0.32825058 \neq 0, \qquad P \partial_{\delta\rho, \delta\rho}^2 \tilde{F}\ebra{\lambda_\star, 0} [\tilde{v}_\star, \tilde{v}_\star] \approx -3.4 \cdot 10^{-18} \approx 0, \\
			& P\ebra{1/3\cdot \partial_{\delta \rho, \delta \rho, \delta \rho}^3 \tilde{F}\ebra{\lambda_\star, 0}[\tilde{v}_\star, \tilde{v}_\star, \tilde{v}_\star] + \partial_{\delta\rho, \delta\rho}^2\tilde{F}\ebra{\lambda_\star, 0}[\tilde{v}_\star, 2 \partial_s w\ebra{\lambda_\star, 0}]} \approx 0.45466642 \neq 0.
		\end{aligned}
	\end{equation*}
	Therefore, we numerically observe the occurrence of a pitchfork bifurcation at the smallest eigenvalue of $K_{\lambda_\star}$.
	
	\section{Proof of main theorems}\label{sec.proof_thms}
	The first two subsections in this section provide supporting lemmas, including general results for arbitrary interactions (Lemma \ref{lemma.1}) and more precise results for quadratic interactions (Lemmas \ref{thm.1}, \ref{lemma.2}, \ref{lemma.2.1}). The last subsection provides proofs of the main results based on these auxiliary lemmas.
	
	\subsection{Auxiliary lemmas}\label{aux_lems}
	
	In this subsection, we derive expressions for derivatives of the Gibbs map up to third order at $\rho_0$. Here, $\rho_0$ is an invariant measure and a fixed point of $\mathcal{T}$ around which we analyze the bifurcation phenomenon.
	
	Lemma \ref{lemma.1} follows from direct calculations. Therefore, we omit its proof.
	
	\begin{lemma}[Derivative formulas for the Gibbs map]\label{lemma.1}
		Assume that (H1)--(H3) hold, and $\mathcal{T}\ebra{\lambda, \cdot}$ has Fr\'{e}chet derivatives up to third order on a Banach space $X$. Moreover, for the invariant measure $\rho_0$ and any $\omega, \omega_1, \omega_2, \omega_3 \in X$, we have
		\begin{align}
				\partial_{\rho}\mathcal{T}\ebra{\lambda, \rho_0} [\omega] &= \lambda K\omega := \lambda \cdot \rho_0 \cdot \left(W * \omega - \int_{\mathbb{R}^n} W * \omega \cdot \rho_0 \dif x\right), \label{eq.58.line1}\\
				\ebra{\partial_\lambda \mathcal{T}} \ebra{\lambda, \rho_0} & = \rho_0 \cdot \left(\ebra{V/\kappa + W * \rho_0} - \int_{\mathbb{R}^n} \ebra{V/\kappa + W * \rho_0} \cdot \rho_0 \dif x\right), \\
				\ebra{\partial_{\lambda, \rho}^2 \mathcal{T}}\ebra{\lambda, \rho_0}[\omega] & = K \omega  + \lambda \cdot \rho_0 \cdot \left(\tilde{K}\omega \cdot \tilde{L} - \int_{\mathbb{R}^n} \tilde{K} \omega \cdot \tilde{L} \cdot \rho_0 \dif x\right), \\
				\partial_{\rho, \rho}^2 \mathcal{T} \ebra{\lambda, \rho_0}[\omega_1, \omega_2] & = \lambda^2 \cdot\rho_0 \cdot \left(\tilde{K} \omega_1 \cdot \tilde{K}\omega_2 - \int_{\mathbb{R}^n} \tilde{K}\omega_1 \cdot \tilde{K}\omega_2 \cdot \rho_0 \dif x\right), \label{eq.4.4.1}\\
				\partial_{\rho, \rho, \rho}^3\mathcal{T}\ebra{\lambda, \rho_0} [\omega_1, \omega_2, \omega_3] & = \lambda^3 \cdot \rho_0 \cdot \left(\tilde{K}\omega_1 \cdot \tilde{K}\omega_2 \cdot \tilde{K}\omega_3 - \int_{\mathbb{R}^n} \tilde{K}\omega_1 \cdot \tilde{K}\omega_2 \cdot  \tilde{K}\omega_3 \cdot \rho_0 \dif x\right. \notag\\
				& \quad - \tilde{K}\omega_1 \cdot \ebra{\int_{\mathbb{R}^n}\tilde{K}\omega_2 \cdot \tilde{K}\omega_3 \cdot \rho_0 \dif x} - \tilde{K}\omega_2 \cdot \ebra{\int_{\mathbb{R}^n}\tilde{K}\omega_1 \cdot \tilde{K}\omega_3 \cdot \rho_0 \dif x} \notag \\
				& \quad \left. - \tilde{K}\omega_3 \cdot \ebra{\int_{\mathbb{R}^n}\tilde{K}\omega_1 \cdot \tilde{K}\omega_2 \cdot \rho_0 \dif x}\right),
				\label{eq.58.line-1}
		\end{align}
		where $\lambda = -2\kappa/\sigma^2$, $\tilde{K}\omega := K\omega/\rho_0$, and
		\begin{equation}\label{eq.59}
			\begin{aligned}
				& \tilde{L} := \ebra{V/\kappa + W * \rho_0} - \int_{\mathbb{R}^n} \ebra{V/\kappa + W * \rho_0} \cdot \rho_0 \dif x = \ebra{\partial_\lambda\mathcal{T}}\ebra{\lambda, \rho_0} / \rho_0.
			\end{aligned}
		\end{equation}
	\end{lemma}

	Now assume that (H1)--(H4) hold. Choose $X = L^1\ebra{\mathbb{R}^n, 1 + |x|^2}$. We state three auxiliary lemmas concerning compactness, spectral, and projection properties of the linearized operators.
	
	In Lemma \ref{thm.1}, we establish compactness and continuity results for the Gibbs map. On a compact manifold, the compactness of the Gibbs map automatically holds (\cite[Theorem 1.2]{CGPS20}). On the whole space, by contrast, the confining property of the potentials is needed to compensate for the lack of compactness. For $\sigma, \kappa > 0$ fixed, we write $\mathcal{T}\rho := \mathcal{T}\ebra{-2\kappa/\sigma^2, \rho}$ for brevity when no confusion can arise.
	\begin{lemma}\label{thm.1}
		Assume that (H1)--(H4) hold. Then $\mathcal{T}$ is continuous and compact on $L^1\ebra{\mathbb{R}^n, 1 + |x|^2}$.
	\end{lemma}
	\begin{proof}\label{proof.2}
		We first show that for any sequence $\rho_j \to \rho$ in $L^1\ebra{\mathbb{R}^n, 1 + |x|^2}$, we have $\mathcal{T} \rho_j \to \mathcal{T} \rho$ in the same norm. Note that
		\begin{equation*}\label{eq.18.1}
			\int_{\mathbb{R}^n} |\rho_j - \rho| \cdot \ebra{1 + |x|^2}^{1/2} \dif x \leq \int_{\mathbb{R}^n} |\rho_j - \rho| \cdot \ebra{1+|x|^2}\dif x.
		\end{equation*}
		Hence, $\lim_{j \to \infty}m_0\ebra{\rho_j} = m_0\ebra{\rho} \text{ and } \lim_{j \to \infty} m_1\ebra{\rho_j} = m_1\ebra{\rho}$, which imply that for any $x \in \mathbb{R}^n$,
		\begin{equation}\label{eq.18}
			\begin{aligned}
				& \lim_{j \to \infty}\exp \ebra{\ebra{-V - \kappa \cdot \ebra{m_0\ebra{\rho_j} \cdot \frac{1}{2}|x|^2 - \eangle{m_1 \ebra{\rho_j}, x}}}\cdot \frac{2}{\sigma^2}}  \\
				&  = \exp \ebra{\ebra{-V - \kappa \cdot \ebra{m_0 \ebra{\rho}\cdot \frac{1}{2}|x|^2 - \eangle{m_1 \ebra{\rho}, x}}}\cdot \frac{2}{\sigma^2}}. \\
			\end{aligned}
		\end{equation}
		By Lemma \ref{lem.1} and the dominated convergence theorem, we have
		\begin{equation}\label{eq.18.1.1}
			\begin{aligned}
				\lim_{j \to \infty} \tilde{Z} \ebra{\rho_j, \kappa, \sigma}  = \tilde{Z}\ebra{\rho, \kappa, \sigma},
			\end{aligned}
		\end{equation}
		where $$
		\tilde{Z}\ebra{\rho, \kappa, \sigma} = \int_{\mathbb{R}^n} \exp \ebra{\ebra{-V - \kappa \cdot \ebra{m_0\ebra{\rho} \cdot |x|^2/2 - \eangle{m_1 \ebra{\rho}, x}}\cdot 2/\sigma^2}} \dif x.
		$$
		Combining (\ref{eq.18}) and (\ref{eq.18.1.1}), we obtain, for any $x \in \mathbb{R}^n$, $\mathcal{T} \rho_j \ebra{x}\to \mathcal{T} \rho\ebra{x}$. Both $\{\mathcal{T}\rho_j\}_{j \in \mathbb{N}}$ and $\{|\cdot|^2\mathcal{T}\rho_j\}_{j \in \mathbb{N}}$ decay exponentially as $|x|\to\infty$ by Lemma \ref{lem.1} and (H2). Consequently, $\mathcal{T}\rho_j \to \mathcal{T}\rho$ in $L^1\ebra{\mathbb{R}^n, 1 + |x|^2}$, again by the dominated convergence theorem.
		
		Next we show that $\mathcal{T}$ is compact on $L^1\ebra{\mathbb{R}^n, 1 + |x|^2}$ by verifying conditions in the Fr\'{e}chet-Kolmogorov-Riesz compactness theorem (see Theorem \ref{thm.2} for details). Take any sequence $\{\rho_l\}_{l \in \mathbb{N}}$ bounded in $L^1\ebra{\mathbb{R}^n, 1 + |x|^2}$. For any $x \in \mathbb{R}^n$, we have
		\begin{align}
				W * \rho_l\ebra{x} & \geq m_0\ebra{\rho_l}/2 \cdot |x|^2 - |m_1 \ebra{\rho_l}| \cdot |x| + m_2\ebra{\rho_l}/2 \notag\\
				& \geq -\ebra{\sup_{l \in \mathbb{N}} \norm{\rho_l}_{L^1\ebra{\mathbb{R}^n}} + 1}/2 \cdot |x|^2 -  C(\sup_{l \in \mathbb{N}} m_2\ebra{|\rho_l|}/2, \sup_l \norm{\rho_l}_{L^1\ebra{\mathbb{R}^n}}), \label{eq.25.line1}\\
				W * \rho_l\ebra{x} & \leq \ebra{\sup_{l \in \mathbb{N}} \norm{\rho_l}_{L^1\ebra{\mathbb{R}^n}} + 1} /2 \cdot |x|^2 +  \sup_{l \in \mathbb{N}} m_2\ebra{|\rho_l|}/2. \label{eq.25.line-1}
		\end{align}
		Then it follows from (H2), Lemma \ref{lem.1}, and (\ref{eq.25.line1})-(\ref{eq.25.line-1}) that $\{\mathcal{T}\rho_l\}_{l \in \mathbb{N}}$ and $\{|\cdot|^2 \mathcal{T}\rho_l\}_{l \in \mathbb{N}}$ are bounded in the $L^1$ norm.
		Note that
		\begin{align}
				\sup_{l \in \mathbb{N}} \int_{\mathbb{R}^n \setminus B_R\ebra{0}} |\mathcal{T}\rho_l| \dif x & = \sup_{l \in \mathbb{N}} \int_{\mathbb{R}^n \setminus B_R\ebra{0}} |x|^{-1} \cdot |x| \cdot|\mathcal{T}\rho_l| \dif x \notag\\
				& \leq R^{-1} \cdot \sup_{l \in \mathbb{N}} m_2\ebra{\mathcal{T} \rho_l}^{1/2} \to 0 \quad \ebra{R \to \infty}. \label{eq.26}
		\end{align}
		Hence, $\{\mathcal{T}\rho_l\}_{l \in \mathbb{N}}$ is uniformly integrable. By Lemma \ref{lem.1} and (\ref{eq.25.line1}), (\ref{eq.25.line-1}), $\{|\cdot|^q \mathcal{T}\rho_l\}_{l \in \mathbb{N}}$ is also $L^1$-bounded uniformly in $l$ for any $q \geq 3$. Since the boundedness of $q$-order moment implies uniform integrability of lower order, similar to (\ref{eq.26}), we obtain the uniform integrability of $\{|\cdot|^2 \mathcal{T}\rho_l\}_{l \in \mathbb{N}}$. We have the following estimate for the gradient
		\begin{equation*}\label{eq.27}
			\begin{aligned}
				\nabla\ebra{\ebra{1+|\cdot|^2}\mathcal{T}\rho_l} \ebra{x} & = 2x \cdot \mathcal{T}\rho_l\ebra{x} + \ebra{1 + |x|^2} \nabla \mathcal{T}\rho_l\ebra{x} \\
				& \leq \left(2|x| + \frac{2}{\sigma^2}\cdot (1 + |x|^2) \cdot C(x)\right)\cdot \mathcal{T}\rho_l\ebra{x}
			\end{aligned}
		\end{equation*}
		with
		$$
		C(x) = |\nabla V|\ebra{x} + \kappa |x| \cdot \sup_{l\in\mathbb{N}}\norm{\rho_l}_{L^1\ebra{\mathbb{R}^n}} + \kappa\cdot \sup_{l\in \mathbb{N}}|m_1\ebra{\rho_l}|.
		$$
		By the assumptions, $V$ can be controlled by a polynomial. Thus $\norm{\nabla\ebra{\ebra{1+|\cdot|^2}\mathcal{T}\rho_l}}$ is bounded in $L^1$ norm uniformly in $l$. Observe that
		\begin{equation*}
			\begin{aligned}
				\norm{\ebra{\tau_h - id}\circ \ebra{\ebra{1+|x|^2}\mathcal{T}\rho_l}}_{L^1\ebra{\mathbb{R}^n}} & = \norm{\int_0^1 \eangle{\nabla\ebra{\ebra{1+|\cdot|^2}\mathcal{T}\rho_l} \ebra{x - th}, h} \dif t}_{L^1\ebra{\mathbb{R}^n}}\\
				& \leq \norm{\nabla\ebra{\ebra{1+|\cdot|^2}\mathcal{T}\rho_l}}_{L^1\ebra{\mathbb{R}^n}}\cdot |h|,
			\end{aligned}
		\end{equation*}
		which, along with the above estimate, implies the uniform small oscillation property.
		
		We conclude that $\mathcal{T}$ is a completely continuous operator on $L^1\ebra{\mathbb{R}^n, 1 + |x|^2}$.
	\end{proof}
	
	Let $\rho_0$ be the invariant measure under investigation. Since $\lambda K$ is the linear part of $\mathcal{T}\ebra{\lambda, \cdot}$, \cite[Lemma 3.1.12]{BT03} implies that $K$ is a compact operator. By classical operator theory, all non-zero spectra of $K$ consist of eigenvalues, i.e., $K \tilde{v} = \lambda^{-1} \tilde{v}$. Since $K$ maps $L^1(\mathbb{R}^n, 1+|x|^2)$ to functions that decay exponentially fast, the form of eigenfunctions implies that it suffices to restrict our attention to a proper Hilbert subspace of $L^1(\mathbb{R}^n, 1+|x|^2)$ that contains the elements of $\operatorname{range} K$ and is equipped with a countable orthogonal basis. We choose the subspace to be $L^2\ebra{\mathbb{R}^n, \rho_0^{-2}\exp\ebra{-|x|^2}}$. 
	
	\begin{remark}\label{explain}
		We note that $K: L^1(\mathbb{R}^n, 1+|x|^2) \to L^2\ebra{\mathbb{R}^n, \rho_0^{-2} \exp\ebra{-|x|^2}} \subset L^1\ebra{\mathbb{R}^n, 1 + |x|^2}$. On the one hand, for any $\omega \in L^1\ebra{\mathbb{R}^n, 1 + |x|^2}$, we have
		\begin{equation*}\label{eq.63}
			\begin{aligned}
				\norm{K\omega}_{L^2\ebra{\mathbb{R}^n,\ \rho_0^{-2}\exp\ebra{-|x|^2}}}^2
				& = \int_{\mathbb{R}^n} \ebra{W * \omega - \int_{\mathbb{R}^n} W * \omega \cdot \rho_0 \dif x}^2 \cdot \exp\ebra{-|x|^2} \dif x < \infty.
			\end{aligned}
		\end{equation*}
		The last inequality follows from the fact that $W * \omega$ can be bounded by a quadratic polynomial. Thus $K \ebra{L^1\ebra{\mathbb{R}^n, 1 + |x|^2}} \subset L^2\ebra{\mathbb{R}^n, \rho_0^{-2} \exp\ebra{-|x|^2}}$. On the other hand, for any $f \in L^2\ebra{\mathbb{R}^n, \rho_0^{-2} \exp\ebra{-|x|^2}}$, the hyper-dissipative condition, together with Lemma \ref{lem.1}, yields
		\begin{equation*}\label{eq.64}
			\begin{aligned}
				\infty > \norm{f}_{L^2\ebra{\mathbb{R}^n,\ \rho_0^{-2}\exp\ebra{-|x|^2}}}^2 & = \int_{\mathbb{R}^n} f^2 \cdot \exp \ebra{\ebra{V + \kappa\cdot (|x|^2/2 - \eangle{m_1\ebra{\rho_0}, x})}\cdot 4/\sigma^2 - |x|^2} \dif x \cdot \tilde{Z}\ebra{\kappa, \sigma}^2 \\
				& \geq \int_{\mathbb{R}^n} f^2 \cdot \exp\ebra{C_1 \cdot |x|^{2+\delta} + C_2} \dif x \qquad \text{with constants }C_1 > 0, \ C_2 \in \mathbb{R}.
			\end{aligned}
		\end{equation*}
		Hence, this weighted $L^2$ estimate implies $L^2\ebra{\mathbb{R}^n, \rho_0^{-2}\exp\ebra{-|x|^2}} \subset L^1\ebra{\mathbb{R}^n, 1 + |x|^2}$.
	\end{remark}
	
	We aim to express the compact operator $K$ in an explicit matrix form, thereby converting the spectral problem into a computationally tractable one. To achieve this, we need to explicitly construct an orthogonal basis for $L^2\ebra{\mathbb{R}^n, \rho_0^{-2}\exp\ebra{-|x|^2}}$. Note that the Fourier basis on $\mathbb{T}^d$ used in \cite{CGPS20} is no longer suitable for the whole space. It is well known that $\{\tilde{H}_\alpha\}_{\alpha \in \mathbb{N}^n}$, the normalized Hermite orthogonal polynomials over $\mathbb{R}^n$, forms a complete orthogonal system on $L^2\ebra{\mathbb{R}^n, \exp\ebra{-|x|^2}}$. To keep the presentation self-contained, we provide a brief introduction to these polynomials in Appendix \ref{appendix.D}. Consider the following transformation
	\begin{equation*}\label{eq.68}
		\iota: f \mapsto \rho_0 \cdot f, \quad
		L^2\ebra{\mathbb{R}^n, \exp\ebra{-|x|^2}} \to L^2\ebra{\mathbb{R}^n, \rho_0^{-2}\exp\ebra{-|x|^2}}.
	\end{equation*}
	When $\rho_0 \ebra{x} > 0$ for any $x \in \mathbb{R}^n$, the operator $\iota$ is one-to-one and onto. Moreover, from the equality
	\begin{equation*}\label{eq.69}
		\norm{\iota f}_{L^2\ebra{\mathbb{R}^n,\ \rho_0^{-2}\exp\ebra{-|x|^2}}} = \norm{f}_{L^2\ebra{\mathbb{R}^n, \ \exp\ebra{-|x|^2}}},
	\end{equation*}
	we know that $\iota$ is an isometric isomorphism between the two spaces. By applying the transformation $\iota$, \hl{we obtain the following normalized orthogonal basis for the weighted space $L^2\ebra{\mathbb{R}^n, \rho_0^{-2}\exp\ebra{-|x|^2}}$}: $\{\psi_k\}_{k \in \mathbb{N}} := \{\rho_0\cdot\tilde{h}_k\}_{k \in \mathbb{N}}$ when $n = 1$, and $\{\Psi_\alpha\}_{\alpha \in \mathbb{N}^n} := \{\rho_0 \cdot \tilde{H}_\alpha\}_{\alpha \in \mathbb{N}^n}$ in higher dimensions. Their explicit expressions when $n=1$ are
	\begin{equation}\label{eq.70}
		\begin{aligned}
			& \psi_0\ebra{x} = \rho_0 \cdot \pi^{-1/4}, \quad \psi_1\ebra{x} = \rho_0 \cdot \pi^{-1/4}\cdot \ebra{\sqrt{2}x}, \\
			& \psi_2\ebra{x} = \rho_0 \cdot \pi^{-1/4} \cdot \ebra{\sqrt{2}x^2 - \sqrt{2}/2}, \ \dots
		\end{aligned}
	\end{equation}
	and for general dimension $n$,
	\begin{equation}\label{eq.71}
		\begin{aligned}
			& \Psi_0 = \rho_0 \cdot \pi^{-n/4}, \qquad\qquad\qquad\ \Psi_{e_i} = \rho_0 \cdot \pi^{-n/4} \cdot \ebra{\sqrt{2} x_i}, \\
			&\Psi_{e_i + e_j} = \rho_0\cdot\pi^{-n/4}\cdot \ebra{2x_ix_j}, \quad\Psi_{2e_i} = \rho_0 \cdot \pi^{-n/4} \cdot \ebra{\sqrt{2}x_i^2 - \sqrt{2}/2}, \ \dots
		\end{aligned}
	\end{equation}
	where $i \neq j$ and $e_i$ denotes the $i$-th unit vector in $\mathbb{R}^n$. For convenience, for $\alpha = \ebra{\alpha_1, \dots, \alpha_n} \in \mathbb{N}^n$, define the indexed moments by
	\begin{equation*}\label{eq.71.1}
		m_\alpha\ebra{\rho_0} := \int_{\mathbb{R}^n} \ebra{\prod_{i = 1}^n x_{i}^{\alpha_i}} \cdot \rho_0 \dif x_1 \cdots \dif x_n.
	\end{equation*}
	
	Consider the ``covariance matrix" generated by $\rho_0$:
	\begin{equation*}\label{eq.71.2}
		\text{Cov}\ \ebra{\rho_0} := \begin{pmatrix}
			m_{e_1+e_1} - m_{e_1}\cdot m_{e_1} & \cdots & m_{e_n + e_1} - m_{e_n} \cdot m_{e_1}\\
			\vdots & \cdots & \vdots\\
			m_{e_1 + e_n} - m_{e_1} \cdot m_{e_n} & \cdots & m_{e_n + e_n} - m_{e_n} \cdot m_{e_n}
		\end{pmatrix}_{n \times n}\ebra{\rho_0},
	\end{equation*}
	Since invariant measures for the gradient flow system (\ref{eq.1}) are absolutely continuous with respect to the Lebesgue
	measure in $\mathbb{R}^n$, $\text{Cov}\ebra{\rho_0}$ has no zero eigenvalues. Since the real symmetric matrix $\text{Cov}\ebra{\rho_0}$ is positive semidefinite, we write its eigenvalues as $0 < \theta_1 \leq \dots \leq \theta_n$ (counted with multiplicity), and the corresponding linearly independent orthogonal unit eigenvectors are $v_1, \dots, v_n$. Assume also that there are $m \leq n$ eigenspaces and a subsequence $\{n_k\}_{k \in \{0, \dots, m\}} \subset \{0, \dots, n\}$, $n_0 = 0, \ n_m = n$ such that
	\begin{equation*}\label{eq.71.3}
		E_1 \oplus E_2 \oplus \cdots \oplus E_m = \mathbb{R}^n, \qquad E_k = \text{span}\left\{v_{n_{k-1}+1}, \ \dots, \ v_{n_{k}}\right\}.
	\end{equation*}
	Finally, we define ($\ebra{v_i}_j$ denotes the $j$-th component of the vector $v_i$)
	\begin{equation*}\label{eq.73.1}
		\begin{aligned}
			& \tilde{v}_1 := \ebra{-\sqrt{2}m_{e_1}\ebra{v_1}_1 - \cdots - \sqrt{2}m_{e_n} \ebra{v_1}_n} \cdot \Psi_0 +  \ebra{v_1}_1 \cdot \Psi_{e_1} + \cdots + \ebra{v_1}_n \cdot \Psi_{e_n},\\
			& \qquad \qquad \qquad \qquad \qquad \qquad \qquad \qquad \dots, \\
			& \tilde{v}_n := \ebra{-\sqrt{2}m_{e_1}\ebra{v_n}_1 - \cdots - \sqrt{2}m_{e_n} \ebra{v_n}_n} \cdot \Psi_0 +  \ebra{v_n}_1 \cdot \Psi_{e_1} + \cdots + \ebra{v_n}_n \cdot \Psi_{e_n}.
		\end{aligned}
	\end{equation*}
	as candidates for eigenvectors of $K$. We now state the main lemma on the spectrum of the operator $K$.
	\begin{lemma}[Spectrum structure for $K$]\label{lemma.2}
		Assume that (H1) -- (H4) hold. Then 
		\begin{enumerate}
			\item The compact operator $K$ maps $L^1\ebra{\mathbb{R}^n, 1 + |x|^2}$ into $\text{span}\{\Psi_{\alpha}: |\alpha| := \sum_{i=1}^n |\alpha_i| \leq 2, \alpha \in \mathbb{N}^n\}$.
			\item All nonzero spectral values of $K$ are semisimple eigenvalues of $-\text{Cov}\ebra{\rho_0}$
			\begin{equation}\label{eq.72}
				\left\{-\theta_1, \ \dots, \ -\theta_n\right\}.
			\end{equation}
			For any one of them, the algebraic multiplicity equals the geometric multiplicity. Their corresponding finite-dimensional eigenspaces are $\tilde{E}_1, \tilde{E}_2, \dots, \tilde{E}_m$, where
			\begin{equation*}\label{eq.73}
				\tilde{E}_k = \text{span}\left\{\tilde{v}_{n_{k-1} + 1}, \ \dots, \ \tilde{v}_{n_k}\right\}.
			\end{equation*}
			\item In particular, when $n=1$, $-m_2\ebra{\rho_0} + m_1\ebra{\rho_0}^2$ and $\text{span}\{-\sqrt{2}m_1\ebra{\rho_0} \cdot \psi_0 + \psi_1\}$ are the non-zero spectrum and the corresponding simple eigenspace, respectively.
		\end{enumerate}
	\end{lemma}
	
	\begin{proof}\label{proof.10}
		For brevity, we only outline the proof here. Please check Appendix \ref{appendix.D} for a detailed proof.
		
		Under assumption (H4), for any $\omega \in L^1\ebra{\mathbb{R}^n, 1+|x|^2}$, the operator $K$ takes the following form
		\begin{equation*}\label{eq.74}
			\begin{aligned}
				K \omega  & = \rho_0 \cdot \ebra{W * \omega - \int_{\mathbb{R}^n} W * \omega \cdot \rho_0 \dif x} \\
				& = \rho_0 \cdot \ebra{\frac{1}{2}m_0\ebra{\omega} \cdot |x|^2 - \eangle{m_1\ebra{\omega}, x} - \frac{1}{2}m_0\ebra{\omega}\cdot m_2\ebra{\rho_0} + \eangle{m_1\ebra{\omega}, m_1\ebra{\rho_0}}}.
			\end{aligned}
		\end{equation*}
		Then $K(L^1\ebra{\mathbb{R}^n, 1+|x|^2}) \subset \text{span}\{\Psi_\alpha: |\alpha| \leq 2\}$, and the eigenfunctions must be products of $\rho_0$ and polynomials of degree less than or equal to two.
		
		Next, we represent $K$ as a $\ebra{n^2/2 + 3n/2 + 1}\times \ebra{n^2/2 + 3n/2 + 1}$ matrix with respect to the basis $\{\Psi_\alpha\}_{\alpha \in \mathbb{N}^n}$, and then apply a standard but lengthy linear algebra procedure to compute the eigenvalues and eigenvectors. The calculations in Appendix \ref{appendix.D} reveal the relation of $K$ to $\text{Cov}\ebra{\rho_0}$. In fact, after appropriate eliminations via row and column transformations, the characteristic polynomial is found to be
		\begin{equation*}\label{eq.90}
			\theta^{n^2/2+n/2+1} \cdot \text{det}\ebra{\theta \textbf{id}_{n \times n} + \text{Cov}\ebra{\rho_0}}.
		\end{equation*}
		This implies (\ref{eq.72}). The eigenvectors of $K$ can be obtained analogously.
	\end{proof}
	
	By now we have shown that $id - \lambda_i K$ ($\lambda_i := -\theta_i^{-1}$) is a Fredholm operator with index $0$, ascent $1$, and null space $\text{span}\{\tilde{v}_{n_{k-1}+1},\  \dots,\ \tilde{v}_{n_k}\}$ if $n_{k-1}< i \leq n_k$. The function space thus admits a direct sum decomposition
	\begin{equation*}\label{eq.90.1}
		L^1\ebra{\mathbb{R}^n, 1 + |x|^2} = \text{null } \ebra{id - \lambda_i K} \oplus \text{range } \ebra{id - \lambda_i K}.
	\end{equation*}
	
	Under assumption (H4), $K: \ L^1\ebra{\mathbb{R}^n, 1 + |x|^2} \to L^1\ebra{\mathbb{R}^n, 1 + |x|^2}$ is given by
	\begin{equation*}\label{eq.90.2}
		K\omega = \rho_0 \cdot \ebra{\sum_{k=1}^n \frac{1}{2} m_0\ebra{\omega}\cdot \ebra{
				x_k^2- m_{2 e_k}\ebra{\rho_0}} - m_{e_k}\ebra{\omega} \cdot \ebra{x_k - m_{e_k}\ebra{\rho_0}}}.
	\end{equation*}
	Then its dual operator $K^*: \ebra{L^1\ebra{\mathbb{R}^n, 1 + |x|^2}}^* \to \ebra{L^1\ebra{\mathbb{R}^n, 1 + |x|^2}}^*$ is defined by
	\begin{equation*}\label{eq.90.3}
		K^* u = \sum_{k=1}^n -\ebra{m_{e_k} \ebra{u \rho_0} - m_{e_k} \ebra{\rho_0}\cdot m_0\ebra{u \rho_0}} \cdot x_k + \frac{1}{2}\ebra{m_{2e_k}\ebra{u \rho_0} - m_{2e_k}\ebra{\rho_0} \cdot m_0\ebra{u \rho_0}},
	\end{equation*}
	where the dual space can be identified as
	\begin{equation*}\label{eq.90.4.1}
		L^\infty\ebra{\mathbb{R}^n, \ebra{(1+|x|^2}^{-1}} := \{u: \ u/(1+|x|^2) \text{ is essentially bounded}\},
	\end{equation*}
	via the dual relation $\eangle{\cdot, \cdot}$
	\begin{equation}\label{eq.90.5}
		\eangle{w, u} = \int_{\mathbb{R}^n} w \cdot u \dif x, \qquad w \in L^1\ebra{\mathbb{R}^n, 1+|x|^2}, \ u \in L^\infty\ebra{\mathbb{R}^n, \ebra{1+|x|^2}^{-1}}.
	\end{equation}
	Moreover, let
	\begin{equation*}\label{eq.90.4}
		\tilde{u}_i := \sum_{k=1}^n \ebra{v_i}_k \cdot \ebra{x_k +\xi_k^{(i)}}, \qquad \xi^{(i)}_k \cdot \ebra{v_i}_k:= \frac{\lambda_i}{2} \cdot \sum_{r=1}^n \ebra{m_{2e_k + e_r} - m_{2e_k} \cdot m_{e_r}} \ebra{\rho_0}\cdot \ebra{v_i}_r.
	\end{equation*}
	Then one can verify that $K^* \tilde{u}_i = \lambda_i^{-1}\tilde{u}_i$ with $\tilde{u}_i \in L^\infty\ebra{\mathbb{R}^n, \ebra{1 + |x|^2}^{-1}}$; that is, $\tilde{u}_i$ is an eigenvector of $K^*$ with the same eigenvalue $\lambda_i^{-1}$ as $K$. The following lemma provides an explicit expression of the spectral projection operator.
	\begin{lemma}[The spectral projection operator for $K$]\label{lemma.2.1}
		Assume that (H1) -- (H4) hold. Then
		\begin{enumerate}
			\item Consider the $i$-th characteristic value $\lambda_i = -\theta_i^{-1}$ of $K$ with eigenspace $\tilde{E}_k = \text{null}\ \ebra{id - \lambda_i K} = \text{span}\{\tilde{v}_{n_{k-1} + 1}, \ \dots, \ \tilde{v}_{n_k}\}$ ($n_{k-1} < i \leq n_k$). Then the projection operator onto $\tilde{E}_k$ is explicitly given by
			\begin{equation}\label{eq.90.6}
				P_{\lambda_i} = l_{n_{k-1}+1}\ebra{\cdot} \tilde{v}_{n_{k-1}+1} + \cdots + l_{n_k}\ebra{\cdot}\tilde{v}_{n_k},
			\end{equation}
			where the bounded linear functionals defined by
			\begin{equation*}\label{eq.90.5.1}
				l_j\ebra{\cdot} := \frac{\eangle{\cdot, \tilde{u}_j}}{\eangle{\tilde{v}_j, \tilde{u}_j}},
			\end{equation*}
			satisfy
			\begin{equation}\label{eq.90.7}
				\begin{aligned}
					&l_j\ebra{z} = 0, \qquad\qquad\qquad  \text{for } z \in \text{range} \ebra{id - \lambda_i K},\\
					&l_j\ebra{\tilde{v}_m} = \delta_{jm}, \qquad\qquad \ j, m \in \{n_{k-1}+1,\  \dots,\ n_k \}.
				\end{aligned}
			\end{equation}
			($\eangle{\cdot, \cdot}$ denotes the dual relation defined in (\ref{eq.90.5}), and $\delta_{jm}$ denotes the Kronecker symbol. )
			\item In particular, when $n=1$, $\lambda = -\ebra{m_2 - m_1^2}^{-1}\ebra{\rho_0}$, $\tilde{E} = \text{span}\{\tilde{v}\} = \text{span}\{\sqrt{2}\pi^{-1/4}\cdot (x - m_1\ebra{\rho_0})\cdot \rho_0\}$, and the projection operator is given by
			\begin{equation*}\label{eq.90.8}
				\begin{aligned}
					& P = l\ebra{\cdot} \tilde{v} = \frac{\eangle{\cdot, \tilde{u}}}{\eangle{\tilde{v}, \tilde{u}}} \cdot \tilde{v}, \qquad \tilde{u} = x - \frac{\ebra{m_3 - m_2 \cdot m_1}}{2\ebra{m_2 - m_1^2}}\ebra{\rho_0}.
				\end{aligned}
			\end{equation*}
		\end{enumerate}
	\end{lemma}
	\begin{proof}\label{proof.11}
		Once (\ref{eq.90.7}) holds, the uniqueness of the projection operator immediately implies that (\ref{eq.90.6}) holds. Hence, it suffices to verify (\ref{eq.90.7}). Note that $\lambda_i = \lambda_j$ as long as $i, j \in \{n_{k-1} + 1, \ \dots,\ n_k\}$. For any $\tilde{v}_m \in \text{span}\{\tilde{v}_j\}^{\perp\tilde{E}_k}$, by the orthogonality of the eigenvectors $v_m$ and $v_j$ of the matrix $\text{Cov}\ebra{\rho_0}$, we have
		\begin{equation*}\label{eq.90.8.1}
			\begin{aligned}
				l_j\ebra{\tilde{v}_m} &= \sqrt{2} \pi^{-n/4}/\eangle{\tilde{v}_j, \tilde{u}_j} \cdot \int_{\mathbb{R}^n} \ebra{\sum_{k=1}^n \ebra{v_m}_k \cdot \ebra{x_k - m_{e_k}}} \cdot \ebra{\sum_{r=1}^n \ebra{v_j}_r \cdot \ebra{x_r + \xi_r^{(j)}}} \cdot \rho_0\ebra{x}\dif x \\
				& = \sqrt{2}\pi^{-n/4} / \eangle{\tilde{v}_j, \tilde{u}_j} \cdot \int_{\mathbb{R}^n} \sum_{k, r = 1}^n (v_m)_k (v_j)_r \cdot \ebra{m_{e_k + e_r} - m_{e_k} \cdot m_{e_r}} = 0.
			\end{aligned}
		\end{equation*}
		
		For any $z \in \text{range}\ebra{id - \lambda_i K}$, there exists $\omega\in L^1\ebra{\mathbb{R}^n, 1+|x|^2}$ such that $\omega - \lambda_i K \omega = z$. Thus,
		\begin{align}
				l_j\ebra{z} & = \int_{\mathbb{R}^n} \left(\omega - \lambda_i \cdot \rho_0  \cdot \ebra{\sum_{r=1}^n \frac{1}{2}m_0\ebra{\omega} \cdot \ebra{
						x_r^2- m_{2 e_r}\ebra{\rho_0}}- m_{e_r}\ebra{\omega} \cdot \ebra{x_r - m_{e_r}\ebra{\rho_0}}}\right)\notag\\
				& \qquad \qquad  \cdot \ebra{\sum_{k=1}^n \ebra{v_j}_k \cdot \ebra{x_k + \xi_k^{(j)}}} \dif x \notag\\
				& = -m_0\ebra{\omega} \cdot \ebra{\lambda_i \cdot \sum_{k=1}^n\ebra{v_j}_k \xi_k^{(j)}\theta_i + \lambda_i \cdot \sum_{k, r=1}^n \frac{1}{2}\ebra{v_j}_k\cdot \ebra{m_{2e_r + e_k} - m_{2e_r} \cdot m_{e_k}}\ebra{\rho_0}} \notag\\
				& \qquad \qquad  + \sum_{k=1}^n m_{e_k}\ebra{\omega} \cdot \ebra{\ebra{v_j}_k + \lambda_i \cdot \theta_i \ebra{v_j}_k} = 0, \label{eq.90.9}
		\end{align}
		where $\text{Cov}\ebra{\rho_0} v_j = \theta_i v_j$ has been used in the first equality in (\ref{eq.90.9}).
	\end{proof}
	
	We end this subsection by recording information about the inverse of the homeomorphism $\ebra{id - \lambda_i K}|_{\text{range}\ebra{id - \lambda_i K}}$. Since $\text{null}\ \ebra{id - \lambda_i K} \subset \text{range } K \subset S := \text{span}\{\Psi_{\alpha}: |\alpha| := \sum_{i=1}^n |\alpha_i| \leq 2, \alpha \in \mathbb{N}^n\}$, $\text{dim} \ S < \infty$, the whole space admits the following direct-sum decomposition
	\begin{equation*}\label{eq.90.10}
		L^1\ebra{\mathbb{R}^n, 1 + |x|^2} = S \oplus R = \text{null}\ \ebra{id - \lambda_i K} \oplus \ebra{\text{null}\ \ebra{id - \lambda_i K}}^{\perp S} \oplus R.
	\end{equation*}
	As $S$ is closed under $K$ and $id - \lambda_i K$, the operator $id - \lambda_i K$ is also a homeomorphism restricted to $\ebra{\text{null }\ebra{id - \lambda_i K}}^{\perp S} = \text{range }\ebra{id - \lambda_i K}|_{S}$.
	
	For any $z \in \text{range} \ebra{id - \lambda_i K}$, we decompose it as
	\begin{equation*}\label{eq.90.11}
		z = z_S + z_R, \qquad z_S \in \text{range }\ebra{id - \lambda_i K}|_S, \ z_{R} \in R.
	\end{equation*}
	Then the desired preimage $\omega_z \in \text{range } \ebra{id - \lambda_i K}$ satisfying $z = \omega_z - \lambda_i K\omega_z$ is
	\begin{equation*}\label{eq.90.13}
		\omega_z = \ebra{\ebra{id - \lambda_i K}|_{\text{range }\ebra{id - \lambda_i K}|_S}}^{-1} \ebra{z_S + \lambda_i K z_R} + z_R.
	\end{equation*}
	Note that $P_{\lambda_i}\ebra{\lambda_i K z_R} = P_{\lambda_i} \ebra{z_R - \ebra{id - \lambda_i K}z_R} = 0$, $\lambda_i K z_R \in \text{range }K \subset S$, so the formula is well defined. Since operators $K$ and $id - \lambda_i K$ can be represented with basis inside $S$ (see Appendix \ref{appendix.D}), $K$ and $id - \lambda_i K$ and their inverses are computable inside $S$.
	
	\begin{remark}\label{rem.5}
		It follows that under the dissipation condition (H3) with $\delta > p-2$ ($p \in \mathbb{N}, p \geq 3$), the proofs above can be extended to polynomial interactions of degree $p$. In fact, a similar procedure also works if we replace the space by $L^1\ebra{\mathbb{R}^n, \ebra{1 + |x|^2}^{p/2}}$, and represent the operator $K$ in the subspace spanned by $\Psi_\alpha$ with $|\alpha| \leq p$.
	\end{remark}
	
	\subsection{Proof of theorems}\label{conti}
	In the continuous-phase-transition case, we assume the existence of a family of trivial solutions $\{\rho_0\ebra{\lambda}\}_{\lambda \in J}$ around the bifurcation point. Note that $\rho_0\ebra{\lambda}$ is also a fixed point of $\mathcal{T}$ for all admissible $\lambda$
	\begin{equation*}\label{eq.55.1}
		0 = (id - \mathcal{T}\ebra{\lambda, \cdot})\ebra{\rho_0} = \rho_0\ebra{\lambda} - \exp \ebra{\lambda \cdot \ebra{V / \kappa + W * \rho_0\ebra{\lambda}}} / Z\ebra{\lambda}.
	\end{equation*}
	This causes both the constant term and the $\partial_\lambda \tilde{F}\ebra{\lambda, 0}$ term in the expansion of the bifurcation equation \eqref{eq.15} to vanish, which is characteristic of a pitchfork bifurcation.
	
	Under (H4), and in the special symmetric case where $V\ebra{-\cdot} = V\ebra{\cdot}$, we can explicitly determine a family of invariant measures with zero first moments, given by
	\begin{equation}\label{eq.55}
		\begin{aligned}
			\rho_0 & = \exp \ebra{\ebra{-V - \kappa \cdot |x|^2/2}\cdot 2/\sigma^2} / \tilde{Z}\ebra{\kappa, \sigma} \\
			& = \exp\ebra{\lambda \cdot \ebra{V / \kappa + |x|^2/2}} / \tilde{Z}\ebra{\lambda} \quad  \ebra{\lambda = -2\kappa / \sigma^2}.
		\end{aligned}
	\end{equation}
	Due to the symmetry of $V$, all odd-order moments of $\rho_0$ vanish.
	
	\hl{The following Lemma {\ref{lemma.1.1}} follows from direct calculations. Therefore, we omit its proof for simplicity.}
	
	\begin{lemma}[Derivative formulas for $\tilde{F}$ in the continuous-phase-transition case]\label{lemma.1.1}
		Consider the bifurcation equation
		\begin{equation*}\label{eq.59.1}
			\tilde{F}\ebra{\lambda, \delta \rho} = F\ebra{\lambda, \rho_0\ebra{\lambda} + \delta \rho} = \left.\ebra{id - \mathcal{T}\ebra{\lambda, \cdot }}\ebra{\rho}\right|_{\rho = \rho_0\ebra{\lambda} + \delta \rho}.
		\end{equation*}
		Assume that (H1) -- (H3) hold, and $\mathcal{T}\ebra{\lambda, \cdot}$ has
		Fr\'{e}chet derivatives up to third order on a Banach space $X$. Then
		\begin{equation*}\label{eq.59.3}
			\begin{aligned}
				\ebra{\partial_\lambda \mathcal{T}}\ebra{\lambda, \rho_0} & := \left.\partial_\lambda \mathcal{T}\ebra{\lambda, \cdot} \ebra{\rho}\right|_{\rho = \rho_0} = \partial_\lambda \rho_0 - \lambda \cdot \rho_0 \cdot \tilde{K}_\lambda \partial_\lambda \rho_0,  \\
				\partial_\lambda\ebra{\mathcal{T}\ebra{\lambda, \rho_0}} & = \ebra{\partial_\lambda \mathcal{T}} \ebra{\lambda, \rho_0} + \ebra{\partial_\rho \mathcal{T}}\ebra{\lambda, \rho_0}[\partial_\lambda \rho_0] = \partial_\lambda \rho_0. \\
			\end{aligned}
		\end{equation*}
		Moreover, for any $\omega, \omega_1, \omega_2, \omega_3 \in X$,
		\begin{equation*}\label{eq.59.2}
			\begin{aligned}
				\partial_{\lambda} \tilde{F} \ebra{\lambda, 0} & = \partial_\lambda \ebra{\rho_0- \mathcal{T}\ebra{\lambda, \rho_0}} = 0, \\
				\partial_{\delta \rho} \tilde{F} \ebra{\lambda, 0}[\omega] &= \partial_\rho \ebra{id - \mathcal{T}\ebra{\lambda, \cdot}}\ebra{\rho_0}[\omega] =  \ebra{id -\lambda K} \omega , \\
				\partial_{\lambda, \delta \rho}^2 \tilde{F}\ebra{\lambda, 0}[\omega] & = -\ebra{\partial_{\lambda, \rho}^2 \mathcal{T}}\ebra{\lambda, \rho_0}[\omega] - \partial_{\rho, \rho}^2\mathcal{T}\ebra{\lambda, \rho_0}[\partial_\lambda \rho_0, \omega], \\
				\partial_{\delta \rho, \delta \rho}^2 \tilde{F}\ebra{\lambda, 0}[\omega_1, \omega_2] & = - \partial_{\rho, \rho}^2 \mathcal{T}\ebra{\lambda, \rho_0}[\omega_1, \omega_2], \\
				\partial_{\delta \rho, \delta \rho, \delta \rho}^3 \tilde{F} \ebra{\lambda, 0}[\omega_1, \omega_2, \omega_3] & = -\partial_{\rho, \rho, \rho}^3 \mathcal{T}\ebra{\lambda, \rho_0}[\omega_1, \omega_2, \omega_3],
			\end{aligned}
		\end{equation*}
		where 
		\begin{equation}\label{eq.59.5}
			\begin{aligned}
				& \partial_{\lambda, \delta \rho}^2 \tilde{F}\ebra{\lambda, 0}[\omega]  = -K_\lambda \omega - \lambda\cdot \rho_0 \cdot \left(\tilde{K}_\lambda\omega \cdot \ebra{\partial_\lambda\rho_0/\rho_0} - \int_{\mathbb{R}^n} \tilde{K}_\lambda\omega \cdot\ebra{\partial_\lambda\rho_0/\rho_0} \cdot \rho_0 \dif x\right)\\
				&\quad = - K_\lambda\omega - \lambda \cdot \rho_0 \cdot \left(\tilde{K}_\lambda\omega \cdot \ebra{\tilde{L} + \lambda \cdot \tilde{K}_\lambda \partial_\lambda \rho_0} - \int_{\mathbb{R}^n} \tilde{K}_\lambda\omega \cdot \ebra{\tilde{L} + \lambda \cdot \tilde{K}_\lambda \partial_\lambda \rho_0} \cdot \rho_0 \dif x \right).
			\end{aligned}
		\end{equation}
		If (H4) and the symmetry of $V$ are further assumed, and if $\{\rho_0\ebra{\lambda}\}$ is chosen as in \eqref{eq.55}, then
		\begin{equation*}\label{eq.59.4}
			\begin{aligned}
				& W * \partial_\lambda \rho_0  = m_2\ebra{\partial_\lambda\rho_0}/2, \qquad\qquad\qquad\quad\ \  \tilde{K}_\lambda \partial_\lambda \rho_0 = 0, \\
				& \ebra{\partial_\lambda \mathcal{T}}\ebra{\lambda, \rho_0}  = \left.\partial_\lambda \mathcal{T} \ebra{\lambda, \cdot}\right|_{\rho = \rho_0} = \partial_\lambda \rho_0, \qquad \partial^2_{\lambda, \delta \rho} \tilde{F} \ebra{\lambda, 0}[\omega] = -\ebra{\partial^2_{\lambda, \rho}\mathcal{T}}\ebra{\lambda, \rho_0}[\omega].
			\end{aligned}
		\end{equation*}
		The detailed expressions for the terms above are given in (\ref{eq.58.line1})--(\ref{eq.58.line-1}) and (\ref{eq.59}).
	\end{lemma}
	
	The preceding lemma shows that, assuming (H4) and the symmetry of $V$, even though the trivial solution family ${\rho_0(\lambda)}$ depends on $\lambda$, the derivatives appearing in the bifurcation equation are the same as those in the $\lambda$-independent case.
	
	\begin{lemma}[Structure of bifurcation equation, continuous-phase-transition case under quadratic interaction]\label{lemma.3}
		Assume that (H1) -- (H4) hold. Choose $X = L^1\ebra{1 + |x|^2}$. Then the Fr\'{e}chet derivatives up to order three exist for given $\lambda < 0$ and the bifurcation equation can be expanded as
		\begin{equation*}\label{eq.56}
			0 = \tilde{F}\ebra{\lambda, \delta \rho} = \ebra{id - \lambda K_\lambda}\ebra{\delta
				\rho} + G\ebra{\lambda, \delta \rho},
		\end{equation*}
		where $K_\lambda$, the operator defined above in (\ref{eq.58.line1}), is compact, and $G$, the higher-order small term, given by $G\ebra{\lambda, \delta \rho} := -\ebra{\mathcal{T}\ebra{\lambda, \rho_0 + \delta \rho} - \rho_0 - \lambda K\delta \rho }$, is completely continuous.
		We note that $K_\lambda$ depends on $\lambda$ via $\rho_0\ebra{\lambda}$. Moreover, the remaining term $G$ satisfies 
		\begin{equation*}\label{eq.60}
			\begin{aligned}
				& G\ebra{\lambda, 0} = 0 \text{\qquad \qquad \ for \text{ all } $\lambda \in J$}, \\
				& G\ebra{\lambda, \delta \rho} = o\ebra{\norm{\delta \rho}} \text{\quad uniformly with respect to the parameter $\lambda$ as $\delta \rho \to 0$. }
			\end{aligned}
		\end{equation*}
		This leads to
		\begin{equation*}\label{eq.61}
			\partial_{\delta \rho} G\ebra{\lambda, 0}  = 0, \qquad
			\partial_{\lambda, \delta \rho}^2 G\ebra{\lambda, 0}[\omega] = \partial_{\lambda, \delta \rho}^2 \left.\ebra{-\mathcal{T}\ebra{\rho_0 + \delta \rho} + \lambda K_\lambda \delta \rho}\right|_{\delta \rho = 0}[\omega] = 0.
		\end{equation*}
	\end{lemma}

	\begin{proof}\label{proof.8}
		We only sketch the proof. To prove the uniform smallness of $G$, we write it as a Lagrange-type remainder. For some $t \in \ebra{0, 1}$,
		\begin{equation*}\label{eq.62}
			\begin{aligned}
				G\ebra{\lambda, \delta \rho} & = -\partial_{\rho, \rho}^2 \ebra{\mathcal{T}\ebra{\lambda, \rho_0 + t\delta \rho}}[\delta \rho, \delta \rho] / 2\\
				& = -\lambda^2/2 \cdot \mathcal{T}\ebra{\lambda, \rho_0 + t \delta \rho} \cdot \left((A_\lambda \delta \rho)^2 - \int_{\mathbb{R}^n} (A_\lambda \delta \rho)^2 \cdot \mathcal{T}\ebra{\rho_0 + t \delta \rho} \dif x\right),
			\end{aligned}
		\end{equation*}
		where
		\begin{equation*}\label{eq.62.1}
			A_\lambda \delta \rho = W * \delta \rho - \int_{\mathbb{R}^n} W * \delta \rho \cdot T(\lambda, \rho_0 + t \delta \rho).
		\end{equation*}
		The uniform smallness when $\lambda \in J = B_{\delta_1}\ebra{\lambda_0}$ can be derived by using (H2) and estimates (\ref{eq.25.line1})-(\ref{eq.25.line-1}) to bound the Gibbs map $\mathcal{T}$. 
		
		Under (H4), the derivatives up to order three of the numerator $\rho \mapsto \exp(\lambda(V/\kappa+W*\rho))$ and of the integrands defining the normalizing functional $\rho \mapsto Z(\lambda,\rho)$ are finite sums of polynomial factors multiplied by the same Gibbs exponential. By the coercivity of \(V\), these terms, after multiplication by the weight \(1+|x|^2\), are uniformly dominated on bounded neighborhoods by an integrable function of the form $C(1+|x|)^N e^{-c|x|^{2+\delta}}.$ Hence the Fr\'{e}chet remainder estimates follow from the dominated convergence theorem. 
	\end{proof}
	
	
	
	We now establish the pitchfork bifurcation theorems.
	\begin{proof}[Proof of Theorem \ref{thm.7}]\label{proof.12}
		The proof follows from Crandall-Rabinowitz's abstract bifurcation theorem \hl{(see Theorem {\ref{thm.5}})} and the derivative formula in Lemma \ref{lemma.1.1}, combined with detailed computations. Indeed, equation (\ref{eq.95}) follows from equation (\ref{eq.59.5}) and
		\begin{equation*}
			\partial_\lambda \rho_0\ebra{\lambda} = \partial_\lambda (\mathcal{T}\rho_0\ebra{\lambda}) = \lambda K_\lambda \partial_\lambda \rho_0\ebra{\lambda} + \rho_0\ebra{\lambda} \cdot \tilde{L}. 
		\end{equation*}
		Equations (\ref{eq.96}), (\ref{eq.97}) and (\ref{eq.98}) are obtained by substituting $(id - \lambda_\star K_{\lambda_\star})\tilde{v}_\star = 0$ into equations (\ref{eq.4.4.1}) and (\ref{eq.58.line-1}).
	\end{proof}
	
	\begin{proof}[Proof of Theorem \ref{thm.8}]\label{proof.13}
		We investigate the phase transitions driven by quadratic interaction. Compactness of $K$ on $L^1\ebra{\mathbb{R}^n, 1 + |x|^2}$ and smallness of $G$ have been established in Lemmas \ref{thm.1} and \ref{lemma.3}. Spectrum properties of $K$ and the spectral projection operator have been established in Lemmas \ref{lemma.2} and \ref{lemma.2.1}. We therefore focus on the computation of the sufficient conditions. We compute the transversality condition as
		\begin{equation*}\label{eq.63.1}
			\begin{aligned}
				P_{\lambda_\star} \partial_{\lambda, \delta \rho}^2 \tilde{F}\ebra{\lambda_\star, 0}[\tilde{v}_\star]/\tilde{v}_\star & = \frac{-\sqrt{2} \pi^{-n/4}}{\eangle{\tilde{v}_\star, \tilde{u}_\star}} \cdot \left(\lambda_\star^{-1} \cdot \sum_{k, r=1}^n \ebra{v_\star}_k \ebra{v_\star}_r\cdot \ebra{m_{e_k + e_r}\ebra{\rho_0} - m_{e_k}\ebra{\rho_0}m_{e_r}\ebra{\rho_0}}\right.\\
				& \qquad + \sum_{k, r = 1}^n \ebra{v_\star}_k \ebra{v_\star}_r \cdot \left(\int_{\mathbb{R}^n} \ebra{x_k - m_{e_k}\ebra{\rho_0}} \ebra{x_r + \xi_r^{(\star)}} \partial_\lambda \rho_0 \dif x\right.\\
				& \qquad \left.\left.- \int_{\mathbb{R}^n} \ebra{x_k - m_{e_k}\ebra{\rho_0}} \cdot \partial_\lambda\rho_0 \dif x \cdot \int_{\mathbb{R}^n} \ebra{x_r + \xi_r^{(\star)}}\cdot \rho_0\dif x\right)\right).
			\end{aligned}
		\end{equation*}
		Since
		\begin{equation*}\label{eq.63.2}
			\eangle{\tilde{v}_\star, \tilde{u}_\star} = -\sqrt{2} \pi^{-n/4} \lambda_\star^{-1}, \qquad \text{Cov}\ebra{\rho_0} v_\star = -\lambda_\star^{-1} v_\star, \qquad \int_{\mathbb{R}^n} \partial_\lambda \rho_0 \dif x = 0,
		\end{equation*}
		we obtain the equation (\ref{eq.103}). The pitchfork condition is computed as
		\begin{equation*}\label{eq.63.3}
			\begin{aligned}
				P_{\lambda_\star} \partial_{\delta \rho, \delta \rho}^2 \tilde{F}\ebra{\lambda_\star, 0}[\tilde{v}_\star, \tilde{v}_\star]/\tilde{v}_\star & = \frac{-2\pi^{-n/2}}{\eangle{\tilde{v}_\star, \tilde{u}_\star}}\cdot \sum_{k, r, s=1}^n \ebra{v_\star}_k \ebra{v_\star}_r \ebra{v_\star}_s\cdot \\
				& \left(\int_{\mathbb{R}^n} \ebra{x_k - m_{e_k}\ebra{\rho_0}} \ebra{x_r - m_{e_r}\ebra{\rho_0}}\ebra{x_s + \xi_s^{(\star)}\ebra{\rho_0}} \cdot \rho_0 \dif x\right.\\
				& \left. - \int_{\mathbb{R}^n} \ebra{x_k - m_{e_k}\ebra{\rho_0}}\ebra{x_r - m_{e_r}\ebra{\rho_0}}\cdot \rho_0 \dif x \cdot \int_{\mathbb{R}^n} \ebra{x_s + \xi_s^{\star}}\cdot \rho_0 \dif x \right),
			\end{aligned}
		\end{equation*}
		which gives \eqref{eq.104} after simplification. The non-degenerate condition (\ref{eq.106}) is a result of direct computations. Equation (\ref{eq.107}) relies on the key observation that
		\begin{equation}\label{eq.63.4}
			\begin{aligned}
				2 \partial_s w\ebra{\lambda_\star, 0} & = - (\partial_{\delta \rho} \tilde{F}\ebra{\lambda_\star, 0})^{-1} (id - P_{\lambda_\star}) \partial_{\delta \rho, \delta \rho}^2 \tilde{F} \ebra{\lambda_\star, 0}[\tilde{v}_\star, \tilde{v}_\star], \\
				& = -\partial_{\delta \rho, \delta \rho}^2 \tilde{F} \ebra{\lambda_\star, 0}[\tilde{v}_\star, \tilde{v}_\star] - \sum_{\lambda_j \neq \lambda_\star} \frac{\lambda_j}{\lambda_j - \lambda_\star} \cdot \frac{\eangle{\lambda_\star K_{\lambda_\star} \partial_{\delta \rho, \delta \rho}^2 \tilde{F} \ebra{\lambda_\star, 0}[\tilde{v}_\star, \tilde{v}_\star], \tilde{u}_j}}{\eangle{\tilde{v}_j, \tilde{u}_j}} \cdot \tilde{v}_j.
			\end{aligned}
		\end{equation}
		Indeed, by the derivative formula, $m_0\ebra{\partial_{\delta \rho, \delta \rho}^2 \tilde{F} \ebra{\lambda_\star, 0}[\tilde{v}_\star, \tilde{v}_\star]}\ebra{\rho_0} = 0$. Moreover, by the pitchfork condition, where $c\neq 0$ is a normalization constant, we have
		\begin{equation*}\label{eq.63.4.1}
			\begin{aligned}
				P_{\lambda_\star} K_{\lambda_\star} \partial_{\delta \rho, \delta \rho}^2 \tilde{F} \ebra{\lambda_\star, 0}[\tilde{v}_\star, \tilde{v}_\star]/\tilde{v}_\star & = c\cdot \eangle{K_{\lambda_\star} \partial_{\delta \rho, \delta \rho}^2 \tilde{F} \ebra{\lambda_\star, 0}[\tilde{v}_\star, \tilde{v}_\star], \tilde{u}_\star} \\
				& = c\cdot \eangle{\partial_{\delta \rho, \delta \rho}^2 \tilde{F} \ebra{\lambda_\star, 0}[\tilde{v}_\star, \tilde{v}_\star], K^*_{\lambda_\star} \tilde{u}_\star} = 0,
			\end{aligned}
		\end{equation*}
		which implies $K_{\lambda_\star} \partial_{\delta \rho, \delta \rho}^2 \tilde{F} \ebra{\lambda_\star, 0}[\tilde{v}_\star, \tilde{v}_\star] \in \text{span}\{\tilde{v}_1, \dots, \tilde{v}_n\} \setminus\{\tilde{v}_\star\}$. Thus $K_{\lambda_\star} \partial_{\delta \rho, \delta \rho}^2 \tilde{F} \ebra{\lambda_\star, 0}[\tilde{v}_\star, \tilde{v}_\star]$ can be represented by its projections onto this basis. One can also verify that
		\begin{equation*}\label{eq.63.5}
			\begin{aligned}
				& (id - \lambda_\star K_{\lambda_\star}) (-2 \partial_{s} w\ebra{\lambda_\star, 0}) = \partial_{\delta \rho, \delta \rho}^2 \tilde{F} \ebra{\lambda_\star, 0}[\tilde{v}_\star, \tilde{v}_\star] - \lambda_\star K_{\lambda_\star} \partial_{\delta \rho, \delta \rho}^2 \tilde{F} \ebra{\lambda_\star, 0}[\tilde{v}_\star, \tilde{v}_\star] \\
				& \qquad \qquad \quad + \sum_{\lambda_j \neq \lambda_\star} \frac{\eangle{\lambda_\star K_{\lambda_\star} \partial_{\delta \rho, \delta \rho}^2 \tilde{F} \ebra{\lambda_\star, 0}[\tilde{v}_\star, \tilde{v}_\star], \tilde{u}_j}}{\eangle{\tilde{v}_j, \tilde{u}_j}} \cdot \tilde{v}_j = \partial_{\delta \rho, \delta \rho}^2 \tilde{F} \ebra{\lambda_\star, 0}[\tilde{v}_\star, \tilde{v}_\star].
			\end{aligned}
		\end{equation*}
		Consequently, (\ref{eq.63.4}) holds.
		
		Finally, using the identity
		\begin{equation*}\label{eq.63.6}
			\begin{aligned}
				\lambda_\star K_{\lambda_\star}(2 \partial_{s} w\ebra{\lambda_\star, 0}) & = 2 \partial_{s} w\ebra{\lambda_\star, 0} + \partial_{\delta \rho, \delta \rho}^2 \tilde{F} \ebra{\lambda_\star, 0} [\tilde{v}_\star, \tilde{v}_\star]\\
				& = \sum_{\lambda_j \neq \lambda_\star} \frac{\lambda_j}{\lambda_j - \lambda_\star} \cdot \frac{\eangle{\lambda_\star\partial_{\delta \rho, \delta \rho}^2 \tilde{F} \ebra{\lambda_\star, 0}[\tilde{v}_\star, \tilde{v}_\star], \tilde{u}_j}}{\sqrt{2}\pi^{-n/4}} \cdot \tilde{v}_j,
			\end{aligned}
		\end{equation*}
		we obtain
		\begin{equation*}\label{eq.63.7}
			\begin{aligned}
				&P_{\lambda_\star} \partial_{\delta \rho, \delta \rho}^2 \tilde{F} \ebra{\lambda_\star, 0}[\tilde{v}_\star, 2 \partial_{s} w \ebra{\lambda_\star, 0}]/\tilde{v}_\star \\
				&= P_{\lambda_\star} \left(-\rho_0 \cdot \left( \frac{\tilde{v}_\star}{\rho_0}\cdot \sum_{\lambda_j \neq \lambda_\star} \frac{\lambda_j}{\lambda_j - \lambda_\star} \cdot \frac{\eangle{\lambda_\star\partial_{\delta \rho, \delta \rho}^2 \tilde{F} \ebra{\lambda_\star, 0}[\tilde{v}_\star, \tilde{v}_\star], \tilde{u}_j}}{\sqrt{2}\pi^{-n/4}} \cdot \frac{\tilde{v}_j}{\rho_0} - \right.\right.\\
				& \qquad\qquad \left.\left.\int_{\mathbb{R}^n} \frac{\tilde{v}_\star}{\rho_0}\cdot \sum_{\lambda_j \neq \lambda_\star} \frac{\lambda_j}{\lambda_j - \lambda_\star} \cdot \frac{\eangle{\lambda_\star\partial_{\delta \rho, \delta \rho}^2 \tilde{F} \ebra{\lambda_\star, 0}[\tilde{v}_\star, \tilde{v}_\star], \tilde{u}_j}}{\sqrt{2}\pi^{-n/4}} \cdot \frac{\tilde{v}_j}{\rho_0} \cdot \rho_0 \dif x\right)\right)/\tilde{v}_\star\\
				& = \frac{\lambda_\star^2}{2 \pi^{-n/2}} \cdot \sum_{\lambda_j \neq \lambda_\star} \frac{\lambda_j \eangle{\partial_{\delta \rho, \delta \rho}^2 \tilde{F} \ebra{\lambda_\star, 0}[\tilde{v}_\star, \tilde{v}_\star], \tilde{u}_j}}{\lambda_j - \lambda_\star} \cdot\left(\eangle{\frac{\tilde{v}_\star \tilde{v}_j}{\rho_0^2}, \tilde{u}_\star\rho_0} - \eangle{\frac{\tilde{v}_\star \tilde{v}_j}{\rho_0^2}, \rho_0}\cdot \eangle{\tilde{u}_\star, \rho_0}\right),
			\end{aligned}
		\end{equation*}
		and
		\begin{equation*}\label{eq.63.8}
			\eangle{\partial_{\delta \rho, \delta \rho}^2 \tilde{F} \ebra{\lambda_\star, 0}[\tilde{v}_\star, \tilde{v}_\star], \tilde{u}_j} = -\eangle{\frac{\tilde{v}_\star \tilde{v}_j}{\rho_0^2}, \tilde{u}_\star\rho_0} + \eangle{\frac{\tilde{v}_\star \tilde{v}_j}{\rho_0^2}, \rho_0}\cdot \eangle{\tilde{u}_\star, \rho_0}.
		\end{equation*}
		Combining these identities, we arrive at \eqref{eq.107} in the nondegeneracy condition.
	\end{proof}
	
	\begin{proof}[Proof of Corollary \ref{cor.1}]\label{proof.14}
		The corollary follows from the vanishing of the odd-indexed moments of the symmetric trivial solutions $\{\rho_0(\lambda)\}$, together with the results of Lemma \ref{lemma.1.1}.
	\end{proof}
	
	Theorem \ref{thm.9} follows from the saddle-node bifurcation theorem \hl{(see Theorem {\ref{thm.6}})} after similar computations. We thus omit the proof.
	
	\section{Proof of examples}\label{ex_proof}
	This section presents proofs for examples. Henceforth, we use the angle bracket $\eangle{\cdot}$ to denote
	\begin{equation*}\label{eq.112}
		\eangle{f}_{\rho} := \int_{\mathbb{R}^n} f \cdot \rho \dif x,
	\end{equation*}
	and omit the subscript $\rho$ when no ambiguity can arise.
	
	For the (a)symmetric double-well potential with quadratic interaction, define $V_s := x^4/4 - sx^3/3 - x^2/2, \ s \in [0, 1]$. The system of moment equations becomes
	\begin{equation}\label{eq.16.2}
		\begin{aligned}
			& m_3 = m_1 + sm_2, \\
			& m_4 = \sigma^2/2 + (1-\kappa)m_2 + sm_3 + \kappa m_1^2, \\
			& m_5 = \sigma^2m_1 + (1-\kappa) m_3 + sm_4 + \kappa m_1m_2, \\
			& m_6 = 3\sigma^2/2 \cdot m_2 + (1-\kappa) m_4 + sm_5 + \kappa m_1 m_3, \quad \cdots,\\
			& m_{p+2} = \sigma^2/2\cdot \ebra{p-1} m_{p-2} + \ebra{1-\kappa} m_p + s m_{p+1} + \kappa m_1 m_{p-1}.
		\end{aligned}
	\end{equation}
	Then the system \eqref{eq.16.2}, combined with Dawson's criterion in one dimension
	\begin{equation*}\label{eq.16.3}
		m_2 - m_1^2 = \sigma^2/\ebra{2\kappa},
	\end{equation*}
	enables us to characterize the phase-transition point.
	
	\begin{remark}\label{rem.10}
		To estimate the moments, we substitute \eqref{eq.16.2} into the Hankel inequalities and obtain expressions depending only on $m_1$ and $m_2$. The Hankel inequalities then provide arbitrarily precise moment estimates, provided that sufficiently many of these inequalities are considered.
	\end{remark}
	
	Each fixed point $m$ of the real-valued map $\eangle{x}_{\rho}\ebra{\sigma, \cdot} := \int_{\mathbb{R}} x \cdot \rho\ebra{\sigma, \cdot} \dif x$ corresponds to an invariant measure. The graph in the symmetric double-well case is shown below.
	
	\begin{figure}[htb]\label{sub.1}
		\centering
		\includegraphics[scale=0.95]{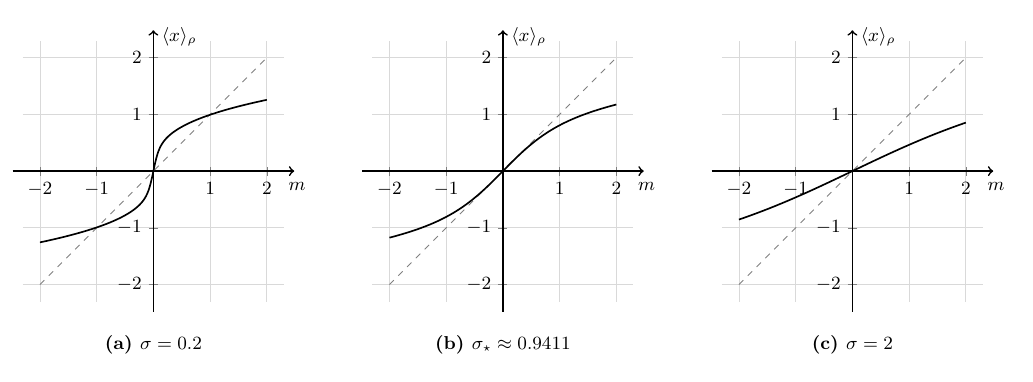}
		\caption{Graphs of $\eangle{x}_{\rho}$ and $id$ with respect to $m$ as $\sigma$ varies, with $\kappa = 1$. In the symmetric double-well case, the graph is centrally symmetric with respect to the point $(0, 0)$. When $0 < \sigma < \sigma_\star$, $\eangle{x}_{\rho}$ has three fixed points. When $\sigma = \sigma_\star$, the graph of $\eangle{x}_{\rho}$ is tangent to $id$ at the origin. When $\sigma > \sigma_\star$, $\eangle{x}_{\rho}$ has only one fixed point.}
	\end{figure}
	
	\begin{proof}[Proof of Example \ref{ex.1}]\label{proof.15}
		\underline{The local argument.} The main idea of the proof is to derive moment estimates from \eqref{eq.mom_eq} and to verify the conditions in Corollary \ref{cor.1}.
		
		\textit{Claim:} If $\ebra{\lambda_\star, \rho_0\ebra{\lambda_\star}}$ is a phase-transition point for $V = x^4/4 - x^2/2$, then $1/3 < m_2 < 1$ for every fixed $\kappa > 0$.
		
		Indeed, taking $s=0$ in moment-equation-system combined with Dawson's criterion (\ref{eq.16.2}), and using symmetry of trivial solutions $\rho_0\ebra{\lambda_\star}$, we deduce that the odd order moments vanish. In this case,  (\ref{eq.16.2}) becomes (we write $m_q\ebra{\rho_0\ebra{\lambda_\star}}, q\in \mathbb{N}$ as $m_q$ for the sake of simplicity)
		\begin{equation*}\label{eq.113}
			\begin{aligned}
				& m_q = 0 \quad \text{for $q$ odd}, \quad m_2 = \sigma_\star^2/\ebra{2\kappa}, \\
				& m_4 = \sigma_\star^2/2 + (1-\kappa) m_2 = m_2, \\
				& m_6 = 3\sigma_\star^2/2 \cdot m_2 + (1-\kappa) m_4 = 3\kappa m_2^2 + (1-\kappa)m_2, \quad \cdots.
			\end{aligned}
		\end{equation*}
		Substituting these equalities into the Hankel inequalities gives
		\begin{equation*}\label{eq.114}
			\begin{aligned}
				\Delta_0 &= 1 \geq 0, \quad \Delta_1 = m_2 - m_1^2 \geq 0,\quad \Delta_2 = m_2m_4 - m_2^3 \geq 0, \\
				\Delta_3 &= m_2 m_4 m_6 - m_4^3 - m_2^3 m_6 + m_2^2 m_4^2 = \kappa m_2^3\cdot\ebra{-3m_2+1}\cdot\ebra{m_2-1} \geq 0.
			\end{aligned}
		\end{equation*}
		The Hankel inequalities take a simple form because of the symmetry of $\rho_0\ebra{\lambda_\star}$. The inequalities $\Delta_0, \Delta_1, \Delta_2\geq 0$ hold automatically, and $\Delta_3 \geq 0$ if and only if
		\begin{equation*}\label{eq.115}
			\frac{1}{3} \leq m_2 \leq 1.
		\end{equation*}
		
		The Hankel determinants satisfy $\Delta_q = 0$ if and only if $\rho_0$ is supported on finitely many points. Since the invariant measure $\rho_0$ is absolutely continuous with respect to the Lebesgue measure, $\Delta_q$ cannot vanish, which yields the strict estimate $1/3 < m_2 < 1$.
		
		Next, we verify conditions required in Corollary \ref{cor.1}. The existence of $\ebra{\lambda_\star, \rho_0\ebra{\lambda_\star}}$ satisfying Dawson's criterion (\ref{eq.109}) can be easily checked  by considering the small noise limit and using intermediate value theorem in 1-dimension, or following Dawson's approach (\cite[Theorem 3.3.2 (c)]{Daw83}). The transversality condition holds because
		\begin{equation*}\label{eq.116}
			\begin{aligned}
				& P_{\lambda_\star} \partial_{\lambda, \delta \rho}^2 \tilde{F}\ebra{\lambda_\star, 0}[\tilde{v}_\star]/\tilde{v}_\star = -\lambda_\star^{-1}  + \lambda_\star \cdot \eangle{x^2}_{\partial_\lambda \rho_0}\\
				& = m_2 - m_2^{-1} \cdot \left(\eangle{x^2 \cdot \ebra{\frac{1}{4\kappa}x^4 + \ebra{\frac{1}{2} - \frac{1}{2\kappa}}x^2}}_{\rho_0} - \eangle{x^2}_{\rho_0} \cdot \eangle{\frac{1}{4\kappa}x^4 + \ebra{\frac{1}{2} - \frac{1}{2\kappa}}x^2}_{\rho_0}\right)\\
				& = m_2 - m_2^{-1} \cdot \left(\ebra{\frac{1}{4} + \frac{1}{4\kappa}}m_2^2 + \ebra{\frac{1}{4} - \frac{1}{4\kappa}}m_2\right) = \ebra{\frac{3}{4} - \frac{1}{4\kappa}}\cdot m_2 - \ebra{\frac{1}{4} - \frac{1}{4\kappa}} \\
				& \geq \frac{1}{2} \wedge \frac{1}{6\kappa} > 0.
			\end{aligned}
		\end{equation*}
		The pitchfork condition holds because
		\begin{equation*}\label{eq.117}
			P_{\lambda_\star} \partial_{\delta \rho, \delta \rho}^2 \tilde{F}\ebra{\lambda_\star, 0}[\tilde{v}_\star, \tilde{v}_\star]/\tilde{v}_\star = \sqrt{2} \pi^{-n/4} \lambda_\star \eangle{\ebra{x - \eangle{x}}^3}_{\rho_0} = 0.
		\end{equation*}
		The non-degenerate condition holds because
		\begin{equation*}\label{eq.128}
			\begin{aligned}
				P_{\lambda_\star} \partial_{\delta \rho, \delta \rho, \delta \rho}^3 \tilde{F} \ebra{\lambda_\star, 0}[\tilde{v}_\star, \tilde{v}_\star, \tilde{v}_\star]/ \ebra{3\tilde{v}_\star} & = 2\pi^{-n/2}/3\cdot  \lambda_\star \cdot \ebra{\eangle{\ebra{x - \eangle{x}}^4}_{\rho_0} - 3 \eangle{\ebra{x-\eangle{x}}^2}_{\rho_0}^2} \\
				& = 2\pi^{-n/2}/3 \cdot \lambda_\star \cdot m_2 \ebra{-3m_2 + 1} > 0.
			\end{aligned}
		\end{equation*}
		It follows from Corollary \ref{cor.1} that the system exhibits a continuous phase transition controlled by a subcritical pitchfork bifurcation at $\ebra{\lambda_\star, \rho_0\ebra{\lambda_\star}}$.
		
		\underline{The global argument.}
		In this example, the moment bifurcation equation (\ref{eq.16}) is essentially
		\begin{equation*}\label{eq.131}
			\begin{aligned}
				& \ebra{id - \tilde{\mathcal{T}}}\ebra{\lambda, m} = m - \int_{\mathbb{R}} x \cdot \rho\ebra{\lambda, m} \dif x = 0, \\
				& \rho\ebra{\lambda, m} = Z\ebra{\lambda, m}^{-1} \cdot \exp\ebra{\lambda\cdot \ebra{V/\kappa + x^2/2 - mx}}.
			\end{aligned}
		\end{equation*}
		Here $V = x^4/4 - x^2/2$ is the symmetric double-well potential. We determine the fixed points of the map $\tilde{\mathcal{T}}$. It is immediate that $0$ is a trivial fixed point. By the symmetry of the equation, it suffices to study the case $m > 0$.
		
		To obtain the global result, we use a version of the Griffiths-Hurst-Sherman (GHS) inequality on $\mathbb{R}$ \cite{Daw83, EMN76}. Under the condition
		\begin{equation*}
			\text{$\bar{V} \in C^3$ even, $\lim\limits_{x \to \pm\infty} \bar{V}\ebra{x} = \infty$, and $\bar{V}^\prime$ convex on $[0, \infty)$. }
		\end{equation*}
		For the partition function $Z\ebra{\lambda, m} := \int_{\mathbb{R}} \exp(\lambda\cdot\ebra{x^2/2 - mx}) \cdot \exp\ebra{\lambda\cdot \bar{V}/\kappa} \dif x$, we have
		\begin{equation*}\label{eq.129}\tag{GHS}
			\begin{aligned}
				\dif^3/\text{d} m^3\
				\ln Z\ebra{\lambda, m} = -\lambda^3 \cdot \eangle{\ebra{x - \eangle{x}}^3}_{\rho} \leq 0, \quad \text{ for all } m\geq 0.
			\end{aligned}
		\end{equation*}
		The derivatives of the logarithm of the partition function generate the central moments of the invariant measure. For example,
		\begin{equation*}\label{eq.130}
			\begin{aligned}
				& \text{d}/\text{d} m\ \ln Z \ebra{\lambda, m} = \int_{\mathbb{R}} -\lambda x \cdot \rho\ebra{\lambda, m} \dif x = -\lambda \tilde{\mathcal{T}} \ebra{\lambda, m}, \\
				& \text{d}^2/\text{d} m^2\ \ln Z \ebra{\lambda, m} = \lambda^2 \cdot \eangle{\ebra{x - \eangle{x}}^2}_{\rho} = -\lambda \partial_m \tilde{\mathcal{T}}\ebra{\lambda, m}.
			\end{aligned}
		\end{equation*}
		Thus, the GHS inequality states that $\tilde{\mathcal{T}}$ is concave with respect to $m$ for fixed $\lambda < 0$.
		
		\textit{Claim:} The following global conclusions hold for any fixed $\lambda < 0$.
		\begin{enumerate}
			\item When $\partial_m \tilde{\mathcal{T}}\ebra{\lambda, 0} = -\lambda \cdot \eangle{\ebra{x - \eangle{x}}^2}_{\rho\ebra{\lambda, 0}} < 1$, the map $\tilde{\mathcal{T}}$ has only the trivial fixed point.
			\item When $\partial_m \tilde{\mathcal{T}}\ebra{\lambda, 0} = -\lambda \cdot \eangle{\ebra{x - \eangle{x}}^2}_{\rho\ebra{\lambda, 0}} > 1$, the map $\tilde{\mathcal{T}}$ has a unique positive fixed point $m_+$. Moreover, $\partial_m \tilde{\mathcal{T}} \ebra{\lambda, m_+} < 1$.
		\end{enumerate}
		
		Indeed, (i) can be obtained by considering the concavity of $\tilde{\mathcal{T}}$ ensured by the (GHS) inequality. For (ii),  $\tilde{\mathcal{T}}\ebra{\lambda, 0} = 0$ and $\partial_m \tilde{\mathcal{T}}\ebra{\lambda, 0} > 1$ implies the existence of a small interval such that for any $m\in \ebra{0, \epsilon}$: $\tilde{\mathcal{T}}\ebra{\lambda, m} > m$ and $\partial_m \tilde{\mathcal{T}}\ebra{\lambda, m} > 1$. Via appropriate change of variables and the Laplace method, we estimate $\tilde{\mathcal{T}}\ebra{\lambda, m} \sim \ebra{\kappa m}^{1/3}$ as $m \to \infty$, so the graph of $\tilde{\mathcal{T}}$ lies below that of $id$ when $m$ is large enough. Then by the intermediate value theorem, we deduce the existence of positive fixed points. Let the smallest one be $m_+$. Now we argue that if $\partial_m \tilde{\mathcal{T}}\ebra{m_+} \geq 1$, then the concavity implies $\partial_m \tilde{\mathcal{T}}\ebra{\lambda, m} \geq 1$ for any $\epsilon \leq m \leq m_+$. Then
		\begin{equation*}\label{eq.132}
			\begin{aligned}
				\tilde{\mathcal{T}}\ebra{\lambda, m_+} &= \int_0^\epsilon \partial_m \tilde{\mathcal{T}}\ebra{\lambda, s} \dif s + \int_\epsilon^{m_+} \partial_m \tilde{\mathcal{T}}\ebra{\lambda, s} \dif s\\
				& > \int_0^\epsilon 1 \dif s + \int_\epsilon^{m_+} 1 \dif s > m_+,
			\end{aligned}
		\end{equation*}
		which yields a contradiction. Therefore, $\partial_m \tilde{\mathcal{T}} \ebra{\lambda, m_+} < 1$ and $\tilde{\mathcal{T}}\ebra{\lambda, m} < m$ when $m > m_+$. Thus the positive fixed point $m_+$ is unique.
		
		By \cite[Theorem 3.3.2]{Daw83}, $\partial_m \tilde{\mathcal{T}}\ebra{\lambda, 0} < 1$ when $\sigma$ is sufficiently large, and $\partial_m \tilde{\mathcal{T}}\ebra{\lambda, 0} > 1$ when $\sigma$ is sufficiently close to zero. Moreover, the critical value $\sigma_\star$ such that $\partial_m \tilde{\mathcal{T}}\ebra{\lambda_\star, 0} = 1$ is unique. Combining these results, we deduce that $\partial_m \tilde{\mathcal{T}}\ebra{\lambda, 0} > 1$ when $0 < \sigma < \sigma_\star$, and $\partial_m \tilde{\mathcal{T}}\ebra{\lambda, 0} < 1$ when $\sigma > \sigma_\star$. This completes the proof of the global argument.
		
		We also record the following local statements, which are, in fact, equivalent to the local pitchfork bifurcation theorem. Assume that $\lambda_\star = -2\kappa/\sigma_\star^2$ is the critical parameter such that $\partial_m \tilde{\mathcal{T}}\ebra{\lambda_\star, 0} = -\lambda_\star \cdot \eangle{\ebra{x - \eangle{x}}^2}_{\rho\ebra{\lambda_\star, 0}} = 1$. We have
		\begin{enumerate}
			\item If the transversality condition $\partial^2_{\lambda, m} \tilde{\mathcal{T}} \ebra{\lambda_\star, 0} < 0$ holds, then $\partial_{m}\tilde{\mathcal{T}}\ebra{\lambda, 0} > 1$ for $\lambda \in \ebra{\lambda_\star - \epsilon, \lambda_\star}$, where $\epsilon > 0$ is small.
			\item If the non-degeneracy condition $\partial_{m, m, m}^3 \tilde{\mathcal{T}}\ebra{\lambda_\star, 0} < 0$ holds, then there is only the trivial fixed point $0$ in the critical case $\partial_m \tilde{\mathcal{T}}\ebra{\lambda_\star, 0} = 1$.
		\end{enumerate}
		
		If we expand $\tilde{\mathcal{T}}$ at $\ebra{\lambda_\star, 0}$ up to second order, all terms except those involving $\partial_m$ and $\partial_{\lambda, m}^2$ vanish
		\begin{equation*}\label{eq.133}
			\begin{aligned}
				\tilde{\mathcal{T}}\ebra{\lambda, m} & = \partial_m \tilde{\mathcal{T}}\ebra{\lambda_\star, 0} m + \partial_{\lambda, m}^2 \tilde{\mathcal{T}}\ebra{\lambda_\star, 0} \ebra{\lambda - \lambda_\star}m + \cdots, \\
				& = \ebra{1 + \partial_{\lambda, m}^2 \tilde{\mathcal{T}}\ebra{\lambda_\star, 0}\ebra{\lambda - \lambda_\star}} m + \cdots,
			\end{aligned}
		\end{equation*}
		which implies (i). At the critical parameter $\lambda_\star$, the Taylor expansion with respect to $m$ is
		\begin{equation*}\label{eq.134}
			\tilde{\mathcal{T}}\ebra{\lambda_\star, m} - m = \partial_{m, m, m}^3 \tilde{\mathcal{T}}\ebra{\lambda_\star, 0} m^3/6 + \cdots,
		\end{equation*}
		which implies that $\tilde{\mathcal{T}}\ebra{\lambda_\star, m} - m$ is negative for $m \in \ebra{0, \epsilon}$. By concavity, it cannot become positive for larger $m$. Hence (ii) follows.
	\end{proof}
	
	\begin{figure}[htb]
		\includegraphics[scale=0.95]{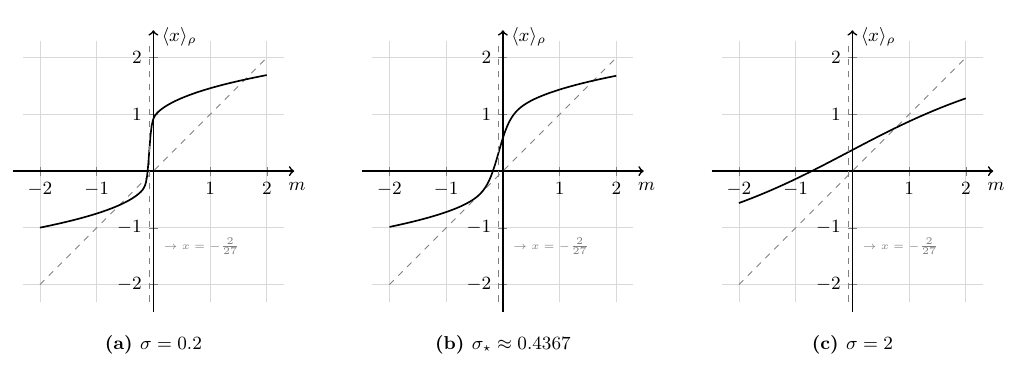}
		\caption{Graphs of $\eangle{x}_{\rho}$ and $id$ with respect to $m$ as $\sigma$ varies, with $\kappa = 1$. In the asymmetric double-well case, the graph is centrally symmetric with respect to the point $(-2/27, 1/3)$. When $0 < \sigma < \sigma_\star$, $\eangle{x}_{\rho}$ has three fixed points; when $\sigma = \sigma_\star$, the graph of $\eangle{x}_{\rho}$ is tangent to $id$; when $\sigma > \sigma_\star$, $\eangle{x}_{\rho}$ has only one fixed point.}\label{sub.2}
	\end{figure}
	
	\begin{proof}[Proof of Example \ref{ex.3}]\label{proof.17}
		We first explain the hidden symmetry in the asymmetric double-well model $V = x^4/4 - x^3/3 - x^2/2$.
		
		\textit{Claim: }
		\begin{enumerate}
			\item (Central symmetry after translation) For any fixed $\lambda$, consider
			\begin{equation*}\label{eq.138}
				\eangle{x}_{\rho}\ebra{\lambda, m} = \int_{\mathbb{R}} x \cdot \rho\ebra{\lambda, m} \dif x, \qquad \rho\ebra{\lambda, m} = \exp\ebra{\lambda \cdot \ebra{\frac{1}{4}x^4 - \frac{1}{3}x^3 - mx}} / \tilde{Z}\ebra{\lambda, m}.
			\end{equation*}
			Then $\eangle{x}_{\rho} - 1/3$ is centrally symmetric with respect to the point $(m_c, 1/3)$, where $m_c = -2/27$.
			\begin{equation*}\label{eq.139}
				\ebra{\eangle{x}_{\rho} - \frac{1}{3}} \ebra{\lambda, m_c + m} = -\ebra{\eangle{x}_{\rho}-\frac{1}{3}} \ebra{\lambda, m_c - m}.
			\end{equation*}
			\item (Concavity by the GHS inequality) $\eangle{x}_\rho\ebra{\lambda, \cdot}$ is strictly concave on the interval $m > m_c$. This implies $\eangle{\ebra{x - \eangle{x}}^3}_\rho\ebra{\lambda, \cdot} < 0$ for $m > m_c$.
		\end{enumerate}
		
		Proof of (i). The central symmetry property originates from the key observation
		\begin{equation*}\label{eq.140}
			\frac{1}{4}x^4 - \frac{1}{3}x^3 - mx = \frac{1}{4}\ebra{x - \frac{1}{3}}^4 - \frac{1}{6}\ebra{x - \frac{1}{3}}^2 - \ebra{m - m_c}\ebra{x - \frac{1}{3}} - \frac{1}{3}m - \frac{1}{108},
		\end{equation*}
		which is even exactly when $m = m_c$. With the substitution $y = x - 1/3$,
		\begin{equation*}\label{eq.141}
			\begin{aligned}
				& \eangle{x}_\rho\ebra{\lambda, m} = \int_{\mathbb{R}} y \cdot \exp\ebra{\lambda \cdot \ebra{\frac{1}{4}y^4 - \frac{1}{6}y^2 - \ebra{m-m_c}y}} / \widetilde{Z\ebra{\lambda, m}} \dif y + \frac{1}{3}, \\
				& \widetilde{Z\ebra{\lambda, m}} = \int_{\mathbb{R}} \exp\ebra{\lambda\cdot \ebra{\frac{1}{4}y^4 - \frac{1}{6}y^2 - \ebra{m - m_c}y}} \dif y.
			\end{aligned}
		\end{equation*}
		We then deduce that
		\begin{equation*}\label{eq.142}
			\widetilde{Z\ebra{\lambda, m_c + m}} = \widetilde{Z\ebra{\lambda, m_c - m}}, \qquad \rho\ebra{\lambda, m_c + m} \ebra{x - 1/3} = \rho\ebra{\lambda, m_c - m}\ebra{-\ebra{x - 1/3}},
		\end{equation*}
		and the result follows.
		
		Proof of (ii). To prove (ii), we use an extended version GHS inequality with particular emphasis on the equality condition \hl{(see Appendix {\ref{appe.GHS}})}. For a potential $\bar{V}$ satisfying
		\begin{equation*}
			\text{ $\bar{V} \in C^3$ even, $\lim\limits_{x \to \pm\infty} \bar{V}\ebra{x} = \infty$, and $\bar{V}^\prime$ convex on $[0, \infty)$},
		\end{equation*}
		and the partition function $Z\ebra{\lambda, m} = \int_{\mathbb{R}} \exp(\lambda\cdot\ebra{\bar{V}\ebra{x} - mx}) \dif x$, we have
		\begin{equation}\label{eq.143}\tag{GHS'}
			\begin{aligned}
				& \dif^3/\text{d} m^3\
				\ln Z\ebra{\lambda, m} = -\lambda^3 \cdot \eangle{\ebra{x - \eangle{x}}^3}_{\rho} \leq 0, \quad \text{ for all } m\geq 0. \\
				& \text{ The inequality is strict for any $m > 0$ if and only if } \bar{V}^{\prime\prime\prime} \not\equiv 0 \text{ on } [0, \infty). \text{ Namely, } \\
				& \text{ } \bar{V} \text{ is not } \text{quadratic or } \rho\ebra{\lambda, m} = \exp\ebra{\lambda \cdot \ebra{\bar{V}\ebra{x} - mx}} / Z\ebra{\lambda, m} \text{ is  not Gaussian. }
			\end{aligned}
		\end{equation}
		In the asymmetric double-well case, taking $\bar{V} = y^4/4 - y^2/6$ and using the identity $\text{d}\ebra{\ln Z}/\text{dm} = -\lambda \cdot \eangle{x}_\rho$, the GHS inequality yields the desired concavity result.

		\underline{The local argument. } Theorem \ref{thm.9} allows us to establish a phase transition point $\ebra{\lambda_\star, \rho\ebra{\lambda_\star, m_\star}}$ where the system exhibits a saddle-node bifurcation.

		We first establish the existence of a critical point $\ebra{\lambda_\star, m_\star}$ satisfying
		\begin{equation}\label{eq.fixed_tangent}
			\eangle{x}\rho\ebra{\lambda_\star, m_\star} = m_\star, \qquad \partial_m \eangle{x}\rho \ebra{\lambda_\star, m_\star} = 1.     
		\end{equation}
		Our main objective is to investigate the small and large noise regimes and draw conclusions via the continuity of $\eangle{x}_\rho$ with respect to temperature. In the small-noise regime, the classical Laplace expansion implies that $\eangle{x}_\rho\ebra{\lambda, m} \to \operatorname{argmin}\{x^4/4 - x^3/3 - mx\}$ as $\lambda \to -\infty$ for $m < m_c$. For any $m \leq -4/27$, the unique minimizer of the effective potential admits an explicit expression:
		\begin{equation*}
			x_{\min}\ebra{m} = \frac{1}{3} + \sqrt[3]{\frac{1}{2}\ebra{m + \frac{2}{27}}+\frac{1}{2}\sqrt{m\ebra{m+\frac{4}{27}}}} + \sqrt[3]{\frac{1}{2}\ebra{m + \frac{2}{27}}-\frac{1}{2}\sqrt{m\ebra{m + \frac{4}{27}}}}. 
		\end{equation*}
		
		In particular, $x_{\min}\ebra{-4/27} = -1/3 < -4/27$. This implies that, for sufficiently large $|\lambda|$, the graph of $\eangle{x}_\rho\ebra{\lambda, \cdot}$ lies locally below the identity line near $m=-4/27$. Given that the graph passes through the point $\ebra{-2/27, 1/3}$, and by invoking its strict concavity alongside the asymptotic growth estimate $\eangle{x}_\rho\ebra{\lambda, \cdot} \sim (m)^{1/3}$ as $m \to -\infty$, we deduce that the map $\eangle{x}_\rho\ebra{\lambda, \cdot}$ possesses exactly two fixed points in $(-\infty, m_c]$ in the low-temperature regime. 
		
		In the large noise regime, by applying the change of variable $y = x / \sqrt{\sigma}$ within the integral and performing an asymptotic expansion, we obtain an expression that holds for any $m$
		\begin{equation*}
			\begin{aligned}
				\eangle{x}_\rho\ebra{\lambda, m} & = \frac{\int_{\mathbb{R}}y \cdot \exp(-\frac{1}{2}y^4 + \frac{2}{3}y^3 \cdot \sqrt[4]{\frac{-\lambda}{2}} + 2my \cdot (\sqrt[4]{\frac{-\lambda}{2}})^3) \dif y}{\int_{\mathbb{R}} \exp(-\frac{1}{2}y^4 + \frac{2}{3}y^3 \cdot \sqrt[4]{\frac{-\lambda}{2}} + 2my \cdot (\sqrt[4]{\frac{-\lambda}{2}})^3) \dif y} \cdot \ebra{\sqrt[4]{\frac{-\lambda}{2}}}^{-1}\\
				& = \frac{\frac{2}{3} \sqrt[4]{\frac{-\lambda}{2}}\cdot \int_{\mathbb{R}} y^4 \exp(-\frac{1}{2}y^4) \dif y + O((\sqrt[4]{\frac{-\lambda}{2}})^3)}{\int_{\mathbb{R}}\exp(-\frac{1}{2}y^4) \dif y + O(\sqrt[4]{\frac{-\lambda}{2}})}\cdot \ebra{\sqrt[4]{\frac{-\lambda}{2}}}^{-1} \\
				& = \frac{1}{3} + O((\sqrt[4]{\frac{-\lambda}{2}})^2) \qquad \text{as $\lambda \to 0$}. 
			\end{aligned}
		\end{equation*}
		This implies that the graph of $\eangle{x}_\rho\ebra{\lambda, \cdot}$ lies above that of the identity map in the large noise regime. 
		
		Since $\eangle{x}_\rho\ebra{\lambda, \cdot}$ changes continuously with respect to $\lambda$, the monotonicity inequality established in Appendix \ref{numer} implies that $\partial_\lambda \eangle{x}_\rho\ebra{\lambda, m} > 0$ for any $m < m_c$. Furthermore, the GHS inequality guarantees the convexity of $\eangle{x}_\rho \ebra{\lambda, m}$ on $m \in (-\infty, m_c]$ for any $\lambda < 0$. If we define
		\begin{equation*}
			l_\lambda := \inf_{m \leq m_c} \{\eangle{x}_\rho \ebra{\lambda, m} - m\}, 
		\end{equation*}
		it follows that $l_\lambda$ is strictly decreasing as $\lambda$ decreases. Consequently, there exists a unique value $\lambda_\star$ such that $l_{\lambda_\star} = 0$. Letting $m_\star$ denote the unique minimizer of $\eangle{x}_\rho \ebra{\lambda_\star, m} - m$, it is straightforward to show that $\ebra{\lambda_\star, m_\star}$ satisfies \eqref{eq.fixed_tangent}.
		
		Next, we verify the transversality and non-degeneracy conditions. Specifically, the transversality condition is guaranteed by the monotonicity inequality established in Appendix \ref{numer}, the property $m_\star = m_1\ebra{\rho_\star} < m_c$, the faster-than-quadratic growth of the even effective potential $\bar{V}$, alongside the central symmetry of $\langle x \rangle_{\rho_\star}$
		\begin{equation*}\label{eq.151.2}
			\begin{aligned}
				\sqrt{2}\pi^{-1/4} \cdot P_{\lambda_\star} \partial_\lambda F\ebra{\lambda_\star, \rho_\star}/\tilde{v}_\star & = \lambda_\star \cdot \left(\eangle{x\cdot \ebra{\frac{1}{4}x^4 - \frac{1}{3}x^3 - m_\star x}}_{\rho_\star} - \eangle{x}_{\rho_\star}\cdot \eangle{\frac{1}{4}x^4 - \frac{1}{3}x^3 - m_\star x}_{\rho_\star}\right) \\
				& = \lambda_\star \cdot \partial_{\lambda} \eangle{x}_{\rho_\star}\ebra{\lambda, m_1\ebra{\rho_\star}} < 0. 
			\end{aligned}
		\end{equation*}
		
		The non-degeneracy condition follows from the strict form of the GHS inequality given above, the property $m_1\ebra{\rho_\star} < m_c$, more than quadratic growth of even effective potential $\bar{V}$, combined with the central symmetry of $\langle x \rangle_{\rho_\star}$
		\begin{equation*}\label{eq.151.3}
			\begin{aligned}
				\sqrt{2}^{-1}\cdot\pi^{1/4} \cdot P_{\lambda_\star} \partial_{\delta \rho, \delta \rho}^2 F \ebra{\lambda_\star, \rho_\star} [\tilde{v}_\star, \tilde{v}_\star]/\tilde{v}_\star = \lambda_\star\cdot \eangle{\ebra{x - \eangle{x}}^3}_{\rho_\star} < 0.
			\end{aligned}
		\end{equation*}
		
		Since the occurrence of saddle-node bifurcation at the bifurcation point $\ebra{\lambda_\star, \rho\ebra{\lambda_\star, m_\star}}$ in the left well, we can determine the local existence of two invariant measures $\rho\ebra{\lambda, m_-^{(1)}}$ and $\rho\ebra{\lambda, m_-^{(2)}}$ in the left well when $\lambda \in \ebra{\lambda_\star - \epsilon, \lambda_\star}$. From the global argument, we know that there is also an invariant measure $\rho\ebra{\lambda, m_+}$ in the right well whenever $\lambda < 0$, with $m_+$ clearly different from $m_-^{(1)}, \ m_-^{(2)}$. By the concavity when $m > m_c$ and the central symmetry property, this model has at most 3 invariant measures. Thus we can determine that $\ebra{\lambda_\star, \rho\ebra{\lambda_\star, m_\star}}$ is a discontinuous phase transition point.
		
		Due to the saddle-node bifurcation at $\ebra{\lambda_\star, \rho\ebra{\lambda_\star, m_\star}}$, the system locally exhibits two invariant measures $\rho\ebra{\lambda, m_-^{(1)}}$ and $\rho\ebra{\lambda, m_-^{(2)}}$ in the left well for $\lambda \in \ebra{\lambda_\star - \epsilon, \lambda_\star}$. Globally, there also exists an invariant measure $\rho\ebra{\lambda, m_+}$ in the right well whenever $\lambda < 0$, with $m_+$ being clearly distinct from $m_-^{(1)}$ and $m_-^{(2)}$. Combined with the concavity property, which restricts the model to at most three invariant measures, we establish that $\ebra{\lambda_\star, \rho\ebra{\lambda_\star, m_\star}}$ constitutes a discontinuous phase transition point.
		
		\underline{The global argument.} From the concavity of $\eangle{x}_\rho$ on $m > m_c = -2/27$, its symmetry about $m = m_c$, the identity $\eangle{x}_\rho\ebra{\lambda, m_c} = 1/3$, and the growth estimate $\eangle{x}_\rho \sim m^{1/3}$ as $m \to \infty$, we deduce that, for any $\lambda < 0$, the graph of $\eangle{x}_\rho\ebra{\lambda, \cdot}$ intersects the graph of $id$ exactly once on $\ebra{m_c, +\infty}$ and at most twice on $\ebra{-\infty, m_c}$. The two inequalities, combined with the definition of $l_\lambda$, imply (\ref{eq.var.1}) and (\ref{eq.var.2}). 
		
		By leveraging the concavity of $\eangle{x}_\rho$ for $m > m_c = -2/27$, its symmetry about $m = m_c$, the relation $\eangle{x}_\rho\ebra{\lambda, m_c} = 1/3$, the growth bound $\eangle{x}_\rho \sim m^{1/3}$ as $m \to \infty$, and the monotonicity inequality in Appendix \ref{numer}, it follows that the graph of $\eangle{x}_\rho\ebra{\lambda, \cdot}$ intersects the diagonal line exactly once on $[m_c, +\infty)$ and at most twice on $(-\infty, m_c]$ for any $\lambda < 0$. These observations, combined with the GHS inequality, the monotonicity inequality with respect to temperature, and the definition of $l_\lambda$, yield \eqref{eq.var.1} and \eqref{eq.var.2}. 
	\end{proof}
	
	\begin{remark}\label{rem_asym_double_wells}
		Consider the problem
		\begin{align}\label{eq.144}
				& V_s = \frac{1}{4}x^4 - \frac{s}{3}x^3 - \frac{1}{2}x^2, \qquad \rho_s\ebra{\lambda, m} = \frac{1}{Z\ebra{\lambda, m}}\cdot \exp\ebra{\lambda\cdot \ebra{V_s + \frac{1}{2}x^2 - mx}}, \ s\in [0, 1]; \notag \\
				& \ebra{\lambda_{\star}^{(s)}, m_{\star}^{(s)}} \text{ solves} \begin{cases}
					& m = \mathcal{T}_1^{(s)}\ebra{\lambda, m} := \int_{\mathbb{R}} x \cdot \rho_s\ebra{\lambda, m} \dif x, \\
					& \lambda = \mathcal{T}_2^{(s)}\ebra{\lambda, m} := -\left(\int_{\mathbb{R}} x^2 \cdot \rho_{s}\ebra{\lambda, m} \dif x - \left(\int_{\mathbb{R}} x \cdot \rho_s\ebra{\lambda, m} \dif x\right)^2\right)^{-1}.
				\end{cases}
		\end{align}
		Note that $\mathcal{T}^{(s)} := \ebra{\mathcal{T}^{(s)}_1, \mathcal{T}_2^{(s)}}$ is obviously a homotopy as $s\in[0,1]$. For $s = 0$, Dawson \cite{DG89} established the uniqueness of the phase transition point, which reduces to the uniqueness of the fixed point of $\ebra{\mathcal{T}^{(s)}_1, \mathcal{T}_2^{(s)}}$. Owing to the homotopy continuation algorithm of Kellogg-Li-Yorke \cite{KLY76}, the continuity method yields an efficient constructive numerical scheme for locating the phase transition point in the discontinuous regime. The code 
		for the homotopy continuation method is available at \url{https://github.com/JunlangHu/MVPTVBT_codes}. 
		
	\end{remark}
	
	\begin{proof}[Proof of Example \ref{ex.2}]\label{proof.16}
		We first show that $\sigma_\star^{(1)} > \sigma_\star^{(2)}$. To this end, consider the following second-order moment problem
		\begin{equation}\label{eq.ex.2.1}
			\frac{\sigma\ebra{a}^2}{2\kappa} = \int_{\mathbb{R}} x^2 \rho_a\ebra{x} \dif x, \qquad \rho_a \ebra{x} = \exp(-\frac{2}{\sigma\ebra{a}^2} \cdot \ebra{a\cdot \ebra{\frac{x^4}{4} - \frac{x^2}{2}} + \kappa \cdot \frac{x^2}{2}}) \cdot Z\ebra{a}^{-1}.
		\end{equation}
		We show that the critical temperature $\sigma\ebra{a}$ is monotonically increasing in $a$ for all $a > 0$. To see this, observe that
		\begin{equation*}\label{eq.ex.2.2}
			\partial_a \rho_a = \rho_a \cdot m_2^{-1}\ebra{\rho_a} \cdot (Q - \int_{\mathbb{R}}Q \cdot \rho \dif x),
		\end{equation*}
		where
		$$Q = \frac{2 \sigma^\prime\ebra{a}}{\kappa \sigma\ebra{a}} \cdot (\frac{ax^4}{4} + \frac{\ebra{\kappa - a}x^2}{2}) - \kappa^{-1}\cdot \ebra{\frac{x^4}{4} - \frac{x^2}{2}}. $$
		Differentiating both sides of \eqref{eq.ex.2.1} and using
		$$m_4\ebra{\rho_a} = m_2\ebra{\rho_a}, \qquad m_6\ebra{\rho_a} = \frac{3\kappa}{a} m_2\ebra{\rho_a}^2 + (1 - \frac{\kappa}{a}) m_2\ebra{\rho_a}, $$
		we obtain
		\begin{equation*}\label{eq.ex.2.3}
			\sigma^\prime\ebra{a} = \frac{\sigma\ebra{a}\cdot \ebra{a + \kappa - (3\kappa + a) m_2\ebra{\rho_a}}}{2a \cdot \ebra{a - \kappa + (3\kappa - a)m_2\ebra{\rho_a}}}.
		\end{equation*}
		Note that by the Hankel inequalities, the critical invariant measure satisfies
		\begin{equation*}\label{eq.ex.2.4}
			\Delta_4 = 2 \kappa^2 / a^2 \cdot m_2\ebra{\rho_a}^5 \cdot \ebra{3m_2\ebra{\rho_a} - 1} \cdot ((1 - m_2\ebra{\rho_a}) + \kappa / a \cdot \ebra{1 - 3 m_2\ebra{\rho_a}}) > 0,
		\end{equation*}
		which implies
		\begin{equation}\label{eq.ex.2.5}
			\frac{1}{3} < m_2\ebra{\rho_a} < \frac{a + \kappa}{a + 3 \kappa}.
		\end{equation}
		Equation (\ref{eq.ex.2.5}) implies that $\sigma^\prime\ebra{a} > 0$ for all  $a > 0$.
		
		\underline{The local argument.} For the first phase transition as the temperature decreases from high to low, choose the trivial solution family $\{\rho_{0}\ebra{\lambda, \ebra{0, 0}}\}$. Since $\rho_0$ is a product measure, the covariance matrix is
		\begin{equation*}\label{eq.135}
			\text{Cov}\ebra{\rho_0\ebra{\lambda, \ebra{0, 0}}} = \text{diag}\{\eangle{x_1^2}_{\rho_0^{(1)}}, \eangle{x_2^2}_{\rho_0^{(2)}}\},
		\end{equation*}
		with eigenvalues $\eangle{x_1^2}_{\rho_0^{(1)}}$ and $\eangle{x_2^2}_{\rho_{0}^{(2)}}$ and eigenvectors $e_1$ and $e_2$ respectively. As the temperature drops, Dawson's criterion is first hit by the eigenvalue in $x_1$ direction, along which the well is deeper.  We can also check that the transversality, pitchfork, and the non-degenerate condition are the same as the $1$-dimensional case. Then by Theorem \ref{thm.7}, we deduce that a pitchfork bifurcation happens in the $e_1$ direction.
		
		For the second-time phase transitions, we have shown in Example \ref{ex.1} that for $0 < \sigma < \sigma_\star^{(1)}$
		\begin{equation}\label{eq.136}
			\begin{aligned}
				& \eangle{x_1^2}_{\rho_0\ebra{\lambda, \ebra{0, 0}}} = \eangle{x_1^2}_{\rho_0^{(1)}\ebra{\lambda, 0}} > \frac{\sigma^2}{2\kappa}, \\
				& \ebra{\eangle{x_1^2} - \eangle{x_1}^2}_{\rho_{0}\ebra{\lambda, \ebra{m^{(1)}\ebra{\lambda}, 0}}} = \ebra{\eangle{x_1^2} - \eangle{x_1}^2}_{\rho_0^{(1)}\ebra{\lambda, m^{(1)}\ebra{\lambda}}} < \frac{\sigma^2}{2\kappa}.
			\end{aligned}
		\end{equation}
		Choose $\{\rho_0\ebra{\lambda, \ebra{0, 0}}\}$ and $\{\rho_0 \ebra{\lambda, (m^{(1)}\ebra{\lambda}, 0)}\}$ as the trivial solution families, respectively. Their covariance matrices are
		\begin{equation*}\label{eq.137}
			\begin{aligned}
				& \text{Cov}\ebra{\rho_0\ebra{\lambda, (0,0)}} = \text{diag}\{\eangle{x_1^2}_{\rho_0^{(1)}\ebra{\lambda, 0}}, \eangle{x_2^2}_{\rho_0^{(2)}\ebra{\lambda, 0}}\},\\
				& \text{Cov}\ebra{\rho_0\ebra{\lambda, \ebra{m^{(1)}\ebra{\lambda}, 0}}} = \text{diag}\{\ebra{\eangle{x_1^2} - \eangle{x_1}}^2_{\rho_0^{(1)}\ebra{\lambda, m^{(1)}\ebra{\lambda}}}, \eangle{x_2^2}_{\rho_0^{(2)}\ebra{\lambda, 0}}\}.
			\end{aligned}
		\end{equation*}
		Inequalities in (\ref{eq.136}) exclude phase transitions in the $x_1$ direction. For either of these two trivial solutions, when the temperature drops to $\sigma_{\star}^{(2)}$, one can check that the system exhibits a pitchfork bifurcation along the $x_2$ direction. The transversality, pitchfork, and non-degenerate conditions are the same as the $1$-dimensional case.
		
		\underline{The global argument.} This can be checked by direct verification.
	\end{proof}
	
	\begin{proof}[Proof of Example \ref{ex.5}]\label{proof.cubic_inter}
		By a process similar to that in Lemma \ref{thm.1} and Appendix \ref{appendix.D}, we represent the operator $K$ in the transformed Hermite basis as
		\begin{equation*}\label{eq.137.2}
			\begin{pmatrix}
				-\frac{1}{2}m_1 - \frac{1}{3}m_3 & \sqrt{2}m_2^2 - \frac{4}{3}\sqrt{2}m_1 m_3 - \frac{\sqrt{2}}{2}m_2 & k_{1, 3} & k_{1, 4}\\
				\frac{\sqrt{2}}{2}m_2 + \frac{\sqrt{2}}{4} & m_3 + \frac{1}{2}m_1 & m_4 - \frac{1}{4} & k_{2, 4}\\
				-\frac{\sqrt{2}}{2} m_1 & -m_2 & -m_3 + \frac{1}{2}m_1 & -\frac{\sqrt{6}}{3}m_4 + \frac{\sqrt{6}}{2} m_2\\
				\frac{\sqrt{3}}{6} & \frac{\sqrt{6}}{6} m_1 & \frac{\sqrt{6}}{6} m_2 - \frac{\sqrt{6}}{12} & \frac{1}{3}m_3 - \frac{1}{2} m_1
			\end{pmatrix}(\rho_\star)
		\end{equation*}
		where
		\begin{equation*}\label{eq.137.3}
			\begin{aligned}
				k_{1, 3} & = \frac{2}{3}\sqrt{2} m_2 m_3 - \sqrt{2} m_1 m_4 - \frac{\sqrt{2}}{3}m_3 + \frac{\sqrt{2}}{4}m_1, \\
				k_{1, 4} & = -\frac{2}{9}\sqrt{3} m_3^2 + \frac{2}{3}\sqrt{3}m_2 m_4 - \frac{2}{3}\sqrt{3} m_1 m_5 - \frac{\sqrt{3}}{3}m_4 - \sqrt{3} m_2^2 + \frac{4}{3}\sqrt{3} m_1 m_3 + \frac{\sqrt{3}}{2} m_2, \\
				k_{2, 4} & = \frac{\sqrt{6}}{3} m_5 + \frac{\sqrt{6}}{6} m_3 - \frac{\sqrt{6}}{2} m_3 - \frac{\sqrt{6}}{4} m_1.
			\end{aligned}
		\end{equation*}
		Its characteristic polynomial is computed as
		\begin{equation*}\label{eq.137.4}
			\ebra{\theta^2 + (-m_1^2m_4 + 2 m_1 m_2 m_3 - m_2^3 + m_2m_4 - m_3^2)} \cdot \theta^2 = (\theta^2 + \Delta_2) \cdot \theta^2.
		\end{equation*}
		Note that $\Delta_2 \geq 0$ by Hankel's inequality. Thus $K$ has no real spectrum other than $0$, and there is no phase transition.
	\end{proof}
	
	\begin{proof}[Proof of Example \ref{ex.6}]\label{proof.18}
		Proof of (I). By the boundedness of $W$, for any sequence $\{\rho_l\}$ converging to $\rho$ in the $L^1$ norm, we have for each $x \in \mathbb{R}$
		\begin{equation*}\label{eq.pr.ex.6.1}
			W * \rho_l \ebra{x} = \int_{\mathbb{R}} -\exp(-(x-y)^2) \cdot \rho_l\ebra{y} \dif y \to W * \rho \ebra{x}.
		\end{equation*}
		Under the assumptions, we have
		\begin{equation}\label{eq.pr.ex.6.2}
			\begin{aligned}
				& V + \kappa \cdot W * \rho_l \geq C_1 |x|^{2+\delta} - C_2  - \sup_l \norm{\rho_l}_{L^1} \quad (C_1 > 0), \\
				& V + \kappa \cdot W * \rho_l \leq C_1^\prime |x|^{p+1} + C_2^\prime + \sup_l \norm{\rho_l}_{L^1}.
			\end{aligned}
		\end{equation}
		Consequently, by the dominated convergence theorem, it follows that as $l \to \infty$,
		\begin{equation*}\label{eq.pr.ex.6.3}
			Z\ebra{\lambda, \rho_l} \to Z\ebra{\lambda, \rho}, \qquad \norm{\mathcal{T} \rho_l - \mathcal{T}\rho}_{L^1} \to 0.
		\end{equation*}
		Furthermore, if $\{\rho_l^\prime\}$ is a bounded sequence in $L^1(\mathbb{R})$, then its image $\{|\cdot|^q \mathcal{T}\rho_l^\prime\}$ $(q \geq 1)$ remains bounded in $L^1(\mathbb{R})$. Proceeding as in the proof of Lemma \ref{thm.1}, we conclude that the operator $\mathcal{T}$ is completely continuous.
		
		Next, we show the existence of a family of symmetric invariant measures using Tikhonov's fixed-point theorem, following the approach in \cite[Theorem 3.4]{Sch86}. Let $\mathcal{M}(\mathbb{R})$ be the space of Radon measures on $\mathbb{R}$. For a sufficiently large constant $N_0$, we define the set
		$$\mathcal{N}:=\{\mu \in \mathcal{P}\ebra{\mathbb{R}}: \mu\ebra{
			-A} = \mu\ebra{A} \text{ for any Borel set $A$}, \ \int_{\mathbb{R}} |x|^2 \cdot \mu(\dif x) \leq N_0\}. $$
		
		Since $\mathcal{M}(\mathbb{R})$ equipped with the weak topology is a locally convex topological vector space and $\mathcal{T}$ is weakly continuous, Tikhonov's fixed-point theorem applies. By the symmetry of $V$ and the definition of $\mathcal{N}$, the set $\mathcal{N}$ is a compact convex subset of $\mathcal{M}(\mathbb{R})$ satisfying $\mathcal{T}\mathcal{N} \subset \mathcal{N}$. Consequently, Tikhonov's fixed-point theorem ensures the existence of a symmetric invariant measure in $\mathcal{N}$. Finally, assertion (II) follows directly from Theorem \ref{thm.7}. 
	\end{proof}
	
	\appendix
	\addcontentsline{toc}{section}{Appendices} 
	\addtocontents{toc}{\protect\setcounter{tocdepth}{0}}
	\section{A coercivity estimate under (H1) and (H3)}\label{appe.1}
	We estimate the lower growth bound for $V$ near infinity:
	\begin{lemma}\label{lem.1}
		Under assumptions (H1) and (H3), $V$ is coercive with at least quadratic growth; namely, for $C = C\ebra{\alpha, L, V, \delta}$
		\begin{equation}\label{eq.20}
			V\ebra{x} \geq \alpha/\ebra{4 + 2\delta} \cdot |x|^{2+\delta} - C.
		\end{equation}
	\end{lemma}
	\begin{proof}\label{proof.1}
		The one-sided Lipschitz condition implies that the negative part of $V$ grows at most quadratically and is therefore bounded on bounded regions. Weak dissipation implies blow-up near infinity. For small $\varepsilon > 0$, we calculate
		\begin{equation*}\label{eq.19}
			\begin{aligned}
				V\ebra{x} - V\ebra{0} & = \left(\int_0^{\varepsilon} + \int_{\varepsilon}^1 \right) \text{d} V\ebra{tx} / \text{d} t \dif t \\
				& = \int_0^{\varepsilon} \eangle{tx - 0, \nabla V\ebra{tx} - \nabla V\ebra{0}} / t \dif t + \int_0^{\varepsilon} \eangle{tx, \nabla V\ebra{0}} / t \dif t + \int_{\varepsilon}^1 \eangle{tx, \nabla V\ebra{tx}} / t \dif t \\
				& \geq \int_0^{\varepsilon} -L \cdot |tx|^2/t \dif t + \varepsilon \cdot \eangle{x, \nabla V\ebra{0}} + \int_{\varepsilon}^1 \ebra{\alpha \cdot |tx|^{2+\delta} - \beta} / t \dif t\\
				&  \geq \ebra{\alpha/\ebra{2+\delta}\cdot\ebra{1 - \epsilon^{2+\delta}} - L/2\cdot\epsilon^2} \cdot |x|^{2+\delta} + \varepsilon \cdot \eangle{x, \nabla V\ebra{0}} + \beta \cdot \log \varepsilon - L \epsilon^2/2.
			\end{aligned}
		\end{equation*}
		Setting $0 < \varepsilon < \min\{1, \ebra{\alpha / \ebra{2\alpha + \ebra{2+\delta}L}}^{1/2}\}$ gives the desired result.
	\end{proof}
	
	\section{An introduction to abstract bifurcation theorems} \label{bifur}
	
	\subsection{Global pitchfork bifurcation}
	Assume that we have a bifurcation equation describing fixed points of the equation under investigation on an abstract Banach space $X$ in the following specific form:
	\begin{equation}\label{eq.29}
		\begin{aligned}
			& F\ebra{\lambda, x} = 0, \quad J := B_{\delta_1}\ebra{\lambda_0},\\
			& F: \cup_{\lambda \in J} \{\lambda\} \times B_{\delta_2}\ebra{x_0 \ebra{\lambda}} =: \Omega \subset \mathbb{R} \times X \to X,
		\end{aligned}
	\end{equation}
	where $x_0\ebra{\lambda}$ is a known solution for any parameter $\lambda \in J$. We are interested in the critical value $\lambda_0$ such that the number of solutions changes. In that case, the number of solution curves generated from the set $\{\ebra{\lambda, x} \in \Omega: F\ebra{\lambda, x} = 0\}$ ``bifurcate at $\lambda_0$ from $x_0$" in the sense that they become multiple on one side of $\lambda_0$, and unique (exactly the one corresponding to $x_0$) on the other side. The finite-dimensional normal form of pitchfork bifurcation is $\ebra{\lambda - \lambda_0}s + s^3 = 0$. After the transformation $\tilde{F} \ebra{\lambda, \delta x} := F\ebra{\lambda, x_0\ebra{\lambda} + \delta x}$, $\tilde{F}$ depends on $\lambda$ only via the first variable along the trivial solution.
	
	
	We impose the following assumptions on the structure of $F$:
	\textit{\begin{equation}\label{eq.30}
			\begin{aligned}
				& F \in C^1\ebra{\Omega, X} \text{, $x_0\ebra{\lambda}$ is a solution for any parameter $\lambda \in J$, }\\
				& \tilde{F} \ebra{\lambda, \delta x} = \ebra{id - \lambda K_{\lambda}} \delta x + G\ebra{\lambda, \delta x} \text{, where } \tilde{F}\ebra{\lambda, \delta x} = F\ebra{\lambda, x_0\ebra{\lambda} + \delta x}. \\
				& \text{The linear part of $\tilde{F}$ is } id - \lambda K_\lambda, \text{ where $K_\lambda$ is a compact operator for any $\lambda \in J$}. \\
				& \text{The completely continuous higher order term } G = o\ebra{\norm{\delta x}} \text{ uniformly in } \lambda \in J.
			\end{aligned}
	\end{equation}}
	
	Combining results in \cite{CR71, CR73}, \cite[Theorem I.5.1, Chapter I.6, Theorem II.3.3]{Kie12}, and \cite[Chapter 10]{Dei85}, we have the following theorem.

	\begin{theorem}\label{thm.5}
		\rm{(The abstract pitchfork bifurcation theorem due to Crandall-Rabinowitz). }
		\begin{enumerate}
			\item Assume $F\ebra{\lambda, x_0} = 0$ for $\lambda \in J$, $dim\ null \ebra{\partial_x F\ebra{\lambda_0, x_0}} = dim\ null \ebra{\partial_{\delta x} \tilde{F}\ebra{\lambda_0, 0}} = 1$, and $F\in C^3\ebra{\Omega, X}$ and $\partial_x F\ebra{\lambda_0, x_0}$ is a Fredholm operator of index zero. Assume also the transversality condition
			\begin{equation}\label{eq.33}
				\partial_{\lambda, \delta x}^2 \tilde{F}\ebra{\lambda_0, 0} \tilde{v}_0 \notin range \ebra{\partial_{\delta x} \tilde{F}\ebra{\lambda_0, 0}},
			\end{equation}
			where $\tilde{v}_0$ is the unit vector in the one-dimensional space $null\ebra{\partial_x F\ebra{\lambda_0, x_0}} = null\ebra{\partial_{\delta x} \tilde{F}\ebra{\lambda_0, 0}}$. Then there exists a non-trivial continuously differentiable solution curve through $\ebra{\lambda_0, x_0}$, parametrized by a single parameter $s$:
			\begin{equation*}\label{eq.34}
				\{\ebra{\lambda\ebra{s}, x_0 + s \tilde{v}_0 + s \tilde{w}\ebra{s}}: s \in B_{\delta_1}\ebra{0}, \lambda\ebra{0} = \lambda_0, \tilde{w}\ebra{s} = w\ebra{\lambda\ebra{s}, s}\}.
			\end{equation*}
			
			\item When $0$ is a simple eigenvalue (the algebraic and geometric multiplicities both equal $1$), let $P$ be the projection onto the first summand in the orthogonal decomposition $X = null\ebra{id - \lambda_0 K} \oplus range\ebra{id - \lambda_0 K}$. Then the bifurcation formula for the zero solution is
			\begin{equation*}\label{eq.35}
				\lambda^\prime\ebra{0} = - 1/2 \cdot P \partial_{\delta x, \delta x}^2 \tilde{F}\ebra{\lambda_0, 0} [\tilde{v}_0, \tilde{v}_0] / P \partial_{\lambda, \delta x}^2 \tilde{F}\ebra{\lambda_0, 0}\tilde{v}_0,
			\end{equation*}
			and the reduced bifurcation equation can be expanded as
			\begin{equation*}\label{eq.35.1}
				P \partial_{\lambda, \delta x}^2 \tilde{F}\ebra{\lambda_0, 0} \tilde{v}_0 \cdot \ebra{\lambda - \lambda_0}s + \left(1/2\cdot P\partial_{\delta x, \delta x}^2 \tilde{F}\ebra{\lambda_0, 0} [\tilde{v}_0, \tilde{v}_0]\right) \cdot s^2 + \cdots = 0.
			\end{equation*}
			A transcritical bifurcation occurs if $\lambda^\prime\ebra{0} \neq 0$.
			
			In the $\lambda^\prime\ebra{0} = 0$ case, namely we assume
			\begin{equation}\label{eq.35.2}
				\partial_{\delta x, \delta x}^2 \tilde{F}\ebra{\lambda_0, 0} [\tilde{v}_0, \tilde{v}_0] \in range \ebra{\partial_{\delta x} \tilde{F}\ebra{\lambda_0, 0}},
			\end{equation}
			the sign of the second-order derivative
			\begin{equation*}\label{eq.36}
				\begin{aligned}
					\lambda^{\prime \prime}\ebra{0} & = \ebra{ -1/3 \cdot P \partial_{\delta x, \delta x, \delta x}^3 \tilde{F}\ebra{\lambda_0, 0} [\tilde{v}_0, \tilde{v}_0, \tilde{v}_0] \\
						& \quad  - P \partial_{\delta x, \delta x}^2 \tilde{F}\ebra{\lambda_0, 0} [\tilde{v}_0, 2\partial_s w\ebra{\lambda_0, 0}]} / P \partial_{\lambda, \delta x}^2 \tilde{F}\ebra{\lambda_0, 0}\tilde{v}_0, \\
				\end{aligned}
			\end{equation*}
			with
			$$
			\partial_s w\ebra{\lambda_0, 0}   =  -1/2 \cdot \ebra{\partial_{\delta x} \tilde{F}\ebra{\lambda_0, 0}}^{-1} \ebra{id - P} \partial_{\delta x, \delta x}^2 \tilde{F}\ebra{\lambda_0, 0} [\tilde{v}_0, \tilde{v}_0],
			$$
			dominates the type of bifurcation. If the non-degenerate condition $\lambda^{\prime \prime}\ebra{0} \neq 0$ is satisfied, which is
			\begin{equation}\label{eq.36.1}
				\begin{aligned}
					& 1/3\cdot \partial_{\delta x, \delta x, \delta x}^3 \tilde{F}\ebra{\lambda_0, 0}[\tilde{v}_0, \tilde{v}_0, \tilde{v}_0]  \\
					& + \partial_{\delta x, \delta x}^2 \tilde{F} \ebra{\lambda_0, 0} [\tilde{v}_0, 2\partial_s w\ebra{\lambda_0, 0}] \notin \text{range} \ebra{\partial_{\delta x} \tilde{F}\ebra{\lambda_0, 0}},
				\end{aligned}
			\end{equation}
			then a pitchfork bifurcation occurs. The reduced bifurcation equation can be further expanded as
			\begin{equation*}\label{eq.36.1.1}
				\begin{aligned}
					& \left( P \partial_{\lambda, \delta x}^2 \tilde{F}_0 \ebra{\tilde{v}_0}\right) \cdot \ebra{\lambda - \lambda_0}s + \left(1/2\cdot P\partial_{\delta x, \delta x}^2 \tilde{F}\ebra{\lambda_0, 0} [\tilde{v}_0, 2\partial_s w_0]\right. \\
					& \left. + 1/6 \cdot P\partial_{\delta x, \delta x, \delta x}^3 \tilde{F}\ebra{\lambda_0, 0}[\tilde{v}_0, \tilde{v}_0, \tilde{v}_0]\right) \cdot s^3 + \cdots = 0,
				\end{aligned}
			\end{equation*}
			and $\lambda^{\prime\prime} \ebra{0} > 0$ corresponds to supercritical bifurcation, while $\lambda^{\prime\prime} \ebra{0} < 0$ corresponds to subcritical bifurcation.
			
			In particular, if $\tilde{F}\ebra{\lambda, \delta x} = \ebra{id - \lambda K_\lambda} \ebra{\delta x} + G\ebra{\lambda, \delta x}$ under assumption \eqref{eq.30}, then $\partial_{\lambda, \delta x}^2 \tilde{F} \ebra{\lambda_0, 0} \tilde{v}_0 = -\partial_\lambda \ebra{\lambda K_\lambda}\tilde{v}_0$, and the second-order mixed partial derivative of $G$ vanishes.
			
			\item The non-zero trivial-solution case can be treated using the chain rule:
			\begin{equation*}\label{eq.36.2}
				\begin{aligned}
					& \partial_{\delta x} \tilde{F} \ebra{\lambda_0, 0} = \partial_x F\ebra{\lambda_0, x_0}, \quad \partial_{\delta x, \delta x}^2 \tilde{F} \ebra{\lambda_0, 0} = \partial_{x, x}^2 F \ebra{\lambda_0, x_0},\\
					& \partial_{\lambda, \delta x}^2 \tilde{F} \ebra{\lambda_0, 0}\tilde{v}_0 = \partial_{\lambda, x}^2 F \ebra{\lambda_0, x_0} \tilde{v}_0 + \partial_{x, x}^2 F \ebra{\lambda_0, x_0}[\partial_\lambda x_0, \tilde{v}_0], \\
					& \partial_{\delta x, \delta x, \delta x}^3 \tilde{F}\ebra{\lambda_0, 0} = \partial_{x, x, x}^3 F \ebra{\lambda_0, x_0}.
				\end{aligned}
			\end{equation*}
			
			\item Moreover, we have the global argument: the nontrivial solution component $C$ crossing over $(\lambda_0, x_0)$ satisfies the dichotomy in Rabinowitz's global bifurcation theorem (see \cite[Theorem II.3.3]{Kie12}). $C$ consists of two subcontinua $C^+$ and $C^-$ such that $C^+ \cap C^- \cap B_\epsilon\ebra{\lambda_0, 0} = \{\ebra{\lambda_0, 0}\}$ and $C^{+, -} \cap \partial B_\epsilon\ebra{\lambda_0, 0} \neq \emptyset$ for all sufficiently small $\epsilon > 0$.
		\end{enumerate}
	\end{theorem}
	
	\subsection{Local saddle-node bifurcation}
	A saddle-node bifurcation describes the phenomenon in which two equilibria collide and disappear, or appear from nothing. It may correspond to a discontinuous phase transition. A finite-dimensional normal form of a saddle-node bifurcation is $\ebra{\lambda - \lambda_0} - s^2 = 0$. \textit{Note that $\bar{x}_0$ is no longer a solution for all admissible $\lambda$ in this case.} We replace the preceding assumptions by the following saddle-node assumptions:
	\begin{equation}\label{eq.47}
		\begin{aligned}
			& F \in C^2\ebra{\Omega, X} \text{, with } F\ebra{\lambda_0, \bar{x}_0} = 0 \text{ for some } \ebra{\lambda_0, \bar{x}_0} \in \Omega; \\
			& 0 \text{ is a simple eigenvalue of the Fredholm operator $\partial_x F\ebra{\lambda_0, \bar{x}_0}$ with index 0. }
		\end{aligned}
	\end{equation}
	
	\begin{theorem}[Saddle-node Bifurcation Theorem]\label{thm.6}
		Besides (\ref{eq.47}), assume additionally the transversality condition
		\begin{equation}\label{eq.48}
			\partial_{\lambda} F\ebra{\lambda_0, \bar{x}_0} \notin range \ebra{\partial_x F\ebra{\lambda_0, \bar{x}_0}}.
		\end{equation}
		Then
		\begin{enumerate}
			\item There is a continuously differentiable local solution curve through $\ebra{\lambda_0, \bar{x}_0}$ that locally contains all solutions of $F\ebra{\lambda, x} = 0$:
			\begin{equation*}\label{eq.49}
				\{\ebra{\lambda\ebra{s}, \bar{x}_0 + s \tilde{v}_0 + s \tilde{w}\ebra{s}}: s \in B_{\delta_1}\ebra{0}, \lambda\ebra{0} = \lambda_0, \tilde{w}\ebra{s} = w\ebra{\lambda\ebra{s}, s}\}.
			\end{equation*}
			\item The following bifurcation formula holds
			\begin{equation*}\label{eq.50}
				\lambda^\prime\ebra{0} = 0, \quad \lambda^{\prime\prime}\ebra{0} = - P\partial_{xx}^2 F\ebra{\lambda_0, \bar{x}_0}[\tilde{v}_0, \tilde{v}_0] / P\partial_{\lambda} F\ebra{\lambda_0, \bar{x}_0},
			\end{equation*}
			where $P$ denotes the projection onto the first space of the orthogonal decomposition $X = null\ebra{\partial_x F_0} \oplus range \ebra{\partial_x F_0}$, as before. Moreover, we have the following second-order expansion of the bifurcation equation
			\begin{equation*}\label{eq.51}
				P \partial_\lambda F\ebra{\lambda_0, \bar{x}_0} \cdot \ebra{\lambda - \lambda_0} + \left(1/2 \cdot P \partial_{xx}^2 F\ebra{\lambda_0, \bar{x}_0}[\tilde{v}_0, \tilde{v}_0]\right)\cdot s^2 + \cdots = 0.
			\end{equation*}
			\item Moreover, if the non-degeneracy condition $\lambda^{\prime \prime}\ebra{0} \neq 0$ holds, i.e.
			\begin{equation}\label{eq.51.1}
				\partial_{xx}^2 F\ebra{\lambda_0, \bar{x}_0}[\tilde{v}_0, \tilde{v}_0] \notin range\ebra{\partial_x F\ebra{\lambda_0, \bar{x}_0}},
			\end{equation}
			then a saddle-node bifurcation occurs at $\ebra{\lambda_0, \bar{x}_0}$.
		\end{enumerate}
	\end{theorem}
	
	\section{Statement of Fr\'{e}chet-Kolmogorov-Riesz compactness theorem}\label{appe.FKR}
	The Fr\'{e}chet--Kolmogorov--Riesz compactness theorem \cite{HOH10} characterizes precompact sets in $L^p$ with $p \geq 1$. It is based on the Ascoli--Arzel\`a theorem and the density of continuous functions in $L^p$ spaces. It is also a crucial step in proving compact Sobolev embeddings.
	\begin{theorem}\label{thm.2}
		Assume $1 \leq p < \infty$ and $n \in \mathbb{N}$. Then a family $\mathcal{F}$ is precompact in $L^p\ebra{\mathbb{R}^n}$ if and only if the following conditions hold:
		\begin{enumerate}
			\item (Uniform boundedness) $\mathcal{F}$ is bounded in $L^p\ebra{\mathbb{R}^n}$
			\begin{equation*}\label{eq.21}
				\sup_{f \in \mathcal{F}} \norm{f}_{L^p} < \infty.
			\end{equation*}
			\item (Uniform integrability) For all $\varepsilon > 0$, there is a compact subset $K_\varepsilon \subset \mathbb{R}^n$ such that $\norm{f}_{L^p\ebra{\mathbb{R}^n \setminus K_\varepsilon}} < \varepsilon$ for all $f \in \mathcal{F}$. Or in the language of limit
			\begin{equation*}\label{eq.23}
				\lim\limits_{R \to \infty} \sup_{f \in \mathcal{F}} \norm{f}_{L^p\ebra{\mathbb{R}^n \setminus B_R\ebra{0}}} = 0.
			\end{equation*}
			\item (Uniform small oscillations) For all $\varepsilon > 0$, there is a $\delta = \delta\ebra{\varepsilon}$ such that $\norm{f\ebra{\cdot - h} - f}_p < \varepsilon$ for all $f \in \mathcal{F},\ h \in \mathbb{R}^n, \ |h| < \delta$. Or in the language of limit
			\begin{equation*}\label{eq.22}
				\lim\limits_{\delta \to 0} \sup_{|h| < \delta} \sup_{f \in \mathcal{F}} \norm{\tau_h f - f}_{L^p} = 0.
			\end{equation*}
		\end{enumerate}
	\end{theorem}
	
	\begin{remark}\label{rem.4}
		The classical $p$-order ($p \geq 1, p \in \mathbb{N}$) uniform integrability reads
		\begin{equation*}\label{eq.24}
			\sup_{f \in \mathcal{F}} \int_{\mathbb{R}^n \setminus B_R\ebra{0}} |x|^p f\ebra{x} \dif x = \sup_{X_\#\mathbb{P} = f \mathcal{L}^n, f \in \mathcal{F}} \int_{|X| > R} |X|^p \dif \mathbb{P} < \infty.
		\end{equation*}
		By analogy, we call (ii) ``zero-order" uniform integrability. A similar property of $p$-uniform integrability also holds for (ii). 
	\end{remark}
	
	\section{A detailed proof for Lemma \ref{lemma.2} }\label{appendix.D}
	\textit{An introduction to Hermite orthogonal polynomials. } In the one-dimensional case, the classical Hermite polynomial system $\{h_k\}_{k \in \mathbb{N}}$ and its normalized version $\{\tilde{h}_k\}_{k \in\mathbb{N}}$ are defined by
	\begin{equation*}\label{eq.65}
		h_k := \ebra{-1}^k \cdot \exp\ebra{x^2} \cdot \dif^k \ebra{\exp\ebra{-x^2}} / \dif x^k, \quad \tilde{h}_k := \ebra{2^k \cdot k!}^{-1/2} \cdot \pi^{-1/4}\cdot h_k.
	\end{equation*}
	The polynomials can be equivalently defined by the Taylor expansion
	\begin{equation*}\label{eq.66}
		\exp\ebra{2xt - t^2} = \sum_{k = 0}^{\infty} h_k\ebra{x} \cdot \frac{t^k}{k!}.
	\end{equation*}
	The system $\{\tilde{h}_n\}_{n \in \mathbb{N}}$ forms a normalized complete orthogonal system on $L^2\ebra{\mathbb{R}, \exp\ebra{-x^2}}$; see Proposition 1.1 in \cite{AS12}. In the higher-dimensional case, the basis is generated by tensor products of one-dimensional Hermite bases:
	\begin{equation*}\label{eq.67}
		H_\alpha \ebra{x_1, \dots, x_n} := \ebra{\otimes_{k=1}^n h_{\alpha_k}} \ebra{x_1, \dots, x_n} = \prod_{k=1}^n h_{\alpha_k}\ebra{x_k}, \quad \tilde{H}_\alpha  := \otimes_{k=1}^n \tilde{h}_{\alpha_k},
	\end{equation*}
	where $\alpha = \ebra{\alpha_1, \dots, \alpha_n} \in \mathbb{N}^n$ is a multi-index. Based on the Hermite polynomials, we can construct the bases $\{\psi_i\}$ and $\{\Psi_\alpha\}$ for the weighted space; see details in (\ref{eq.70}) and (\ref{eq.71}).
	
	\begin{proof}[Detailed proof of Lemma \ref{lemma.2}] \label{proof.9}
		We calculate the $n=1$ case in detail as a starting point and then sketch the higher-dimensional case analogously.
		
		\underline{When $n = 1$,} $K$ maps the basis $\psi_0, \psi_1, \psi_2$  to
		\begin{equation*}\label{eq.75}
			\begin{aligned}
				& K \psi_0  = \rho_0 \cdot \pi^{-1/4} \cdot \ebra{x^2/2 - x \cdot m_1\ebra{\rho_0} + m_1\ebra{\rho_0}^2 - m_2\ebra{\rho_0}/2}, \\
				& K \psi_1  = \rho_0 \cdot \pi^{-1/4} \cdot \ebra{\sqrt{2}/2m_1\ebra{\rho_0} \cdot x^2 - \sqrt{2}m_2\ebra{\rho_0} \cdot x + \sqrt{2}/2m_1\ebra{\rho_0}\cdot m_2\ebra{\rho_0}}, \\
				& K \psi_2 = \rho_0 \cdot \pi^{-1/4} \cdot \ebra{(\sqrt{2}/2 m_2\ebra{\rho_0}-\sqrt{2}/4)\cdot x^2 + (- \sqrt{2}m_3\ebra{\rho_0} + \sqrt{2}m_1(\rho_0)/2)\cdot  x \\
					& \qquad\quad + \sqrt{2}m_1\ebra{\rho_0} m_3\ebra{\rho_0} - \sqrt{2}/2\cdot m_1^2(\rho_0)- \sqrt{2}/2m_2\ebra{\rho_0}^2} + \sqrt{2}/4\cdot m_2(\rho_0).
			\end{aligned}
		\end{equation*}
		Expanding $K\psi_0, K\psi_1, K\psi_2$ in this basis, we obtain the square matrix representation of $K$ restricted to $\text{span}\{\psi_0, \psi_1, \psi_2\}$:
		\begin{equation*}\label{eq.76}
			\begin{pmatrix}
				-m_2/2 + m_1^2 + 1/4 & \sqrt{2}/4 \cdot m_1 \ebra{1 + 2m_2} & -\sqrt{2}/2\ebra{m_2 - 1/2}^2 + \sqrt{2} m_1 \ebra{m_3 - m_1/2} \\
				-\sqrt{2}m_1/2 & -m_2 & -m_3 + m_1/2 \\
				\sqrt{2}/4 & m_1/2 & m_2/2 - 1/4
			\end{pmatrix} \ebra{\rho_0}.
		\end{equation*}
		A simple linear-algebra calculation gives the characteristic polynomial
		\begin{equation*}\label{eq.77}
			\ebra{\theta + m_2\ebra{\rho_0} - m_1\ebra{\rho_0}^2} \cdot \theta^2,
		\end{equation*}
		with eigenspace $\text{span}\{-\sqrt{2}m_1\ebra{\rho_0} \cdot \psi_0 + \psi_1\}$ corresponding to the spectral value $-m_2\ebra{\rho_0} + m_1\ebra{\rho_0}^2$.
		
		\underline{The general $n \geq 1$ case.} We consider the restriction of $K$ to $$\text{span}\{\Psi_\alpha: |\alpha| \leq 2\} = \text{span}\{\Psi_0, \Psi_{e_1}, \dots, \Psi_{e_n}, \Psi_{2e_1}, \dots, \Psi_{2e_n}, \Psi_{e_1 + e_2}, \dots, \Psi_{e_{n-1} + e_n}\},$$ which is a $1 + n + n + C_n^2 = n^2/2 + 3n/2 + 1$-dimensional space. The operator $K$ maps the basis as follows (from here onward, we omit the notation ``$\ebra{\rho_0}$" in $m_\alpha\ebra{\rho_0}$ when there is no confusion):
		\begin{equation*}\label{eq.78}
			\begin{aligned}
				& K \Psi_0 = \rho_0 \cdot \pi^{-n/4} \cdot \ebra{|x|^2/2 - \ebra{m_{e_1} x_1 + \cdots + m_{e_n} x_n} - m_2\ebra{\rho_0}/2 + m_{e_1}^2 + \cdots + m_{e_n}^2}, \\
				& K\Psi_{e_i} = \rho_0 \cdot \pi^{-n/4} \cdot\ebra{\sqrt{2}/2 m_{e_i} \cdot |x|^2 - \sqrt{2}\ebra{m_{e_i + e_1} x_1 + \cdots + m_{e_i + e_n}x_n} \\
					& \qquad \qquad \qquad \qquad \quad - \sqrt{2}/2 m_{e_i}\cdot m_2 + \sqrt{2} \ebra{m_{e_i + e_1}m_{e_1} + \cdots + m_{e_i + e_n}m_{e_n}}}, \\
				& K \Psi_{2e_i} = \rho_0 \cdot \pi^{-n/4} \cdot \ebra{\ebra{\sqrt{2}m_{2 e_i}\ebra{\rho_0}/2 - \sqrt{2}/4}\cdot \ebra{|x|^2 - m_2} + \ebra{\ebra{\sqrt{2} m_{2e_i + e_1} - \sqrt{2}m_{e_1}/2} \\
						& \qquad \qquad \qquad \qquad \quad \ \cdot \ebra{m_{e_1} - x_1} + \cdots + \ebra{\sqrt{2} m_{2e_i + e_n} - \sqrt{2}m_{e_n}/2} \cdot \ebra{m_{e_n} - x_n}}}, \\
				& K \Psi_{e_i + e_j} = \rho_0 \cdot \pi^{-n/4} \cdot \ebra{m_{e_i + e_j} \cdot |x|^2 - 2\ebra{m_{e_i + e_j + e_1} x_1 + \cdots + m_{e_i + e_j + e_n} x_n} \\
					& \qquad \qquad \qquad \qquad \qquad- m_{e_i + e_j} \cdot m_2 + 2\ebra{m_{e_i + e_j + e_1}m_{e_1} + \cdots + m_{e_i + e_j + e_n}m_{e_n}}}, \ (i \neq j).
			\end{aligned}
		\end{equation*}
		
		Expanding these expressions again in the basis
		\begin{equation*}\label{eq.80}
			\begin{aligned}
				& K\Psi_0 = \ebra{n/4 - m_2/2 + m_{e_1}^2 + \cdots + m_{e_n}^2} \cdot \Psi_0 - \sqrt{2}/2\ebra{m_{e_1} \cdot \Psi_{e_1} + \cdots + m_{e_n} \cdot \Psi_{e_n}}  \\
				& \qquad \quad \ \ + \sqrt{2} / 4 \ebra{\Psi_{2e_1} + \cdots + \Psi_{2e_n}}, \\
				& K\Psi_{e_i} = \sqrt{2} \ebra{\ebra{n/4 - m_2/2} m_{e_i} + \ebra{m_{e_i + e_1}m_{e_1} + \cdots + m_{e_i + e_n}m_{e_n}}} \cdot \Psi_0 \\
				& \qquad \quad \ \ - \ebra{m_{e_i + e_1} \cdot \Psi_{e_1} + \cdots + m_{e_i + e_n} \cdot \Psi_{e_n}} + m_{e_i}/2 \cdot \ebra{\Psi_{2e_1} + \cdots + \Psi_{2e_n}}, \\
				& K \Psi_{2e_i} = \ebra{\sqrt{2}\ebra{m_{2e_i} - 1/2} \ebra{n/4 - m_2/2} + \sqrt{2}\ebra{ \ebra{m_{2e_i + e_1} - m_{e_1}/2} m_{e_1} + \cdots \\
						& \qquad \quad \ \  + \ebra{m_{2e_i + e_n} - m_{e_n}/2} m_{e_n}}} \cdot \Psi_0 + \ebra{m_{2e_i}/2 - 1/4}\cdot \ebra{\Psi_{2e_1} + \cdots + \Psi_{2e_n}} \\
				& \qquad \quad \ \ + \ebra{\ebra{-m_{2e_i + e_1} + m_{e_1}/2}\cdot \Psi_{e_1} + \cdots + \ebra{-m_{2e_i + e_n} + m_{e_n}/2} \cdot \Psi_{e_n}}, \\
				& K \Psi_{e_i + e_j} = \ebra{m_{e_i + e_j} \ebra{n/2 - m_2} + 2\ebra{m_{e_i + e_j + e_1}m_{e_1} + \cdots + m_{e_i + e_j + e_n}m_{e_n}}}\cdot \Psi_0 \quad (i \neq j)\\
				& \qquad \quad \ \ - \sqrt{2} \ebra{m_{e_i + e_j + e_1} \cdot \Psi_{e_1} + \cdots + m_{e_i + e_j + e_n} \cdot \Psi_{e_n}} + \sqrt{2}/2 m_{e_i + e_j} \cdot \ebra{\Psi_{2e_1} + \cdots + \Psi_{2e_n}}.
			\end{aligned}
		\end{equation*}
		Thus, we can write the restriction of $K$ to $\text{span}\{\Psi_\alpha: |\alpha| \leq 2\}$ in square-matrix form:
		\begin{equation*}\label{eq.81}
			\left(\begin{array}{c:ccc:ccc:ccc}
				\textbf{M}_{1,1} &  & \ebra{{\textbf{M}_{1,2}}}_{1\times n} &  &  &\ebra{\textbf{M}_{1,3}}_{1 \times n} &  & & \ebra{\textbf{M}_{1,4}}_{1 \times C_n^2} & \\ \hdashline
				&  &  &  &  &  &  & &\\
				\ebra{\textbf{M}_{2,1}}_{n \times 1} &  & \ebra{\textbf{M}_{2,2}}_{n \times n} &  &  & \ebra{\textbf{M}_{2,3}}_{n \times n} &  & & \ebra{\textbf{M}_{2,4}}_{n \times C_n^2} & \\
				&  &  &  &  &  &  & & &\\ \hdashline
				&  &  &  &  &  &  & & &\\
				\ebra{\textbf{M}_{3,1}}_{n \times 1} &  & \ebra{\textbf{M}_{3,2}}_{n \times n} &  &  & \ebra{\textbf{M}_{3,3}}_{n \times n} &  & & \ebra{\textbf{M}_{3,4}}_{n \times C_n^2} & \\
				&  &  &  &  &  &  & & & \\ \hdashline
				\textbf{O}_{C_n^2 \times 1}  &  & \textbf{O}_{C_n^2 \times n} &  &  & \textbf{O}_{C_n^2 \times n} &  & & \textbf{O}_{C_n^2 \times C_n^2} &
			\end{array}\right),
		\end{equation*}
		where $\textbf{O}$ is the zero matrix with specified size; $C_n^2 := n\ebra{n-1}/2$; $m_2 = m_{2e_1} + \cdots + m_{2e_n}$;
		\allowdisplaybreaks
		\begin{equation*}
			\begin{aligned}
				& \textbf{M}_{1,1} = \frac{n}{4} - \frac{1}{2}m_2 + m_{e_1}^2 + \cdots + m_{e_n}^2, \quad \ebra{\textbf{M}_{1,2}}_{1 \times n} = \left(-\frac{\sqrt{2}}{2} m_{e_i} \ebra{m_2 - \frac{n}{2}} + \sqrt{2} \cdot \sum_{k=1}^n m_{e_i + e_k} m_{e_k}\right), \\
				& \ebra{\textbf{M}_{1,3}}_{1 \times n} = \left(-\frac{\sqrt{2}}{2}\ebra{m_{2e_i} - \frac{1}{2}}\ebra{m_2 - \frac{n}{2}} + \sqrt{2} \cdot \sum_{k=1}^n \ebra{m_{2e_i + e_k} - \frac{1}{2}m_{e_k}} m_{e_k}\right)_{i \in \{1, \dots, n\}}, \\
				& \ebra{\textbf{M}_{1,4}}_{1 \times C_n^2} = \left(-m_{e_i + e_j}\ebra{m_2 - \frac{n}{2}} + 2\sum_{k = 1}^n m_{e_i + e_j + e_k} m_{e_k}\right)_{i, j\in \{1, \dots, n\}, i \neq j}, \\
				& \ebra{\textbf{M}_{2,1}}_{n \times 1} = \left(-\frac{\sqrt{2}}{2}m_{e_1}, \ \dots,\ -\frac{\sqrt{2}}{2}m_{e_n}\right), \qquad\quad\ \ebra{\textbf{M}_{3,1}}_{n \times 1} = \left(\frac{\sqrt{2}}{4}, \ \dots ,\ \frac{\sqrt{2}}{4}\right)_{n \times 1}, \\
				& \ebra{\textbf{M}_{2,2}}_{n \times n} = \left(\begin{array}{ccc}
					-m_{e_1 + e_1} &  \cdots & -m_{e_n + e_1}\\
					\vdots  & \vdots &  \vdots \\
					-m_{e_1 + e_n} &  \cdots & -m_{e_n + e_n}
				\end{array}\right)_{n \times n}, \ \ \
				\ebra{\textbf{M}_{3,2}}_{n \times n} = \left(\begin{array}{ccc}
					\frac{1}{2}m_{e_1} &  \cdots & \frac{1}{2}m_{e_n}\\
					\vdots  & \vdots &  \vdots \\
					\frac{1}{2}m_{e_1} &  \cdots & \frac{1}{2}m_{e_n}
				\end{array}\right)_{n \times n}, \\
			\end{aligned}
		\end{equation*}
		\begin{equation*}
			\begin{aligned}
				& \ebra{\textbf{M}_{2,3}}_{n \times n} = \left(\begin{array}{ccc}
					-m_{2e_1 + e_1} + \frac{1}{2}m_{e_1} &  \cdots & -m_{2e_n + e_1} + \frac{1}{2}m_{e_1}\\
					\vdots  & \vdots &  \vdots \\
					-m_{2e_1 + e_n} + \frac{1}{2} m_{e_n} &  \cdots & -m_{2e_n + e_n} + \frac{1}{2}m_{e_n}
				\end{array}\right)_{n \times n}, \qquad\qquad\qquad\qquad\qquad\qquad\\
				& \ebra{\textbf{M}_{3,3}}_{n \times n} =
				\left(\begin{array}{ccc}
					\frac{1}{2} m_{2e_1} - \frac{1}{4}  &  \cdots & \frac{1}{2} m_{2e_n} - \frac{1}{4} \\
					\vdots  & \vdots &  \vdots \\
					\frac{1}{2} m_{2e_1} - \frac{1}{4}  &  \cdots & \frac{1}{2} m_{2e_n} - \frac{1}{4}
				\end{array}\right)_{n \times n}, \\
			\end{aligned}
		\end{equation*}
		\begin{equation*}\label{eq.82}
			\begin{aligned}
				& \ebra{\textbf{M}_{2,4}}_{n \times C_n^2} = \left(\begin{array}{ccc}
					-\sqrt{2} m_{e_1 + e_2 + e_1} &  \cdots & -\sqrt{2}m_{e_{n-1} + e_n + e_1}\\
					\vdots  & \vdots &  \vdots \\
					-\sqrt{2}m_{e_1 + e_2 + e_n} &  \cdots & -\sqrt{2} m_{e_{n-1} + e_n + e_n}
				\end{array}\right)_{n \times C_n^2}, \qquad\qquad\qquad\qquad\qquad\qquad\ \ \\
				& \ebra{\textbf{M}_{3,4}}_{n \times C_n^2} =
				\left(\begin{array}{ccc}
					\frac{\sqrt{2}}{2} m_{e_1 + e_2}  &  \cdots & \frac{\sqrt{2}}{2} m_{e_{n-1} + e_n} \\
					\vdots  & \vdots &  \vdots \\
					\frac{\sqrt{2}}{2} m_{e_1 + e_2}  &  \cdots & \frac{\sqrt{2}}{2} m_{e_{n-1} + e_n}
				\end{array}\right)_{n \times C_n^2}.
			\end{aligned}
		\end{equation*}
		
		\underline{Computation of the characteristic polynomial:} Many terms can be eliminated. After suitable row and column transformations, the characteristic polynomial
		\begin{equation*}
			\textbf{det}\left(\theta \textbf{id}_{n \times n} - K|_{\text{span}\{\Psi_\alpha: |\alpha| \leq 2\}}\right)
		\end{equation*} becomes
		\begin{equation*}\label{eq.82.1}
			\theta^{n^2/2 - n/2} \cdot \textbf{det}\left(\begin{array}{c:ccc:ccc}
				\theta &  & \textbf{O}_{1\times n} &  &  & \textbf{O}_{1 \times n} &  \\ \hdashline
				&  &  &  &  & \\
				-\ebra{\textbf{M}_{2,1}}_{n \times 1} &  & \theta \textbf{id}_{n \times n} + \text{Cov}\ebra{\rho_0} &  &  & \widetilde{\ebra{\textbf{M}_{2,3}}}_{n \times n} &  \\
				&  &  &  &  &  & \\ \hdashline
				&  &  &  &  &  & \\
				-\ebra{\textbf{M}_{3,1}}_{n \times 1} &  & \textbf{O}_{n \times n} &  &  & \theta \textbf{id}_{n \times n} &  \\
				&  &  &  &  &  &  \\
			\end{array}\right),
		\end{equation*}
		where
		\begin{equation*}
			\widetilde{\ebra{\textbf{M}_{2,3}}}_{n \times n} = \left(\begin{array}{ccc}
				m_{2e_1 + e_1} - m_{2e_1} m_{e_1} &  \cdots & m_{2e_n + e_1} - m_{2e_n} m_{e_1}\\
				\vdots  & \vdots &  \vdots \\
				m_{2e_1 + e_n} - m_{2e_1} m_{e_n} &  \cdots & m_{2e_n + e_n} - m_{2e_n} m_{e_n}
			\end{array}\right)_{n \times n}.
		\end{equation*}
		(Indeed, this reduction is obtained by adding $\sqrt{2}m_{e_i}$ times row $\ebra{i+1}$ and $\sqrt{2}\ebra{m_{2e_i} - 1/2}$ times row $\ebra{i+n+1}$ to row $1$ ($i \in \{1, \dots, n\}$), and then adding $-\sqrt{2}m_{e_i}$ times column $1$ to column $\ebra{i+1}$ ($i \in \{1, \dots, n\}$), and $\sqrt{2}\ebra{1/2 - m_{e_i}}$ times column $1$ to column $\ebra{i+n+1}$.)
		
		Using the block determinant formula, we obtain the characteristic polynomial in the form
		\begin{equation*}\label{eq.86}
			\theta^{n^2/2 + n/2 + 1} \cdot \textbf{det}\ebra{\theta \textbf{id}_{n \times n} + \text{Cov}\ebra{\rho_0}},
		\end{equation*}
		which implies that the non-zero eigenvalues are the negatives of those of $\mathrm{Cov}(\rho_0)$ defined above.
		
		\underline{Computation of the eigenspaces:} We solve the linear equation for the eigenvectors $\tilde{v}$:
		\begin{equation*}
			\ebra{\theta \textbf{id} - K|_{\text{span}\{\Psi_\alpha: |\alpha| \leq 2\}}} \tilde{v} = 0. \qquad (\theta \neq 0)
		\end{equation*}
		This can be done by classical Gaussian elimination. After several row transformations, the eigen-matrix becomes (we omit the details)
		\begin{equation}\label{eq.86.1}
			\left(\begin{array}{c:ccc:ccc:ccc}
				0 &  & \textbf{O}_{1\times n} &  &  & 2\sqrt{2}\theta^2/n \cdot \textbf{id}_{1 \times n} &  & & \cdots & \\ \hdashline
				&  &  &  &  &  &  & &\\
				\textbf{O}_{n \times 1} &  & \theta \textbf{id}_{n \times n} + \text{Cov}\ebra{\rho_0} &  &  & 2\theta \cdot \text{diag}\{\ebra{m_{e_i}}_{i \in \{1, \dots, n\}}\} + \widetilde{\ebra{\textbf{M}_{2,3}}}_{n \times n} &  & & \cdots & \\
				&  &  &  &  &  &  & & &\\ \hdashline
				&  &  &  &  &  &  & & &\\
				\widetilde{\ebra{\textbf{M}_{3,1}}}_{n \times 1} &  & \widetilde{\ebra{\textbf{M}_{3,2}}}_{n \times n} &  &  & \widetilde{\ebra{\textbf{M}_{3,3}}}_{n \times n} &  & & \cdots & \\
				&  &  &  &  &  &  & & & \\ \hdashline
				\textbf{O}_{C_n^2 \times 1}  &  & \textbf{O}_{C_n^2 \times n} &  &  & \textbf{O}_{C_n^2 \times n} &  & & \textbf{O}_{C_n^2 \times C_n^2} &
			\end{array}\right),
		\end{equation}
		where
		\begin{equation*}\label{eq.86.2}
			\left(\begin{array}{c:c:c}
				\widetilde{\ebra{\textbf{M}_{3,1}}}_{n \times 1} &
				\widetilde{\ebra{\textbf{M}_{3,2}}}_{n \times n} &
				\widetilde{\ebra{\textbf{M}_{3,3}}}_{n \times n}
			\end{array}\right)
		\end{equation*}
		takes the form
		\begin{equation*}\label{eq.86.3}
			\left(\begin{array}{c:ccc:ccccc}
				0 & 0 & \cdots & 0 & \theta & \cdots & 0 & -\theta \\
				\vdots & \vdots & \vdots & \vdots & \vdots & \ddots & \vdots &  \vdots \\
				0 & 0 & \cdots & 0 & 0 & \cdots & \theta & -\theta \\
				-\frac{\sqrt{2}}{4} & -\frac{1}{2} m_{e_1} & \cdots & -\frac{1}{2} m_{e_n} & -\frac{1}{2} m_{2 e_1} + \frac{1}{4} & \cdots & -\frac{1}{2} m_{2 e_{n-1}} + \frac{1}{4} & \theta - \frac{1}{2}m_{2e_n} + \frac{1}{4}
			\end{array}\right)_{n \times \ebra{2n+1}}.
		\end{equation*}
		It can be easily seen that the last $C_n^2$-th elements in the eigenvector are all zeros. Thus we do not care about the last $C_n^2$ columns. Combining the first and $\ebra{n+2}$ -- $\ebra{2n}$-th lines, we obtain that the $\ebra{n+2}$ -- $\ebra{2n+1}$-th elements of the eigenvector are also zeros. Considering the $2$ -- $\ebra{n+1}$-th and $\ebra{2n+1}$-th line in matrix (\ref{eq.86.1}), since the $i$-th eigenvector of covariance matrix corresponding to the eigenvalue $\theta_i$ satisfies
		\begin{equation*}\label{eq.88}
			\ebra{-\theta_i \cdot id_{n \times n} + \text{Cov}\ebra{\rho_0}} \tilde{v}_i = 0,
		\end{equation*}
		and that $-\theta_i$ belongs to the spectrum of $K$, we obtain the corresponding eigenvectors of $K$:
		\begin{equation*}\label{eq.89}
			\tilde{v}_i = \left(-\sqrt{2}m_{e_1}\ebra{v_i}_1 - \cdots - \sqrt{2}m_{e_n} \ebra{v_i}_n,\ \ebra{v_i}_1,\ \dots,\ \ebra{v_i}_n, \ 0,\  \cdots\right), \quad i \in \{1, \dots, n\},
		\end{equation*}
		and the eigenspace is generated by the corresponding eigenspace of the covariance matrix.
		
		The proof is complete.
	\end{proof}
	
	\section{The condition for Ellis-Monroe-Newman's GHS inequality to be strict}\label{appe.GHS}
	In this section, we will recover the proof of Ellis et al. \cite{EMN76} in $\mathbb{R}$ and $N=1$ (the mean-field assumption enables us to just consider one particle) case to find when the inequality becomes strict. The main result is: if $V^{\prime\prime\prime} \not\equiv 0$ on $[0, \infty)$, then it is strict for $m \neq 0$. Let $V$ be an even, coercive potential with $V^{\prime \prime \prime} \geq 0$ on $[0, \infty)$. Our goal is to prove the third-order central moment $\eangle{\ebra{x - \eangle{x}}^3}_\rho\ebra{\lambda, m} \leq 0$ as long as $m \geq 0$, where $\rho\ebra{\lambda, m} = \exp\ebra{\lambda \cdot \ebra{V - mx}}/Z\ebra{\lambda, m}$ with $Z$ being the normalizing constant.
	
	The proof of Ellis et al. uses a four-copy coupling technique to transform analytic complexity into combinatorial complexity. Define two orthogonal transformation matrices
	\begin{equation*}\label{eq.F.1}
		A = \frac{1}{2}
		\begin{pmatrix}
			1 & 1 & 1 & 1\\
			1 & -1 & 1 & -1\\
			1 & 1 & -1 & -1\\
			-1 & 1 & 1 & -1
		\end{pmatrix},
		\qquad B = \frac{1}{2}
		\begin{pmatrix}
			1 & 1 & 1 & 1\\
			-1 & 1 & -1 & 1\\
			-1 & -1 & 1 & 1\\
			1 & -1 & -1 & 1
		\end{pmatrix}.
	\end{equation*}
	$B$ can be obtained by reversing the signs of rows $2$--$4$ of $A$. Let $X^{(1)}, X^{(2)}, X^{(3)}, X^{(4)}$ be four \textit{i.i.d.} copies of $X$, with $\text{law}\ebra{X} \sim \exp(\lambda\cdot V) / Z\ebra{\lambda, 0} = \rho\ebra{\lambda, 0}$.
	The key step is
	\begin{equation}\label{eq.F.3}
		\begin{aligned}
			& \partial_{m, m, m}^3 \ln Z\ebra{\lambda, m} \cdot \ebra{-\lambda^{-3}} = \eangle{\ebra{x - \eangle{x}}^3}_\rho = \eangle{x^3}_\rho - 3\eangle{x^2}_\rho\eangle{x}_\rho + 2\eangle{x}^3_\rho\\
			& = \ebra{\frac{Z\ebra{\lambda, 0}}{Z\ebra{\lambda, m}}}^4 \cdot \int_{\mathbb{R}^4} \ebra{\frac{\partial^3}{\partial m_1^3} - 3 \frac{\partial^3}{\partial m_1^2 \partial m_2} + 2 \frac{\partial^3}{\partial m_1 \partial m_2 \partial m_3}} \exp\ebra{\textbf{m}\cdot \textbf{x}} \cdot \rho\ebra{\lambda, 0}^{\otimes 4}\ebra{\textbf{x}} \dif \textbf{x}|_{\textbf{m} = \ebra{m, \dots, m}} \\
			& = \ebra{\frac{Z\ebra{\lambda, 0}}{Z\ebra{\lambda, m}}}^4\cdot \mathbb{E}\ebra{\ebra{\frac{\partial^3}{\partial m_1^3} - 3 \frac{\partial^3}{\partial m_1^2 \partial m_2} + 2 \frac{\partial^3}{\partial m_1 \partial m_2 \partial m_3}} \exp\ebra{\textbf{m}\cdot\textbf{X}}}|_{\textbf{m} = \ebra{m, \dots, m}}.
		\end{aligned}
	\end{equation}
	Here $\textbf{m} = \ebra{m_1, \dots, m_4}$, $\textbf{X} = \ebra{X^{(1)}, \dots, X^{(4)}}$. Now we will see that the matrix $B$ is indeed diagonalizing the operator appearing above. Let $\textbf{s} = B \textbf{m}$. Then $B \textbf{X}$ writes
	\begin{equation*}\label{eq.F.2}
		\begin{aligned}
			& (B\textbf{X})^{(1)} = \frac{1}{2} \cdot \ebra{X^{(1)} + X^{(2)} + X^{(3)} + X^{(4)}}, \qquad \ebra{B\textbf{X}}^{(2)} = \frac{1}{2} \cdot \ebra{X^{(1)} - X^{(2)} + X^{(3)} - X^{(4)}}, \\
			& \ebra{B\textbf{X}}^{(3)} = \frac{1}{2} \cdot \ebra{X^{(1)} + X^{(2)} - X^{(3)} - X^{(4)}}, \qquad \ebra{B\textbf{X}}^{(4)} = \frac{1}{2}\cdot \ebra{-X^{(1)} + X^{(2)} + X^{(3)} - X^{(4)}}.
		\end{aligned}
	\end{equation*}
	By the orthogonality of $B$
	\begin{equation*}\label{eq.F.4}
		\text{m} \cdot \text{X} = \textbf{m}^T B^T B \textbf{X} = \textbf{s} \cdot B\textbf{X},
	\end{equation*}
	using the fact that for any $i \neq j \neq k$
	\begin{equation*}\label{eq.F.5}
		\begin{aligned}
			& \mathbb{E} \frac{\partial^3}{\partial m_i^3} \exp\ebra{\textbf{m}\cdot \textbf{X}}=\mathbb{E}\frac{\partial^3}{\partial m_1^3}\exp\ebra{\textbf{m}\cdot \textbf{X}}, \\
			& \mathbb{E} \frac{\partial^3}{\partial m_i^2 \partial m_j} \exp\ebra{\textbf{m}\cdot \textbf{X}} = \mathbb{E} \frac{\partial^3}{\partial m_1^2 \partial m_2} \exp\ebra{\textbf{m}\cdot \textbf{X}}, \\
			& \mathbb{E} \frac{\partial^3}{\partial m_i \partial m_j \partial m_k} \exp\ebra{\textbf{m}\cdot \textbf{X}} = \mathbb{E}\frac{\partial^3}{\partial m_1 \partial m_2 \partial m_3} \exp \ebra{\textbf{m}\cdot \textbf{X}},
		\end{aligned}
	\end{equation*}
	one can check that
	\begin{equation*}\label{eq.F.6}
		\begin{aligned}
			\mathbb{E} \ebra{B\textbf{X}}^{(2)} \ebra{B \textbf{X}}^{(3)}& \ebra{B \textbf{X}}^{(4)} \exp\ebra{\textbf{s}\cdot B\textbf{X}}\\
			& = -2 \cdot \mathbb{E} \ebra{\ebra{X^{(1)}}^3 - 3 \ebra{X^{(1)}}^2 X^{(2)} + 2 X^{(1)} X^{(2)} X^{(3)}} \cdot \exp\ebra{\textbf{m} \cdot \textbf{X}}.
		\end{aligned}
	\end{equation*}
	This implies (\ref{eq.F.3}) equals
	\begin{equation*}\label{eq.F.7}
		-2 \cdot \ebra{\frac{Z\ebra{\lambda, 0}}{Z\ebra{\lambda, m}}}^4 \cdot \mathbb{E} \ebra{B\textbf{X}}^{(2)} \ebra{B \textbf{X}}^{(3)}\ebra{B \textbf{X}}^{(4)} \exp\ebra{\textbf{s}\cdot B\textbf{X}}|_{\textbf{s} = \ebra{2m, 0, 0, 0}}.
	\end{equation*}
	By the Taylor expansion, we have
	\begin{equation*}\label{eq.F.8}
		\begin{aligned}
			\exp\ebra{\textbf{s}\cdot B\textbf{X}} & = \sum_{n=0}^{+\infty} \frac{1}{n!}\cdot \ebra{\sum_{i=1}^4 s_i \cdot \ebra{B\textbf{X}}_i}^n\\
			& = \sum_{n=0}^{+\infty} \sum_{|\textbf{m}| = n} \prod_{i=1}^4 \frac{1}{m^{(i)}!}\cdot \ebra{s_i \ebra{B\textbf{X}}_i}^{m^{(i)}},
		\end{aligned}
	\end{equation*}
	where $|\textbf{m}| := m^{(1)} + \cdots + m^{(4)}$. Substituting this into \eqref{eq.F.3} and evaluating at $\ebra{2m, 0, 0, 0}$, we obtain (note that $V$ is at least quadratic by assumption, so the expansion is justified) ($\textbf{m}! := m^{(1)}! \cdots m^{(4)}!$)
	\begin{equation}\label{eq.F.9}
		\begin{aligned}
			& \partial_{m, m, m}^3 \ln Z\ebra{\lambda, m} \cdot \ebra{-\lambda^{-3}} \\
			& = -2\cdot \ebra{\frac{Z\ebra{\lambda, 0}}{Z\ebra{\lambda, m}}}^4 \cdot \sum_{n=0}^{+\infty} \sum_{|\textbf{m}|=n} \frac{1}{\textbf{m}!}\cdot\mathbb{E} \ebra{s_1 \ebra{B\textbf{X}}^{(1)}}^{m^{(1)}} \prod_{i=2}^4 \ebra{s_i\ebra{B\textbf{X}}^{(i)}}^{m^{(i)}}\cdot \ebra{B\textbf{X}}^{(i)}|_{\textbf{s} = \ebra{2m, 0, 0, 0}}\\
			& = -2 \cdot \ebra{\frac{Z\ebra{\lambda, 0}}{Z\ebra{\lambda, m}}}^4 \cdot \sum_{n=0}^{+\infty} \frac{1}{n!}\cdot\mathbb{E} \ebra{2m \ebra{B\textbf{X}}^{(1)}}^{n} \ebra{B\textbf{X}}^{(2)} \ebra{B\textbf{X}}^{(3)} \ebra{B\textbf{X}}^{(4)}\\
			& = -2 \cdot \ebra{\frac{Z\ebra{\lambda, 0}}{Z\ebra{\lambda, m}}}^4 \cdot \sum_{n \text{ odd}} \frac{1}{n!} \cdot \ebra{2m}^n \cdot \mathbb{E} \ebra{\ebra{B\textbf{X}}^{(1)}}^n \ebra{B\textbf{X}}^{(2)}\ebra{B\textbf{X}}^{(3)}\ebra{B\textbf{X}}^{(4)}.
		\end{aligned}
	\end{equation}
	The last equality follows from the symmetry argument
	\begin{equation}\label{eq.F.10}
		\mathbb{E} \prod_{i=1}^4 \ebra{\ebra{B\textbf{X}}^{(i)}}^{m^{(i)}} \neq 0 \text{ if and only if } m^{(i)} \text{ all even or all odd}.
	\end{equation}
	Indeed, since $\text{law } X^{(i)}$ is even and by the definition of $B$, we have the following permutation results
	\begin{equation*}\label{eq.F.11}
		\begin{aligned}
			& \mathbb{E} \prod_{i=1}^4 \ebra{\ebra{B\textbf{Y}}^{(i)}}^{m^{(i)}} = \ebra{-1}^{m^{(1)} + \cdots + m^{(4)}} \cdot \mathbb{E} \prod_{i=1}^4 \ebra{\ebra{B\textbf{X}}^{(i)}}^{m^{(i)}},  &\textbf{Y} = -\textbf{X}, \\
			& \mathbb{E} \prod_{i=1}^4 \ebra{\ebra{B\textbf{Y}}^{(i)}}^{m^{(i)}} = \ebra{-1}^{m^{(2)} + m^{(4)}} \cdot \mathbb{E} \prod_{i=1}^4 \ebra{\ebra{B\textbf{X}}^{(i)}}^{m^{(i)}},  &\textbf{Y} = \ebra{X^{(1)}, X^{(2)}, -X^{(3)}, -X^{(4)}}, \\
			& \mathbb{E} \prod_{i=1}^4 \ebra{\ebra{B\textbf{Y}}^{(i)}}^{m^{(i)}} = \ebra{-1}^{m^{(3)} + m^{(4)}} \cdot \mathbb{E} \prod_{i=1}^4 \ebra{\ebra{B\textbf{X}}^{(i)}}^{m^{(i)}},  &\textbf{Y} = \ebra{X^{(1)}, -X^{(2)}, X^{(3)}, -X^{(4)}}, \\
			& \mathbb{E} \prod_{i=1}^4 \ebra{\ebra{B\textbf{Y}}^{(i)}}^{m^{(i)}} = \ebra{-1}^{m^{(2)} + m^{(3)}} \cdot \mathbb{E} \prod_{i=1}^4 \ebra{\ebra{B\textbf{X}}^{(i)}}^{m^{(i)}},  &\textbf{Y} = \ebra{-X^{(1)}, X^{(2)}, X^{(3)}, -X^{(4)}}.
		\end{aligned}
	\end{equation*}
	Thus \eqref{eq.F.10} follows. It remains to show that the expectation is non-negative when $n$ is odd, or more generally, when all indices in $\textbf{m} = \ebra{m^{(1)}, \cdots, m^{(4)}}$ are odd. By symmetry, we have the following identity, which transforms integration on $\mathbb{R}^4$ into integration on $\mathbb{R}^4_+$:
	\begin{equation*}\label{eq.007}
		\begin{aligned}
			\mathbb{E} \prod_{i=1}^4 \ebra{\ebra{B\textbf{X}}^{(i)}}^{m^{(i)}} & = \int_{\mathbb{R}^4} x_1^{m^{(1)}} x_2^{m^{(2)}} x_3^{m^{(3)}} x_4^{m^{(4)}} \cdot B_\sharp \ebra{\rho\ebra{\lambda, 0}^{\otimes4}}\ebra{\textbf{x}} \dif \textbf{x}, \\
			& = \sum_{\sigma^{(i)} \in \{-1, 1\}} \int_{\mathbb{R}^4_+} \prod_{i=1}^4 \ebra{\sigma^{(i)} x_i}^{m^{(i)}} \cdot B_\sharp\ebra{\rho\ebra{\lambda, 0}^{\otimes 4}}\ebra{\sigma^{(1)}x_1, \cdots, \sigma^{(4)}x_4} \dif \textbf{x}.
		\end{aligned}
	\end{equation*}
	Since $B^{-1} = B^T$, and $\det B = 1$, we have
	\begin{equation*}\label{eq.008}
		\begin{aligned}
			B_\sharp\ebra{\rho\ebra{\lambda, 0}^{\otimes 4}}\ebra{\sigma^{(1)}x_1, \cdots, \sigma^{(4)}x_4} & = \ebra{\text{diag}\{\sigma^{(1)}, \cdots, \sigma^{(4)}\} \circ B}_{\sharp} \ebra{\rho\ebra{\lambda, 0}^{\otimes 4}}\ebra{\textbf{x}}\\
			& = \rho\ebra{\lambda, 0}^{\otimes 4}\ebra{B^T \text{diag}\{\sigma^{(1)}, \cdots, \sigma^{(4)}\} \textbf{x}},
		\end{aligned}
	\end{equation*}
	by the Monge--Amp\`ere equation for the push-forward measures.
	By the evenness of $\rho$, we get
	\begin{equation*}\label{eq.009}
		B_\sharp\ebra{\rho\ebra{\lambda, 0}^{\otimes 4}}\ebra{\sigma^{(1)}x_1, \cdots, \sigma^{(4)}x_4} =
		\begin{cases}
			B_\sharp \ebra{\rho\ebra{\lambda, 0}^{\otimes 4}}\ebra{\textbf{x}}, \qquad \prod_{i=1}^4 \sigma^{(i)} = 1, \\
			A_\sharp \ebra{\rho\ebra{\lambda, 0}^{\otimes 4}}\ebra{\textbf{x}}, \qquad \prod_{i=1}^4 \sigma^{(i)} = -1. \\
		\end{cases}
	\end{equation*}
	Therefore, we obtain
	\begin{equation*}\label{eq.010}
		\mathbb{E} \prod_{i=1}^4 \ebra{\ebra{B\textbf{X}}^{(i)}}^{m^{(i)}} = 8 \int_{\mathbb{R}^4_+} \prod_{i=1}^4 x_i^{m^{(i)}} \cdot \ebra{B_\sharp \ebra{\rho\ebra{\lambda, 0}^{\otimes 4}} - A_\sharp\ebra{\rho\ebra{\lambda, 0}^{\otimes 4}}} \ebra{\textbf{x}} \dif \textbf{x}.
	\end{equation*}
	By the expressions of $A$, $B$, and $\rho\ebra{\lambda, 0}$, it remains to prove the following property of $V$:
	\begin{equation}\label{eq.011}
		\begin{aligned}
			& V\ebra{x_1 - x_2 - x_3 + x_4} + V\ebra{x_1 + x_2 - x_3 - x_4} + V\ebra{x_1 - x_2 + x_3 - x_4} \\
			& + V\ebra{x_1 + x_2 + x_3 + x_4}- V\ebra{x_1 + x_2 + x_3 - x_4} - V\ebra{x_1 - x_2 + x_3 + x_4} \\
			& - V\ebra{x_1 + x_2 - x_3 + x_4} - V\ebra{x_1 - x_2 - x_3 - x_4} \geq 0 \qquad (x_i \geq 0).
		\end{aligned}
	\end{equation}
	This is exactly an eight-point difference scheme for the third-order derivative. Writing
	\begin{equation*}\label{eq.012}
		\begin{aligned}
			\text{(\ref{eq.011})} & = \int_{-1}^{1} \ebra{V^\prime\ebra{x_1 - x_2 - x_3 + r x_4} + V^\prime \ebra{x_1 + x_2 + x_3 + rx_4} \\
				&\qquad\qquad- V^\prime \ebra{x_1 - x_2 + x_3 + r x_4} -V^\prime\ebra{x_1 + x_2 - x_3 + r x_4}}\cdot x_4 \dif r,
		\end{aligned}
	\end{equation*}
	we find that \eqref{eq.011} is equivalent, for $x_i \geq 0$, to:
	\begin{equation}\label{eq.013}
		\begin{aligned}
			V^\prime\ebra{x_1 - x_2 - x_3} + V^\prime \ebra{x_1 + x_2 + x_3} - V^\prime \ebra{x_1 - x_2 + x_3} -V^\prime\ebra{x_1 + x_2 - x_3} \geq 0.
		\end{aligned}
	\end{equation}
	This process transforms a variable into the derivative. Repeating this process, we obtain
	\begin{equation*}\label{eq.014}
		\begin{aligned}
			\text{(\ref{eq.011})} &\Leftrightarrow \text{(\ref{eq.013})}\\
			& \Leftrightarrow V^{\prime \prime} \ebra{x_1 + x_2} - V^{\prime\prime}\ebra{x_1 - x_2} \geq 0, \qquad x_i \geq 0.\\
			& \Leftrightarrow V^{\prime\prime\prime} \ebra{x_1}\geq 0 , \qquad x_1 \geq 0,
		\end{aligned}
	\end{equation*}
	which is precisely the condition required above. Moreover, by the continuity of $V$, if $V^{\prime \prime\prime}\ebra{x_1} > 0$ for some $x_1>0$, the inequality \eqref{eq.011} becomes strict in a neighborhood containing $x_1$ in $\mathbb{R}^4$. This yields
	\begin{equation*}\label{eq.015}
		\mathbb{E} \prod_{i=1}^4 \ebra{\ebra{B\textbf{X}}^{(i)}}^{m^{(i)}} > 0,
	\end{equation*}
	as long as $m^{(i)}$ are all odd. This leads to a condition under which the GHS inequality is strict for any $m > 0$: $V^{\prime\prime\prime} \not\equiv 0$ on $\ebra{0, +\infty}$.
	
	From another aspect, we can observe from \text{(\ref{eq.F.9})} that, when $m^{(i)}$ all odd, the expectation is non-negative and irrelevant to $m$. So we have the dichotomy: either $\eangle{\ebra{x - \eangle{x}}^3}_\rho \ebra{\lambda, \cdot} \equiv 0$ for all $m \geq 0$, or $\eangle{\ebra{x - \eangle{x}}^3}_\rho \ebra{\lambda, \cdot} < 0$ for all $m > 0$. However, $\partial^3_{m, m, m} \ln Z \equiv 0$ implies $V$ is quadratic, or $\rho\ebra{\lambda, 0}$ is Gaussian, which is also the same condition for equality.
	
	\section{A monotonicity in temperature inequality for the first moment}\label{numer}
	In this appendix, we establish a monotonicity inequality for the first moment with respect to temperature, a result that may be of independent interest. As an application, this inequality yields the transversality condition as well as characterizes the global behavior of the double-well phase transition model. Specifically, the result states that for any potential $\bar{V}$ exhibiting superquadratic growth and satisfying
	\begin{equation*}
		\text{ $\bar{V} \in C^3$ even, $\lim\limits_{x \to \pm\infty} \bar{V}\ebra{x} = \infty$, and $\bar{V}^\prime$ convex on $[0, \infty)$},
	\end{equation*}
	with the partition function 
	$$Z\ebra{\sigma, m} = \int_{\mathbb{R}} \exp(-2/\sigma^2\cdot\ebra{\bar{V}\ebra{x} - mx}) \dif x, $$ 
	and the Gibbs measure
	$$\rho\ebra{\sigma, m}(x) = \exp(-2/\sigma^2\cdot \ebra{\bar{V}(x) - mx}) / Z\ebra{\sigma, m}, $$
	the following statements hold:
	\begin{equation}\label{eq.143.numer}\tag{Ineq}
		\begin{aligned}
			& \partial_\sigma (\int_{\mathbb{R}} x \cdot \rho\ebra{\sigma, m} \dif x)\leq 0, \quad \text{ for all } m\geq 0. \\
			& \text{The inequality is strict for any $m > 0$ if and only if } \bar{V}^{\prime\prime\prime} \not\equiv 0 \text{ on } [0, \infty). \text{ Namely, } \\
			& \text{} \bar{V} \text{ is not } \text{quadratic or } \rho\ebra{\lambda, m} = \exp\ebra{\lambda \cdot \ebra{\bar{V}\ebra{x} - mx}} / Z\ebra{\sigma, m} \text{ is  not Gaussian. }
		\end{aligned}
	\end{equation}
	
	To prove this inequality, let $X$ and $Y$ be i.i.d. random variables. By invoking the covariance identity $\mathbb{E}\ebra{f(X) - f(Y)}\ebra{g(X) - g(Y)} = 2 \ebra{\mathbb{E} f(X) g(X) - \mathbb{E}f(X) \mathbb{E}g(X)}$, we obtain
	\begin{equation*}\label{eq.numer.1}
		\begin{aligned}
			\partial_\sigma \left(\int_{\mathbb{R}} x \cdot \rho\ebra{\sigma, m} \dif x\right) & = 4/\sigma^3 \cdot \left(\eangle{x \cdot \Phi_m\ebra{x}}_{\rho} - \eangle{x}_\rho \cdot \eangle{\Phi_m\ebra{x}}_\rho\right) \\
			& = \frac{2}{\sigma^3}\cdot \int_{\mathbb{R}^2} \ebra{x - y}  \ebra{\Phi_m\ebra{x} - \Phi_m\ebra{y}} \cdot \rho^{\otimes2}\ebra{\sigma, m}\ebra{x, y} \dif x \dif y,
		\end{aligned}
	\end{equation*}
	where 
	\begin{equation*}
		\Phi_m\ebra{x} := V\ebra{x} - mx, \qquad \rho^{\otimes2}\ebra{\sigma, m}\ebra{x, y} := \rho\ebra{\sigma, m} \ebra{x} \cdot \rho\ebra{\sigma, m}\ebra{y}. 
	\end{equation*}
	By the identity
	\begin{equation*}
		\begin{aligned}
			& \ebra{x - y}\cdot \ebra{\Phi_m\ebra{x} - \Phi_m\ebra{y}} = \frac{1}{2} \ebra{x - y}^2 \cdot \ebra{V^\prime\ebra{x} + V^\prime\ebra{y} - 2m} + R\ebra{x, y}, \\
			& R\ebra{x, y} := \ebra{x - y}\cdot\ebra{V\ebra{x} - V\ebra{y} - \frac{1}{2}\ebra{x - y}\cdot \ebra{V^\prime\ebra{x} + V^\prime\ebra{y}}}, 
		\end{aligned}
	\end{equation*}
	upon integrating against the corresponding product measure, an application of the integration-by-parts formula reveals that the expectation of the first term vanishes identically
	\begin{equation*}
		\begin{aligned}
			& \int_{\mathbb{R}^2} \ebra{x - y}^2 \cdot \ebra{V^\prime\ebra{x} + V^\prime\ebra{y} - 2m} \cdot \exp\ebra{{-\frac{2}{\sigma^2}}\cdot \ebra{\Phi_m\ebra{x} + \Phi_m\ebra{y}}} \dif x \dif y\\
			& = -\frac{\sigma^2}{2} \cdot \int_{\mathbb{R}^2} \ebra{x - y}^2 \cdot \text{div} \left(\exp\ebra{{-\frac{2}{\sigma^2}}\cdot \ebra{\Phi_m\ebra{x} + \Phi_m\ebra{y}}}, \exp\ebra{{-\frac{2}{\sigma^2}}\cdot \ebra{\Phi_m\ebra{x} + \Phi_m\ebra{y}}}\right) \dif x \dif y\\
			& = \frac{\sigma^2}{2} \cdot \int_{\mathbb{R}^2} \nabla \ebra{\ebra{x - y}^2} \cdot \left(\exp\ebra{{-\frac{2}{\sigma^2}}\cdot \ebra{\Phi_m\ebra{x} + \Phi_m\ebra{y}}}, \exp\ebra{{-\frac{2}{\sigma^2}}\cdot \ebra{\Phi_m\ebra{x} + \Phi_m\ebra{y}}}\right) \dif x \dif y\\
			& = 0. 
		\end{aligned}
	\end{equation*}
	The application of the integration-by-parts formula is justified by the superquadratic growth of $V$. Thus, it suffices to show that for $m \geq 0$,
	\begin{equation}\label{eq.numer.2}
		\partial_\sigma \left(\int_{\mathbb{R}} x \cdot \rho\ebra{\sigma, m} \dif x\right) = \frac{2}{\sigma^3} \cdot \int_{\mathbb{R}^2} R\ebra{x, y} \cdot \rho^{\otimes 2} \ebra{\sigma, m} \ebra{x, y} \leq 0. 
	\end{equation}
	Introducing the change of variables $z = x + y$ and $w = x - y$, \eqref{eq.numer.2} reduces to
	\begin{equation}\label{eq.numer.3}
		\frac{1}{Z^2\cdot \sigma^3} \int_{\mathbb{R}^2} \tilde{R}\ebra{z, w} \cdot \exp(-\frac{2}{\sigma^2}\cdot \ebra{H_0\ebra{z, w}}) \cdot \exp\ebra{\frac{2}{\sigma^2}\cdot mz} \dif z \dif w, 
	\end{equation}
	with 
	\begin{equation*}
		\begin{aligned}
			& \tilde{R} \ebra{z, w} = w \cdot \left(V\ebra{\frac{1}{2}\ebra{z + w}} - V\ebra{\frac{1}{2}\ebra{z - w}} - \frac{1}{2} w \cdot \ebra{V^\prime \ebra{\frac{1}{2}\ebra{z + w}} + V^\prime \ebra{\frac{1}{2}\ebra{z - w}}}\right), \\
			& H_0\ebra{z, w} = V\ebra{\frac{1}{2}\ebra{z + w}} + V\ebra{\frac{1}{2}\ebra{z - w}}. 
		\end{aligned}
	\end{equation*}
	By the evenness of $\bar{V}$, the functions $\tilde{R}$ and $H_0$ satisfy the following symmetry relations
	\begin{equation*}
		\begin{aligned}
			& \tilde{R}\ebra{z, -w} = \tilde{R}\ebra{z, w}, \qquad \ \ \ \tilde{R}\ebra{-z, w} = -\tilde{R}\ebra{z, w}, \\
			& H_0\ebra{-z, w} = H_0\ebra{z, w}, \qquad H_0\ebra{z, -w} = H_0\ebra{z, w}. 
		\end{aligned}
	\end{equation*}
	Exploiting these symmetries allows us to restrict the integration in \eqref{eq.numer.3} to the positive quadrant $\mathbb{R}^2_+$
	\begin{equation*}
		\partial_\sigma \left(\int_{\mathbb{R}} x \cdot \rho\ebra{\sigma, m} \dif x\right) = \frac{4}{Z^2\cdot \sigma^3} \int_{\mathbb{R}^2_+} \tilde{R}\ebra{z, w} \cdot \exp(-\frac{2}{\sigma^2}\cdot \ebra{H_0\ebra{z, w}}) \cdot \sinh\ebra{\frac{2}{\sigma^2}\cdot mz} \dif z \dif w. 
	\end{equation*}
	When $m > 0$ and $z > 0$, we clearly have $\sinh(2mz/\sigma^2) > 0$. Thus, we are left to show that $\tilde{R}(z, w) \leq 0$ for $z, w \geq 0$. This is equivalent to showing that $\tilde{R} \leq 0$ in the region defined by $x + y > 0$ and $x - y > 0$,
	\begin{equation*}
		V\ebra{x} - V\ebra{y} - \frac{1}{2} \ebra{x - y}\cdot \ebra{V^\prime \ebra{x} + V^\prime\ebra{y}} \leq 0. 
	\end{equation*}
	Indeed, when $0 \leq y < x$, the desired inequality follows directly from the convexity of $V^\prime$ on $[0, \infty)$
	\begin{equation}\label{eq.numer.4}
		\begin{aligned}
			V\ebra{x} - V\ebra{y} & = \int_0^1 \frac{\text{d}}{\text{dt}} V\ebra{y + t\ebra{x - y}} \dif t = (x - y) \int_0^1 V^\prime(tx + (1 - t)y) \dif t\\
			& \leq (x - y) \int_0^1 t V^\prime\ebra{x} + (1 - t) V^\prime\ebra{y} dt = \frac{1}{2}\ebra{x - y}\cdot \ebra{V^\prime \ebra{x} + V^\prime \ebra{y}}. 
		\end{aligned}
	\end{equation}
	Alternatively, when $y < 0 < x$, the condition $-y < x$ implies that
	\begin{equation}\label{eq.numer.5}
		\begin{aligned}
			V(x) - V(y) & = \int_0^1 \frac{\text{d}}{\text{dt}} V\ebra{-y + t\ebra{x + y}} \dif t = (x + y) \int_0^1 V^\prime(tx + (1 - t)\cdot \ebra{-y}) \dif t\\
			& \leq (x + y) \int_0^1 tV^\prime\ebra{x} + \ebra{1 - t}V^\prime\ebra{-y} \dif t = \frac{1}{2}(x + y) \cdot \ebra{V^\prime\ebra{x} -  V^\prime\ebra{y}} \\
			& \leq \frac{1}{2}\ebra{x - y}\cdot \ebra{V^\prime\ebra{x} + V^\prime\ebra{y}}. 
		\end{aligned}
	\end{equation}
	Here, the last inequality holds because the property $V^\prime(x)/x \geq V^\prime(-y)/(-y)$ is guaranteed by the oddness and convexity of $V^\prime$ on $[0, \infty)$. Notably, the inequalities \eqref{eq.numer.4} and \eqref{eq.numer.5} become strict provided that $V^\prime$ is strictly convex.
	

	\addtocontents{toc}{\protect\setcounter{tocdepth}{2}}   
	\section*{Acknowledgements}
	This work is supported by the National Key R\&D Program of China (No. 2023YFA1009200), NSFC (Grants 12531009 and 11925102), and Liaoning Revitalization Talents Program
	(Grant XLYC2202042).
	
	\section*{Data availability}
	No data was used for the research described in the article.


\end{document}